\documentclass[12pt]{amsart}
\usepackage{preamble}

\usepackage[
    style=alphabetic,
    doi=false,
    isbn=false,
    url=false,
    eprint=false,
    sorting=nyt
]{biblatex}
\title[Symplectic circle actions on 4-manifolds]{Classification of symplectic non-Hamiltonian circle actions on 4-manifolds}
\author{Rei Henigman}
\address{School of Mathematical Sciences, Tel Aviv University, Tel Aviv
69978, Israel}
\email{rei.henigman@gmail.com}

\begin{document}

\begin{abstract}
    We classify symplectic non-Hamiltonian circle actions on compact connected symplectic 4-manifolds, up to equivariant symplectomorphisms. Namely, we define a set of invariants, show that the set is complete, and determine which values are attainable by constructing a space for each valid choice. We work under the assumption that the group of periods of the one-form $\iota_X \omega$ is discrete, which allows us to define a circle-valued Hamiltonian for the action, and apply tools from Karshon-Tolman's work on the classification of complexity one spaces. This assumption is always satisfied if the symplectic form is rational, or if the quotient space has first Betti number one.

\end{abstract}
\maketitle

\tableofcontents


\section{Introduction}
Let a torus $T^m \cong {(S^1)}^m$ act effectively on a compact connected symplectic manifold $(M^{2n}, \omega)$. The action is called \textbf{symplectic} if $T^m$ acts by symplectomorphisms. Let $\mathfrak{t}$ be the Lie algebra of $T^m$, $\mathfrak{t}^*$ its dual, and denote by $\langle \cdot , \cdot \rangle:\mathfrak{t}^* \times \mathfrak{t} \rightarrow \mathbb{R}$ the natural pairing. A symplectic action is called \textbf{Hamiltonian} if it admits a \textbf{momentum map}, that is, a $T^m$-invariant smooth map $\Phi: M\rightarrow \mathfrak{t}^*$ that satisfies
\begin{equation}\label{momentum_map_equation}
    - d\langle \Phi, \xi \rangle = \iota_{\xi_M} \omega,
\end{equation}
for every $\xi \in \mathfrak{t}$, where $\xi_M \in \mathfrak{X}(M)$ is the vector field induced by $\xi$. The triple $(M, \omega, \Phi)$ is called a \textbf{Hamiltonian $T^m$-manifold}.

In this paper, we focus on actions of the circle $S^1 \cong \mathbb{R}/(2\pi\mathbb{Z})$. In this case, the momentum map is a time-independent Hamiltonian function $\Phi:M\rightarrow \mathbb{R}$ whose Hamiltonian vector field $X$, defined by
\begin{equation}\label{hamilton_equation}
    - d \Phi = \iota_{X} \omega,
\end{equation}
generates a $2\pi$-periodic flow.

Originating in the formalism of classical mechanics, Hamiltonian actions have been extensively studied in mathematics. In 1982, Atiyah~\cite{atiyah_convexity} and Guillemin--Sternberg~\cite{gs_convexity} independently proved the celebrated convexity theorem stating that the image of the momentum map is a convex polytope in $\mathfrak{t}^*$, and that the level sets of the momentum map are connected. Shortly thereafter, Delzant~\cite{delzant} proved a famous classification theorem, now known as Delzant's classification --- he showed that $2n$-dimensional compact connected Hamiltonian-$T^n$-manifolds are classified up to equivariant symplectomorphisms by their momentum polytopes.

Delzant's classification inspired additional equivariant classifications.
Karshon and Lerman~\cite{lerman_karshon} generalized Delzant's classification to non-compact $2n$-dimensional Hamiltonian $T^n$-manifolds.
In dimension four, Karshon~\cite{periodic_hamiltonians} classified compact connected Hamiltonian $S^1$-manifolds in terms of decorated graphs, following earlier works by Audin~\cite{audin_paper} and Ahara-Hattori~\cite{ahara_hattori}. See also Audin's book~\cite{audin_book}. In a series of papers, Karshon and Tolman~\cite{centered_hamiltonians}~\cite{tall_uniqueness}~\cite{tall_existence} generalized Karshon's result to $2n$-dimensional Hamiltonian $T^{n-1}$-manifolds, whose momentum maps are proper and have reduced spaces of dimension two. These spaces are also known as tall complexity one spaces. For semitoric integrable systems, see~\cite{vu_ngoc_convexity, semitoric_uniqueness, semitoric_existence}. For non-abelian compact group actions, see, e.g.,~\cite{patrick_classifiation, woodward_classification, chiang_classification, knop_classification, knop_conjecture}. For Hamiltonian circle actions on orbifolds, see~\cite{orbifold_ham_s1}, and for Hamiltonian torus actions on orbifolds, see~\cite{lerman_tolman}.

Non-Hamiltonian actions are also of great interest in both physics and mathematics, see, e.g., Novikov's classical paper~\cite{novikov_paper}. Notably, McDuff~\cite{symplectic_circle_actions} proved that symplectic non-Hamiltonian circle actions on four-dimensional manifolds do not have fixed points. A key observation in her paper is that after a small perturbation of the symplectic form, every symplectic non-Hamiltonian circle action is generated by a circle-valued Hamiltonian. In the same paper, McDuff also constructed a non-Hamiltonian circle action on a six dimensional manifold, which has fixed tori. The question of whether there exists a compact symplectic non-Hamiltonian circle action with fixed points that are all isolated, often called \textit{McDuff's Conjecture}, was answered in the affirmative by Tolman~\cite{tolman_isolated_points}. Her proof also utilizes circle-valued Hamiltonians.
In certain cases, equivariant classification results were established for non-Hamiltonian symplectic torus actions. In dimension four, Pelayo~\cite{pelayo_two_tori} extended Delzant's classification to non-Hamiltonian actions --- he classified non-Hamiltonian symplectic $T^2$ actions on compact connected four-dimensional symplectic manifolds. This was partially based on a previous work of Pelayo and Duistermaat~\cite{duistermaat_pelayo} in which they classified symplectic torus actions with coisotropic principal orbits (see also~\cite{benoist}). For more works on symplectic torus actions on symplectic manifolds, see, e.g.,~\cite{kotschick_contractible_orbits, kim_frankel_complexity_one, jang_tolman, pelayo_tolman_fixed_points, ginzburg_symplectic_actions, giacobbe}.
\\

In this paper, we study symplectic non-Hamiltonian circle actions on four-dimensional compact connected symplectic manifolds. Our motivating problem is the following:
\begin{problem}\label{the_problem}
    Classify symplectic non-Hamiltonian circle actions on four-dimensional compact connected symplectic manifolds, up to equivariant symplectomorphisms. That is, find a complete set of invariants for these actions, and determine which values of the invariants are realizable.
\end{problem}
This classification problem extends Karshon's classification to non-Hamiltonian circle actions, and it extends Pelayo's classification from $T^2$ actions to circle actions. The main result of this paper is a solution of Problem~\ref{the_problem}, under an integrability assumption.
\\

Our strategy is to choose a circle-valued Hamiltonian in the sense of McDuff (see~\cite{symplectic_circle_actions}), and apply the tools developed by Karshon and Tolman for the classification of tall complexity one spaces (see~\cite{centered_hamiltonians},~\cite{tall_uniqueness},~\cite{tall_existence}). To choose a circle-valued Hamiltonian, we assume that the group of period $P$ of the one-form $\iota_X \omega$, defined by
\begin{equation*}
    P := \{\int_A \iota_X \omega\ |\ A \in H_1(M, \mathbb{Z})\},
\end{equation*}
is a discrete subgroup of $\mathbb{R}$. We then define a circle-valued Hamiltonian $\Phi:M \rightarrow \mathbb{R}/P$, where $\mathbb{R}/P\cong S^1$ since $P$ is discrete. The assumption that $P$ is discrete is always satisfied if the class $[\omega]$ is rational, or if the first Betti number of the quotient space $M/S^1$ is $1$. 

Topologically, we exhibit the quotient space $M/S^1$ as a surface bundle over the circle $\mathbb{R}/P$, with special marked points in each fiber. The fibers are the reduced spaces, and the marked points correspond to non-free orbits, and form ``braids'' in $M/S^1$, that circle around the base of the surface bundle. With the help of this bundle, we define seven invariants for our spaces:
\begin{enumerate}
    \item The group of periods $P$ of $\iota_X \omega$.
    \item The genus of the reduced spaces.
    \item The area of the reduced spaces.
    \item The isotropy data of the level sets.
    \item A monodromy invariant which encodes both the topological type of the surface bundle $M/S^1 \rightarrow \mathbb{R}/P$, and the topological type of the braiding of the marked points.
    \item A cohomological invariant which encodes the topological type of the $S^1$-fibration $M \rightarrow M/S^1$.
    \item A cohomological invariant which determines the cohomology class of the symplectic form.
\end{enumerate}

We determine which values of these invariants can be attained by constructing a space for each valid combination of invariants (see Theorem~\ref{existence_theorem}). Moreover, we show that two spaces are equivariantly symplectomorphic if and only if their invariants agree (see Theorem~\ref{uniqueness_theorem}). Together, these two theorems give a complete classification up to equivariant symplectomorphisms, thus solving Problem~\ref{the_problem} under the assumption that the group of period $P$ of the one-form $\iota_X \omega$ is discrete.
\\

The first invariant captures the ``size'' of the image of the circle-valued Hamiltonian. The second, third, and fourth invariants determine the isomorphism type of a neighborhood of a level set of the circle-valued Hamiltonian (cf.~\cite[Theorem 1]{centered_hamiltonians}. See also Lemma~\ref{local_uniqueness_lemma}). The last two invariants encode non-trivial global information on the manifold, which was not apparent in complexity one spaces (there, the fibration $M \rightarrow M/S^1$, and the cohomology class of the symplectic form were determined by the Duistermaat-Heckman function alone). When the genus of the reduced spaces is zero, the last two invariants turn out to be trivial.

Lastly, the monodromy invariant is particularly interesting, as it captures non-trivial topological data related to the subset of non-free orbits. As described above, the non-free orbits come in families that form braids in the surface bundle $M/S^1$. This topological phenomenon is analogous to the \textit{painting invariant}, a topological invariant introduced by Karshon and Tolman in~\cite{tall_uniqueness} for tall complexity one spaces. It is known that the painting invariant is trivial in dimension four. Hence, it is surprising that for non-Hamiltonian circle actions, this kind of topological information appears even in lower-dimensional spaces.
\\

The idea of constructing a principal $S^1$-bundle over a mapping torus, where the total space of the fibration is symplectic, was used by~\cite{gray_morgan_marisa} to construct examples of non-Kähler symplectic manifolds. Ginzburg~\cite[Section 2]{ginzburg_symplectic_actions} noted that this approach might be useful for formalizing a classification theorem for free symplectic $S^1$ actions. However, he points outs that describing the gluing data is a very complex problem in general, as it involves data on the group of symplectomorphisms of the reduced spaces, which is known to be very hard in dimension four or higher.

In this work, we only focus on (non-Hamiltonian) complexity one actions, where the reduced spaces are two dimensional. This allows us to explicitly characterize the gluing data using the monodromy invariant. Furthermore, by omitting the assumption that the circle action is free, we generalize the ideas of~\cite{ginzburg_symplectic_actions},~\cite{gray_morgan_marisa}. This generalization introduces substantial difficulties to the proofs, and forces us to adapt tools from~\cite{centered_hamiltonians},~\cite{tall_uniqueness}, and~\cite{haefliger_salem}, but it also reveals the braiding phenomenon that is captured by the monodromy invariant.
\\

It seems plausible that some of the results of this paper can be generalized to higher dimensional non-Hamiltonian complexity one actions, under the assumption that the orbits of the action are isotropic (see Remark~\ref{high_complexity_remark}). On the contrary, as with the case of Hamiltonian actions, generalizing these results to actions of complexity larger than one seems significantly harder.

Finally, we hope that this work will be useful for addressing the following two problems:
\begin{itemize}
    \item Which four-dimensional symplectic $S^1$-manifolds admit an equivariant Kähler structure or more generally an equivariant complex structure?
    \item Which symplectic $S^1$ actions on four-manifolds extend to symplectic $T^2$ actions?
\end{itemize}
These problems were solved in~\cite{periodic_hamiltonians} for Hamiltonian $S^1$ actions in dimension four, based on her classification theorem. For symplectic $T^2$ actions in dimension four, Pelayo and Duistermaat~\cite{complex_pelayo_duistermaat} solved the first question, based on Pelayo's classification theorem~\cite{pelayo_two_tori}. For higher dimensional complexity one spaces, there are works in progress addressing these questions (\cite{kahler_complexity_one_spaces} for equivariant Kähler structures on complexity one spaces, and~\cite{extending_tall_complexity_one},~\cite{extending_short_complexity_one} for extending Hamiltonian $T^{n-1}$ actions to Hamiltonian $T^n$ actions). Surprisingly, these works find connections between the above problems and the painting invariant. It is therefore natural to ask if for non-Hamiltonian $S^1$ actions, these problems are related to the monodromy invariant. See Remarks~\ref{kahler_structure_remark} and~\ref{extending_to_pelayo_remark} for more information, and the paragraph thereafter for a related observation.

\subsubsection*{Acknowledgements}
First, I wish to thank my advisor, Yael Karshon, for illuminating discussions and indispensable help with this project. Additionally, I would like to thank Liat Kessler, Leonid Polterovich, Daniele Sepe, Susan Tolman, and Yoav Zimhony for useful discussions related to this paper. Special thanks goes to Joé Brendel for his excellent suggestions on how to improve this manuscript. We also thank the anonymous referee for their helpful comments, which improved the clarity of this work in general, and in particular for their invaluable suggestions regarding the formulation of the fibration and de Rham invariants.

This paper is partially based on results from my master's thesis, Invariants of symplectic non-Hamiltonian circle actions on 4-manifolds, completed at Tel Aviv University under the supervision of Professor Yael Karshon.

Some of the discussions that contributed to this work took place in the workshops ``Workshop on Lie Groups, Singular Spaces, and Higher Structures'' and ``Workshop on Hamiltonian Geometry and Quantization'', both taken place at the Fields Institute for Research in Mathematical Sciences. We thank the Fields Institute and the organizers of these great workshops.

This research was partially supported by an NSF-BSF Grant 2021730.


\section{Definitions and main results}
In this section, we introduce some preliminary definitions and lemmas, and then state the main results of the paper (Theorem~\ref{existence_theorem} and Theorem~\ref{uniqueness_theorem}). 

Let $X \in \mathfrak{X}(M)$ be a vector field generating a circle action on a four-dimensional compact connected symplectic manifold $(M, \omega)$. The action is symplectic if and only if the one-form $\iota_X \omega$ is closed, and it is Hamiltonian if and only if the one-form $\iota_X \omega$ is exact. We will assume that the action is symplectic but non-Hamiltonian, or equivalently, that the one-form $\iota_X \omega$ is closed but not exact.

The \textbf{group of periods} $P$ of $\iota_X \omega$ is the subgroup of $\mathbb{R}$ defined by 
\begin{equation*}
    P := \{\int_A \iota_X \omega\ |\ A \in H_1(M, \mathbb{Z})\}.
\end{equation*}
It is invariant under equivariant symplectomorphisms. The circle action is Hamiltonian if and only if the group of periods is trivial, hence in our case $P$ is a non-trivial subgroup. If $P$ is a discrete subgroup of $\mathbb{R}$, then $\mathbb{R}/P$ is diffeomorphic to a circle, and the circle action is generated by a circle-valued Hamiltonian $\Phi:M \rightarrow  \mathbb{R}/P$, which is unique up to an additive constant.

In the following paragraphs, we define invariants for spaces that admit a circle-valued Hamiltonian, and so we will assume that our spaces have discrete group of periods, so that every space will admit a circle-valued Hamiltonian $\Phi:M \rightarrow \mathbb{R}/P$. By Lemma~\ref{betti_one_group_of_periods_discrete}, this includes every space which satisfies $b_1(M/S^1) = 1$ (see also Equation~\eqref{equation_cohomology} for a description of $H^1(M/S^1)$ in terms of the monodromy invariant). Moreover, it includes every space with rational symplectic form $[\omega] \in H^2(M, \mathbb{Q})$.

\begin{example}\label{intro_example}
    Let $S^1$ act on $(\mathbb{T}^2 \times S^2, \omega = 2\pi dp \wedge dq + dh \wedge d\theta)$ by rotating the torus $k$ times, and the sphere $\ell$ times. Assume that $k > 0$, $\ell > 0$, and $\gcd(k, \ell) = 1$. This action is effective, and is generated by the circle-valued Hamiltonian
    \begin{equation*}
        \Phi(p, q, h, \theta) = kp + \ell h,
    \end{equation*}
    which is defined modulo $k$, since $P(\iota_X \omega) = k \mathbb{Z}$.

    When $k>1$, this action is not free - the subtori $\mathbb{T}^2 \times \{n\}$ and $\mathbb{T}^2 \times \{s\}$ are stabilized by the subgroup $\mathbb{Z}_k$ of $S^1$, where $n$ and $s$ are the north pole and south pole of the sphere, respectively. The subtori can be thought of as $\mathbb{Z}_k$-tori, analogously to the well-known $\mathbb{Z}_k$-spheres from~\cite{periodic_hamiltonians}.

    We recommend keeping this example in mind when reading the following paragraphs on the definitions of the invariants, and we will refer to it after each definition. See Subsection~\ref{sphere_torus_subsection} for a more in-depth description of this example, and in general Appendix~\ref{examples_appendix} for more examples.
\end{example}

By~\cite[Proposition 2]{symplectic_circle_actions}, the circle action has no fixed points (see Lemma~\ref{no_fixed_points}). It follows that every non-free orbit $\mathcal{O}$ of the action has stabilizer $\mathbb{Z}_n$ for some $n \ge 2$. Its \textbf{isotropy representation} is defined as the linear symplectic isomorphism class of the representation of $\mathbb{Z}_n$ on the tangent spaces at the points in $\mathcal{O}$. It can be characterized by a pair $(n,a)$ of coprime integers that satisfy $0 < a < n$, where $n$ is the order of the stabilizer, and $a$ characterizes its action on the symplectic slice (see Section~\ref{preliminaries_section}). We refer to such a pair as a \textbf{coprime residue class}, and often denote it by $C$.

For every $y\in \mathbb{R}/P$, the level set $\Phi^{-1}(y)$ of $\Phi:M \rightarrow \mathbb{R}/P$ is compact, and has finitely many non-free orbits. The unordered list of isotropy representations of the non-free orbits in the level set is called the \textbf{isotropy data} of the level set. 

In Section~\ref{section_circle_valued}, we prove the following fiber connectivity lemma:
\begin{lemma}\label{fiber_connectivity_of_circle_actions}
    Let $(M, \omega)$ be a compact connected symplectic manifold, equipped with a symplectic non-Hamiltonian $S^1$ action with a discrete group of periods $P$, and let $\Phi:M \rightarrow \mathbb{R}/P$ be a circle-valued Hamiltonian for the action. Then the level sets of $\Phi$ are connected.
\end{lemma}
We emphasize that Lemma~\ref{fiber_connectivity_of_circle_actions} serves as our main motivation for using $\mathbb{R}/P$-valued Hamiltonians throughout the paper, instead of the usual $\mathbb{R}/\mathbb{Z}$-valued Hamiltonians used by most other authors. Furthermore, we note that Lemma~\ref{fiber_connectivity_of_circle_actions} does not make any assumption on the dimension of the manifold. This lemma generalizes the connectivity part of Atiyah-Guillemin-Sternberg's convexity theorem, to circle-valued Hamiltonians. See also~\cite{benoist} and~\cite{giacobbe} for similar generalizations.

It follows by Lemma~\ref{fiber_connectivity_of_circle_actions} that the reduced space at $y$ is homeomorphic to a compact connected oriented surface of some genus $g$. In fact, $\Phi^{-1}(y)~\rightarrow~\Phi^{-1}(y)/S^1$ is a Seifert fibration with Seifert invariant $\{g;(1,b),(n_1, a_1),\dots,(n_k, a_k)\}$, where each isotropy representation is given by the coprime residue class $(n_j, a_j)$, and $b$ is some integer. The Seifert Euler number of the fibration is given by
\begin{equation}\label{intro_seifert_euler_eq}
    e := -b -\sum_{j=1}^k \frac{a_j}{n_j}.
\end{equation}
See also~\eqref{seifert_euler_equation}. By Lemma~\ref{duistermaat_heckman_constant} below, $e$ is always integer, and thus the sum $\sum_{j=1}^k -\frac{a_j}{n_j}$ is always an integer. See~\cite[Corollary 1.6]{seifert_fibrations} for our notation of Seifert invariants.

The fixed points of the action coincide with the critical points of the Hamiltonian. Because the circle action has no fixed points, the circle-valued Hamiltonian $\Phi:M \rightarrow \mathbb{R}/P$ is a proper submersion. It follows that all the reduced spaces have the same genus $g$, and all the level sets of $\Phi$ have the same isotropy data, which can be characterized by a list of coprime residue classes $C_1,\dots,C_k$  (see Lemma~\ref{same_genus_and_isotropy_data}). Both $g$ and $C_1,\dots,C_k$ are invariant under equivariant symplectomorphisms, and we call them the \textbf{genus invariant} and the \textbf{isotropy data invariant}.

For instance, Example~\ref{intro_example} has reduced spaces diffeomorphic to spheres, with two singularities at the north and south poles (if $k>1$). Its genus invariant is $0$, and its isotropy data invariant is the two coprime residue classes $(k, \ell \mod k), (k, -\ell \mod k)$ when $k > 1$, or an empty list when $k=1$.

The \textbf{Duistermaat-Heckman measure} on $\image(\Phi)$ is the pushforward of the \textbf{Liouville measure} $\frac{\omega^2}{2!}$ under the Hamiltonian $\Phi$. By~\cite{duistermaat_heckman}, it is equal to $\rho_{\DuHe} \lambda$ where $\rho_{\DuHe}$ is a piece-wise polynomial, and $\lambda$ is the Lebesgue measure on $\image(\Phi)$ (see also Theorem~\ref{duistermaat_heckman_thm}). 

In Section~\ref{section_circle_valued}, we prove the following lemma on the Duistermaat-Heckman measure of non-Hamiltonian circle actions in dimension four:
\begin{lemma}\label{duistermaat_heckman_constant}
    Let $(M^4, \omega)$ be a compact connected symplectic manifold, equipped with a symplectic non-Hamiltonian $S^1$ action with a discrete group of periods $P$, and let $\Phi:M \rightarrow \mathbb{R}/P$ be a circle-valued Hamiltonian for the action. Then the Duistermaat-Heckman measure on $\mathbb{R}/P$ is translation invariant, and the Seifert Euler number of each level set of $\Phi$ is 0.
\end{lemma}
It follows from Lemma~\ref{duistermaat_heckman_constant} that the Duistermaat-Heckman measure is equal to $c_{\DuHe}\lambda$ for some constant $c_{\DuHe}$. We call $c_{\DuHe}$ the \textbf{Duistermaat-Heckman constant} of the space. It is invariant under equivariant symplectomorphisms. In fact, we have the useful equality
\begin{equation}\label{dh_vol_eq}
    \vol(M, \omega) = 2\pi \tau c_{\DuHe}.
\end{equation}
For instance, the Duistermaat-Heckman constant of Example~\ref{intro_example} is $4\pi / k$.

To define the monodromy invariant, we first define the mapping class group of a surface with labeled points. Fix an oriented surface of genus $g$, and label $k$ distinct points with coprime residue classes $C_1,\dots,C_k$. We denote the labeled surface by $\Sigma_{g,C_1,\dots,C_k}$. Let $\Homeo_+(\Sigma_{g,C_1,\dots,C_k})$ be the group of orientation-preserving homeomorphisms of $\Sigma_{g,C_1,\dots,C_k}$ that send each labeled point to a point with the same label. We define the \textbf{mapping class group} of $\Sigma_{g,C_1,\dots,C_k}$ by
\begin{equation*}
    \mcg(\Sigma_{g,C_1,\dots,C_k}) := \Homeo_+(\Sigma_{g,C_1,\dots,C_k})/{\sim},
\end{equation*}
where $f \sim h$ if they are continuously isotopic relative to the labeled points.

We give a topological description of the quotient $M/S^1$ as a mapping torus whose fiber is an oriented surface with special marked points:
\begin{proposition}\label{topological_model_proposition}
     Let $(M^4, \omega)$ be a compact connected symplectic manifold, equipped with a symplectic non-Hamiltonian $S^1$ action with a discrete group of periods $P$, and let $\tau$ be the positive generator of $P$. Moreover, let $\Phi:M \rightarrow \mathbb{R}/P$ be a circle-valued Hamiltonian for the action, and let $\overline{\Phi}: M/S^1 \rightarrow \mathbb{R}/P$ be the induced map on the quotient space.
     
     Then there exists a unique integer $g \ge 0$, a unique finite unordered list of coprime residue classes $C_1,\dots,C_k$, and a homeomorphism $f$ in $\Homeo_+(\Sigma_{g,C_1,\dots,C_k})$, such that $M/S^1$ is homeomorphic to the mapping torus
     \begin{equation*}
        M_f:=\bigslant{\Sigma_{g, C_1,\dots,C_k} \times [0, \tau]}{(x,0) \sim (f(x), \tau)},
    \end{equation*}
    by an orientation-preserving homeomorphism $\Psi:M/S^1 \rightarrow M_f$ that satisfies:
    \begin{enumerate}
        \item The following diagram commutes:
            \begin{center}
                \begin{tikzcd}
                    M/S^1 \arrow[rr, "\Psi"] \arrow[dr, "\overline{\Phi}"'] & & M_f \arrow[dl, "\pi"] \\
                    & \mathbb{R}/P      
                \end{tikzcd}
            \end{center}
        \item Let $A \subset M/S^1$ be the set of non-free orbits. Then the image of $\Psi|_A$ is the set of points in $M_f$ that project to labeled points in $\Sigma_{g,C_1,\dots,C_k}$. Furthermore, the isotropy representation of each point in $A$ corresponds to the label of its image.
    \end{enumerate}    
    
    Finally, the conjugacy class of $[f]$ in $\mcg(\Sigma_{g,C_1,\dots,C_k})$ is independent of the choice of~$f$.
\end{proposition}
We prove Proposition~\ref{topological_model_proposition} in Section~\ref{painted_surface_bundle_section}. The \textbf{monodromy invariant} is the conjugacy class of the element $[f]$ in $\mcg(\Sigma_{g,C_1,\dots,C_k})$, given by Proposition~\ref{topological_model_proposition}. For example, the monodromy invariant of Example~\ref{intro_example} is the conjugacy class of the identity map $\id:S^2 \rightarrow S^2$. See Definition~\ref{monodromy_invariant_definition} for another definition of the monodromy invariant, and Remark~\ref{monodromy_definitions_agree} on why these definitions agree.

We note that the first cohomology of the quotient space $M/S^1$ can be related to the monodromy invariant by:
\begin{equation}\label{equation_cohomology}
    H^1(M/S^1) \cong \mathbb{Z}^{\dim(\ker(f_1 - \id)) + 1},
\end{equation}
where $f_1:H^1(\Sigma_g) \rightarrow H^1(\Sigma_g)$ is the map induced by $f$ on the first cohomology (see also Equation~\eqref{long_exact_sequence_equation}). In particular, when $f_1$ has no eigenvalues $1$, then $b_1(M/S^1) = 1$, and the group of periods must be discrete by Lemma~\ref{betti_one_group_of_periods_discrete}. This includes for example the case where the reduced spaces are homeomorphic to spheres.
\begin{remark}
    For tall complexity one spaces, the \textit{painting invariant} defined in~\cite{tall_uniqueness} encodes the topological properties of the quotient space, and in dimension $4$ it is known to be trivial (see~\cite[Remark 2.3]{tall_uniqueness}). Surprisingly, the monodromy invariant which serves as the analogous topological invariant in our setting, turns out to be non-trivial, which means that non-trivial topological phenomena appear already in dimension $4$. This occurs since the image of the momentum map is a circle, which allows for the $\mathbb{Z}_k$-tori to form ``braids'' when circling around the momentum values. Here, $\mathbb{Z}_k$-tori refer to the connected components of the set of non-free orbits (cf. $\mathbb{Z}_k$-spheres in~\cite{periodic_hamiltonians}).
\end{remark}

In Section~\ref{painted_surface_bundle_section}, we show that the group of periods, genus invariant, isotropy data invariant, and monodromy invariant classify the spaces up to $\Phi$-diffeomorphisms between their quotients (see Definition~\ref{phi_diffeo} and Proposition~\ref{uniqueness_theorem_phi_diffeo} for a precise definition and statement). In Section~\ref{fibration_invariant_section}, we formalize the \textbf{fibration invariant}, which determines whether two spaces with $\Phi$-diffeomorphic quotients are also equivariantly diffeomorphic (or $\Phi$-$T$-diffeomorphic to be more precise, see Definition~\ref{phi_t_diffeo}). Intuitively, it encodes the topological type of the fibration $M \rightarrow M/S^1$. The construction is based on a characteristic class introduced by Fintushel in~\cite{fintushel2}, for the classification of smooth circle actions in dimension four. This class generalizes the Chern class of principal $S^1$-bundles to the context four-manifolds with locally-free actions. We note that for complexity one spaces, the fibration $M \rightarrow M/S^1$ is determined by the Duistermaat-Heckman function, and thus this invariant is special for the non-Hamiltonian case.

We give a short description of the fibration invariant. By Proposition~\ref{topological_model_proposition}, the quotient space $M/S^1$ is diffeomorphic to a mapping torus $M_f$, where $f:\Sigma_g \rightarrow \Sigma_g$ is a representative of the monodromy invariant. Therefore, identifying the quotient space with $M_f$, the characteristic class $c$ of the principal $S^1$-orbibundle $M \rightarrow M_f$, introduced by Fintushel in~\cite{fintushel2}, is an element of $H^2(M_f, \mathbb{Z})$, which is equal to the Chern class of a principle $S^1$-bundle $\tilde M \rightarrow M_f$ constructed by extracting a neighborhood of the non-free orbits, and replacing them with free orbits. By Lemma~\ref{duistermaat_heckman_constant}, the Seifert Euler class of each level set vanishes. It follows that the restriction of $c$ on each fiber of $\pi:M_f \rightarrow \mathbb{R}/P$ is equal to $b [\Sigma_g]$, where $[\Sigma_g]$ is the generator of $H^2(\Sigma_g, \mathbb{Z})$, and $b$ is the sum $-\sum_{j=1}^k \frac{a_j}{n_j}$ by Equation~\eqref{intro_seifert_euler_eq}. Using the short exact sequence
\begin{equation}\label{short_exact_sequence_introduction}
    0 \rightarrow H^1(\Sigma_g, \mathbb{Z})/\image(f_1 - \id) \rightarrow H^2(M_f, \mathbb{Z}) \overset{p}{\rightarrow} H^2(\Sigma_g, \mathbb{Z}) \rightarrow 0
\end{equation}
given by the Mayer--Vietoris sequence, the class $c$ can be regarded as an element of the affine subspace
\begin{equation}\label{affine_subspace_intro}
    \mathcal{H} := p^{-1}(b [\Sigma_g]) \subset H^2(M_f, \mathbb{Z}),
\end{equation}
which is modeled on $H^1(\Sigma_g, \mathbb{Z})/\image(f_1 - \id)$. Here, $p:H^2(M_f, \mathbb{Z}) \rightarrow H^2(\Sigma_g, \mathbb{Z})$ is the restriction map to a fiber, and $f_1:H^1(\Sigma_g, \mathbb{Z}) \rightarrow H^1(\Sigma_g, \mathbb{Z})$ is the map induced on first cohomology by $f$. The \textbf{fibration invariant} of the space $(M, \omega, \Phi)$ is represented by the element $c$ of the affine subspace $\mathcal{H}$, and two elements $c$ and $c'$ are said to be \textbf{equivalent} if there exists an orientation-preserving homeomorphism $F:M_f \rightarrow M_f$, respecting the map $M_f \rightarrow \mathbb{R}/P$, such that $F^*c' = c$. See Section~\ref{fibration_invariant_section} for more details.
\\

Next, we wish to determine whether two equivariantly diffeomorphic spaces (or more precisely, $\Phi$-$T$-diffeomorphic) are equivariantly symplectomorphic. In Section~\ref{de_rham_invariant_section}, using equivariant Moser's trick, we formalize the \textbf{de Rham invariant}, which together with the Duistermaat-Heckman constant, determines exactly that. Intuitively, the de Rham invariant describes the possible non-equivalent choices for the cohomology class of the symplectic form. As before, the cohomology class of the symplectic form of complexity one spaces is determined by the Duistermaat-Heckman function, and thus the de Rham invariant is special for the non-Hamiltonian case.

We give a short description of the de Rham invariant. Let $(M, \omega)$ and $(M', \omega'$) be $S^1$-spaces as above, that have the same six invariants. Then the spaces are equivariantly diffeomorphic with a map respecting the circle-valued Hamiltonians. For every such equivariant diffeomorphism $\Psi:M \rightarrow M'$, the form $\Psi^*\omega' - \omega$ is a basic form, which determines a cohomology class ${[\Psi^*\omega' - \omega]}_{\reduced}$ in $H^2(M/S^1, \mathbb{R})$. If there exists such an equivariant diffeomorphism $\Psi$ for which ${[\Psi^*\omega' - \omega]}_{\reduced}$ vanishes, we say that the forms $\omega$ and $\omega'$ are \textbf{equivalent}, and we can apply equivariant Moser's trick to see that $(M, \omega)$ and $(M', \omega')$ are equivariantly symplectomorphic. The \textbf{de Rham invariant} of a space is the equivalence class of its symplectic form, under these cohomological equivalences. Since both forms have the same Duistermaat-Heckman constant, the cohomology class ${[\Psi^*\omega' - \omega]}_{\reduced}$ vanishes on each fiber of $M/S^1 \rightarrow \mathbb{R}/P$. Therefore using the short exact sequence
\begin{equation}\label{real_short_exact_sequence_introduction}
    0 \rightarrow H^1(\Sigma_g, \mathbb{R})/\image(f_1 - \id) \overset{\overline{d^*}}{\rightarrow} H^2(M/S^1, \mathbb{R}) \overset{p}{\rightarrow} H^2(\Sigma_g, \mathbb{R}) \rightarrow 0
\end{equation}
as in Equation~\eqref{short_exact_sequence_introduction}, we can identify the class ${[\Psi^*\omega' - \omega]}_{\reduced}$ with an element of the vector space $H^1(\Sigma_g, \mathbb{R})/\image(f_1 - \id)$. See also Equation~\eqref{mayer_vietoris_short_exact_sequence}. This identification gives a useful parametrization for the possible values of the de Rham invariant, as elements of $H^1(\Sigma_g, \mathbb{R})/\image(f_1 - \id)$, where the equivalences are described by Equation~\eqref{vanish_equation}. This parametrization depends on the chosen space $(M, \omega)$, relative to which we measure the differences.

\begin{remark}
    The fibration invariant and the de Rham invariant are both trivial when the first Betti number of the quotient space $b_1(M/S^1)$ is $1$. For example, this applies to Example~\ref{intro_example}. See Remarks~\ref{b1_trivial_fibration_invariant_remark} and~\ref{b1_trivial_de_rham_invariant_remark}, and also Corollaries~\ref{phi_t_diffeo_and_eq_symp_betti_one} and~\ref{phi_diffeo_and_phi_t_diffeo_betti_one}. The invariants are not trivial in general though -- In Appendix~\ref{examples_appendix} we describe calculations of these invariants in some cases where $b_1(M/S^1) > 1$. In general, these invariants are harder to compute than the first five invariants introduced in the paper.
\end{remark}

We now state our two main theorems, which classify our spaces up to equivariant symplectomorphisms.

\begin{theorem}[Existence Theorem]\label{existence_theorem}
    Suppose that $P \subset \mathbb{R}$ is a non-trivial discrete subgroup, $c_{\DuHe}$ is a positive real, and $g$ is a non-negative integer. Furthermore, let $(C_1,\dots,C_k):=((n_1,a_1),\dots,(n_k,a_k))$ be an unordered list of coprime residue classes, such that the sum $\sum_{j=1}^k \frac{a_j}{n_j}$ is an integer. In addition, let $[f]$ be an element in the mapping class group $\mcg(\Sigma_{g, C_1,\dots,C_k})$, and let $A$ be an element in the affine subspace~$\mathcal{H}$, introduced in Equation~\eqref{affine_subspace_intro}.
    
    Then, there exists a four-dimensional compact connected symplectic manifold $(M, \omega)$, equipped with a symplectic non-Hamiltonian circle action, with group of periods $P$, Duistermaat-Heckman constant $c_{\DuHe}$, genus invariant $g$, isotropy data invariant $C_1,\dots,C_k$, monodromy invariant the conjugacy class of $[f]$, and fibration invariant represented by $A$.
    
    Furthermore, for every $B$ in $H^1(\Sigma_g, \mathbb{R})/\image(f_1 - \id)$, there exists an invariant symplectic form $\omega'$ on $M$, such that $(M, \omega')$ has the same six invariants, and such that
    \begin{equation*}
        \overline{d^*}(B) = {[\omega' - \omega]}_{\reduced},
    \end{equation*}
    where $\overline{d^*}:H^1(\Sigma_g, \mathbb{R})/\image(f_1 - \id) \rightarrow H^2(M/S^1, \mathbb{R})$ is given by the short exact sequence in Equation~\eqref{real_short_exact_sequence_introduction}.
\end{theorem}
\begin{proof}
    By Theorem~\ref{existence_theorem_phi_diffeo}, there exists a space $(N, \sigma)$ with group of periods $P$, Duistermaat-Heckman constant $c_{\DuHe}$, genus invariant $g$, isotropy data invariant $C_1,\dots,C_k$, monodromy invariant the conjugacy class of $[f]$. We prove Theorem~\ref{existence_theorem_phi_diffeo} in Section~\ref{construction_section}, by constructing a symplectic form on a family of Seifert manifolds, and gluing it using a strengthened version of the local uniqueness theorem for complexity one spaces proved in~\cite{centered_hamiltonians}.
    
    By Lemma~\ref{fibration_existence_lemma}, there exists a space $(M, \omega)$ with the same five invariants as $(N, \sigma)$, but fibration invariant represented by $A$. To prove Lemma~\ref{fibration_existence_lemma}, we construct non-equivalent fibrations over the quotient space of $(N, \sigma)$, and show that they admit an invariant symplectic form that generates the $S^1$-action with the same circle-valued Hamiltonian.
    
    Lastly, for every $B$ in $H^1(\Sigma_g, \mathbb{R})/\image(f_1 - \id)$, using Lemma~\ref{de_rham_invariant_existence_lemma}, we perturb $\omega$ to another invariant symplectic form $\omega'$ on $M$, which also generates the $S^1$-action with the same circle-valued Hamiltonian, such that $(M, \omega')$ has the same six invariants as $(M, \omega)$, and with $\overline{d^*}(B) = {[\omega' - \omega]}_{\reduced}$.
\end{proof}

\begin{theorem}[Uniqueness Theorem]\label{uniqueness_theorem}
    Suppose that $(M, \omega)$ and $(M', \omega')$ are four-dimensional compact connected symplectic manifolds, both equipped with symplectic non-Hamiltonian $S^1$ actions. Let $X$ and $X'$ be the corresponding generating vector fields, and assume that the groups of periods of $\iota_X \omega$ and $\iota_{X'} \omega'$ are discrete.
    
    Then, the manifolds $(M, \omega)$ and $(M', \omega')$ are equivariantly symplectomorphic if and only if they have the same group of periods, Duistermaat-Heckman constant, genus invariant, isotropy data invariant, monodromy invariant, fibration invariant, and de Rham invariant.
\end{theorem}
\begin{proof}
    Fix circle-valued Hamiltonians $\Phi:M \rightarrow \mathbb{R}/P$ and $\Phi':M' \rightarrow \mathbb{R}/P$ that generate the circle actions. Suppose that the spaces have the same invariants. By Proposition~\ref{uniqueness_theorem_phi_diffeo}, since their group of periods, isotropy data invariant, genus invariant, and monodromy invariant agree, then their quotients are $\Phi$-diffeomorphic. To prove Proposition~\ref{uniqueness_theorem_phi_diffeo}, we apply tools from~\cite{tall_uniqueness} to reduce the classification up to $\Phi$-diffeomorphisms to the classification of their so called ``associated painted surface bundles'' (see Section~\ref{sheaves_section}). We then classify these painted surface bundles, formalizing the monodromy invariant in the process (see Section~\ref{painted_surface_bundle_section}).
    
    By Proposition~\ref{fibration_invariant_uniqueness}, because their fibration invariants agree, they are also $\Phi$-$T$-diffeomorphic. We prove Proposition~\ref{fibration_invariant_uniqueness} in Section~\ref{fibration_invariant_section}, following the ideas of Fintushel's~\cite{fintushel2} work on smooth circle actions, and Haefligher-Salem's~\cite{haefliger_salem} work on Čech cohomology for $T^n$-orbifolds.
    
    Lastly, by Proposition~\ref{de_rham_invariant_uniqueness}, because their Duistermaat-Heckman constant and de Rham invariants agree, they are equivariantly symplectomorphic. We prove Proposition~\ref{de_rham_invariant_uniqueness} in Section~\ref{de_rham_invariant_section} using equivariant Moser's trick.
\end{proof}

An immediate consequence of the main theorems is that every space with a non-discrete group of period is equivariantly diffeomorphic to at least one of the spaces in our classification:
\begin{corollary}\label{cor_eq_diff}
    Let $(M, \omega)$ be a four-dimensional compact connected symplectic manifold, equipped with a symplectic non-Hamiltonian $S^1$ action. Then there is an orientation-preserving equivariant diffeomorphism to at least one of the spaces constructed by Theorem~\ref{existence_theorem}.
\end{corollary}
\begin{proof}
    Following~\cite{symplectic_circle_actions}, we choose an invariant symplectic form $\omega'$ with integer cohomology class. It follows that the group of periods of $\iota_X \omega'$ is discrete. Therefore, by Theorem~\ref{uniqueness_theorem}, $(M, \omega')$ is equivariantly symplectomorphic to some space $(N, \sigma)$ constructed by Theorem~\ref{existence_theorem}. Hence, $(M, \omega)$ is equivariantly diffeomorphic to $(N, \sigma)$ by the same map.
\end{proof}
\begin{remark}
    If one drops the assumption that the group of periods of $\iota_X \omega$ is discrete, the function $\Phi: M \rightarrow \mathbb{R}/P$ can still be defined, however, the space $\mathbb{R}/P$ may no longer be a smooth manifold. The methods used throughout this paper are no longer applicable in this case. It's possible to perturb $\omega$ such that the group of periods becomes discrete, as in the proof of Corollary~\ref{cor_eq_diff} above. Then, Proposition~\ref{topological_model_proposition} gives a topological description of the quotient space, but it is unclear if the resulted genus invariant, isotropy data invariant, monodromy invariant, and fibration invariant will be independent of the perturbation. 
\end{remark}
\begin{remark}\label{high_complexity_remark}
    It seems reasonable to apply the techniques in this paper to symplectic non-Hamiltonian $T^{n-1}$ actions on $2n$-dimensional compact symplectic manifolds, under the assumptions that the orbits of the action are isotropic, and that the (generalized) group of periods of the action is a discrete subgroup of $\mathfrak{t}^* \cong \mathbb{R}^{n-1}$ (see~\cite[Section 5.2]{ratiu_ortega_book} or~\cite{cylinder_momentum_map} for more information on cylinder-valued momentum maps). This line of research will require some substantial work, and we expect the generalizations of the invariants to be much more complicated. For free actions, this direction of research was also suggested in~\cite[Section 2]{ginzburg_symplectic_actions}.
\end{remark}
\begin{remark}\label{kahler_structure_remark}
    A corollary of Karshon's work~\cite{periodic_hamiltonians} on the classification of four-dimensional Hamiltonian $S^1$-manifolds, is that every such space admits a Kähler structure. There is a conjecture that complexity one spaces admit an equivariant Kähler structure if and only if their painting is trivial (There is a work in progress~\cite{kahler_complexity_one_spaces} on this problem). It would be interesting to try and answer the same question for the spaces in this work. Inspired by the above conjecture, a natural question is whether there is a relation between the monodromy invariant and the existence of a Kähler structure. Recall that the Kodaira-Thurston manifold gives an example of a symplectic $S^1$-manifold with no Kähler structure. Using this paper, one can construct many other such examples in the spirit of~\cite{gray_morgan_marisa}. We note that for symplectic $T^2$ actions in dimension four, this problem was solved in~\cite{complex_pelayo_duistermaat}, based on Pelayo's classification theorem~\cite{pelayo_two_tori}. In the context of Kähler manifolds, it is worth mentioning Frankel's theorem~\cite{frankel}, which states that circle actions with fixed points on Kähler manifolds must be Hamiltonian.
\end{remark}
\begin{remark}\label{extending_to_pelayo_remark}
    Also in Karshon's work~\cite{periodic_hamiltonians}, she specifies conditions on when a Hamiltonian $S^1$ action can be extended to a Hamiltonian $T^2$ action. There are soon to be published works~\cite{extending_tall_complexity_one},~\cite{extending_short_complexity_one}, generalizing this result to higher dimensional complexity one spaces. They give a necessary and sufficient condition on the painting invariant, for when the action can be extended. It's an interesting question to find which non-Hamiltonian $S^1$ actions extend to non-Hamiltonian $T^2$ actions, which were classified in~\cite{pelayo_two_tori}. Specifically, the above works hint that there might be an interesting connection between the monodromy invariant and this problem.
\end{remark}

Following the last two remarks, we make an observation based on~\cite{complex_pelayo_duistermaat} and~\cite{non_complex_paper}. Consider a space $(M, \omega)$ with genus invariant $1$, trivial isotropy data, and monodromy invariant given by the gluing map $F:T^2 \rightarrow T^2$:
\begin{equation}\label{equation_mapping_torus_non_complex}
    F(p,q) = (p + q, p).
\end{equation}
Moreover, assume that the fibration invariant is non-trivial. It is known that the total space $E$ of a principal circle bundle $E \rightarrow T^2$, with Chern number one, is diffeomorphic to a mapping torus $M_F$ of $T^2$ with the gluing map $F$, and therefore by definition, $E \cong M/S^1$. Therefore $M \rightarrow M/S^1 \rightarrow T^2$ can be regarded as a non-trivial principal $S^1$-bundle over a non-trivial principal $S^1$-bundle over $T^2$, and hence, by the main theorem of~\cite{non_complex_paper}, $M$ does not admit a complex structure. However, by the main theorem of~\cite{complex_pelayo_duistermaat}, every compact symplectic four-manifold with a symplectic $T^2$ action admits an equivariant complex structure. Hence, the circle action on $M$ does not extend to a symplectic $T^2$ action. Note that spaces with these choices of invariants indeed exist by Theorem~\ref{existence_theorem} (see also Subsection~\ref{kodaira_like_subsection} for their relation to the Kodaira-Thurston manifold).
\\

In Appendix~\ref{examples_appendix}, we explicitly calculate the values of the invariants for selected symplectic non-Hamiltonian circle actions. In particular, we calculate them for diagonal circle actions on $S^2 \times S^2$ and $S^2 \times T^2$, for a symplectic non-Hamiltonian circle action on the Kodaira-Thurston manifold, as well as for other spaces that are $\Phi$-diffeomorphic to the aforementioned ones but not isomorphic to them.

\subsection{Outline}
In Section~\ref{preliminaries_section}, we introduce basic definitions and theorems that will be used throughout the paper.
In Section~\ref{section_circle_valued}, we define circle-valued Hamiltonians, and circle-valued Hamiltonian $S^1$-manifolds (These are spaces with a fixed circle-valued Hamiltonian). Moreover, we prove Lemmas~\ref{fiber_connectivity_of_circle_actions} and~\ref{duistermaat_heckman_constant}, regarding the connectivity and the Duistermaat-Heckman measure of the level sets of the circle-valued Hamiltonians.
In Section~\ref{moser_section} we define $\Phi$-$T$-diffeomorphisms (see Definition~\ref{phi_t_diffeo}), and use an equivariant version of Moser's trick to show Lemma~\ref{moser_with_same_cohomology}, which gives a condition on when a $\Phi$-$T$-diffeomorphism can be isotoped to an equivariant symplectomorphism. We also prove Corollary~\ref{classifying_circle_valued_spaces_instead}, which shows that every equivariant symplectomorphism between circle-valued Hamiltonian $S^1$-manifolds can be isotoped to respect the fixed Hamiltonians. This shows that our classification problem is equivalent to classifying circle-valued Hamiltonian $S^1$-manifolds up to isomorphism (i.e., up to equivariant symplectomorphisms that respect the Hamiltonians. See Definition~\ref{hamiltonian_s1_manifold_definition}).
In Section~\ref{phi_diffeo_section}, we define $\Phi$-diffeomorphisms (see Definition~\ref{phi_diffeo}), and use ideas from~\cite{haefliger_salem} to prove Lemma~\ref{lifting_phi_diffeo_condition}, which gives a condition on when a $\Phi$-diffeomorphism can be lifted to a $\Phi$-$T$-diffeomorphism.
In Section~\ref{sheaves_section}, we define a unique up to isomorphism ``painted surface bundle" for every space, and show that two spaces are $\Phi$-diffeomorphic if and only if their associated painted surface bundles are isomorphic.
In Section~\ref{painted_surface_bundle_section}, we classify (legal) painted surface bundles up to isomorphism, and formalize the monodromy invariant (see Definition~\ref{monodromy_invariant_definition}). We show that two spaces have $\Phi$-diffeomorphic quotients if and only if they have the same group of periods, genus invariant, isotropy data invariant, and monodromy invariant (see Proposition~\ref{uniqueness_theorem_phi_diffeo}).
In Section~\ref{construction_section}, we construct a circle-valued Hamiltonian $S^1$-manifold for every (valid) choice of values for the group of periods, genus invariant, isotropy data invariant, Duistermaat-Heckman constant, and monodromy invariant (see Theorem~\ref{existence_theorem_phi_diffeo}). Together with the previous section, this gives a full classification of our spaces up to $\Phi$-diffeomorphisms.
In Section~\ref{fibration_invariant_section}, we formalize the fibration invariant (see Definition~\ref{fibration_invariant_definition}), and show that two spaces with $\Phi$-diffeomorphic quotients are $\Phi$-$T$-diffeomorphic if and only if their fibration invariants agree (see Lemma~\ref{fibration_invariant_uniqueness}). Moreover, we construct a space for every value of the fibration invariant (see Lemma~\ref{fibration_existence_lemma}), thus giving a full classification of our spaces up to $\Phi$-$T$-diffeomorphisms.
In Section~\ref{de_rham_invariant_section}, we formalize the de Rham invariant (see Definition~\ref{de_rham_invariant_definition} and Proposition~\ref{de_rham_invariant_description}), and show that two $\Phi$-$T$-diffeomorphic spaces are equivariantly symplectomorphic if and only if their de Rham invariants agree (see Lemma~\ref{de_rham_invariant_uniqueness}). Moreover, we construct an appropriate symplectic form for every value of the de Rham invariant (see Lemma~\ref{de_rham_invariant_existence_lemma}), thus giving a full classification of our spaces up to equivariant symplectomorphisms (see Theorems~\ref{existence_theorem} and~\ref{uniqueness_theorem} and their proofs in the introduction).
In Appendix~\ref{examples_appendix}, we present some calculations of the values of the invariants for diagonal actions on product of surfaces, for a symplectic circle action on the Kodaira-Thurston manifold, and for other related spaces.
In Appendix~\ref{mcg_appendix}, we prove a classical theorem on the equivalence of different definitions of the mapping class group of an oriented surface with marked points, whose proof we have not found in the literature. This claim is needed in Section~\ref{painted_surface_bundle_section}, for showing that the monodromy invariant does not depend on the smooth structure of the painted surface bundles.


\section{Preliminaries}\label{preliminaries_section}
In this section, we remind the reader of definitions and results about symplectic actions that will be used throughout this paper.
\\

Let a torus $T^m$ act on a symplectic manifold $(M^{2n}, \omega)$, by symplectomorphisms. Let $\mathfrak{t}$ and $\mathfrak{t}^*$ be the Lie algebra and dual Lie algebra of $T^m$, and let $\langle \cdot, \cdot \rangle : \mathfrak{t}^* \times \mathfrak{t} \rightarrow \mathbb{R}$ be the natural pairing. A \textbf{momentum map} $\Phi: M\rightarrow \mathfrak{t}^*$ is a $T^m$-invariant map that satisfies
\begin{equation*}
    -d\langle \Phi, \xi \rangle = \iota_{X_\xi} \omega
\end{equation*}
for every $\xi \in \mathfrak{t}$, where $X_\xi \in \mathfrak{X}(M)$ is the vector field induced by $\xi$. A symplectic action that admits a momentum map is called a \textbf{Hamiltonian action}.
\\

Let $y \in \mathfrak{t}^*$ be a point in the image of $\Phi$. If $T^m$ acts freely on the level set $\Phi^{-1}(y)$, then the quotient space $\Phi^{-1}(y)/T^m$ is a smooth manifold of dimension $2n - 2m$. Famously, Mardsen and Weinstein showed in~\cite{mardsen_weinstein} that the quotient space $\Phi^{-1}(y)/T^m$ admits a unique symplectic structure $\omega_{\reduced}$ that satisfies:
\begin{equation*}
    \pi^* \omega_{\reduced} = i^* \omega
\end{equation*}
where $i:\Phi^{-1}(y) \rightarrow M$ is the inclusion map, and $\pi:\Phi^{-1}(y) \rightarrow \Phi^{-1}(y)/T^m$ is the quotient map. The symplectic manifold $(\Phi^{-1}(y)/T^m, \omega_{\reduced})$ is called the \textbf{reduced space} of $(M, \omega)$ at $y$. If $T^m$ acts non-freely on the level set $\Phi^{-1}(y)$, but $y$ is a regular value, then the same procedure yields a symplectic orbifold $(\Phi^{-1}(y)/T^m, \omega_{\reduced})$. The singular points of the orbifold correspond to the non-free orbits in the level set. See~\cite{lerman_tolman} for a comprehensive treatment of symplectic orbifolds and symplectic reduction.
\\

A \textbf{Hamiltonian $T$-manifold} $(M, \omega, \Phi, U)$ over an open convex set $U \subset \mathfrak{t}^*$ is a connected symplectic manifold $(M, \omega)$ with an effective $T$-action generated by a momentum map $\Phi: M \rightarrow U$. A Hamiltonian $T$-manifold is called \textbf{proper} if the Hamiltonian $\Phi$ is proper as a map to $U$. The dimension of $T$ is at most half of the dimension of $M$, and the number $\frac{1}{2}\dim M - \dim T$ is called the \textbf{complexity} of the space. A proper Hamiltonian $T$-manifold with complexity $k$ is also called a \textbf{complexity $k$ space}. For example, a four-dimensional proper Hamiltonian $S^1$-manifold is also called a four-dimensional complexity one space.
\\

In this paper, we are interested in symplectic circle actions on four-dimensional manifolds. Compact four-dimensional Hamiltonian $S^1$-manifolds were classified up to equivariant symplectomorphisms in terms of decorated graphs in~\cite{periodic_hamiltonians}. Later, this classification was extended in~\cite{centered_hamiltonians},~\cite{tall_uniqueness}, and~\cite{tall_existence} to complexity one spaces of all dimensions, under the additional assumption that the spaces are tall. A complexity one space is called \textbf{tall} if all of its reduced spaces are $2$-dimensional. Very roughly, their proof realizes the orbit space of a complexity one space as a surface bundle over the momentum image, and examine it with differential topology tools.
\\

Given a proper Hamiltonian $T$-manifold, the subset of points in $M$ that belong to non-free orbits is a union of symplectic submanifolds of codimension at least $2$. An orbit is called \textbf{exceptional} if it has a neighborhood in which every other orbit with the same momentum value has a strictly smaller stabilizer. For proper complexity one spaces, every level set of the momentum map has finitely many exceptional orbits.
\\

Let $x \in M$ be a point in an orbit $\mathcal{O}$. The \textbf{isotropy representation of $x$} is the representation of the stabilizer of $\mathcal{O}$ on the tangent space $T_x M$. This representation is a direct sum of a trivial representation and the representation of the stabilizer on the \textbf{symplectic slice} ${(T_x O)}^\omega / (T_x O \cap {(T_x O)}^\omega )$. If $y \in M$ is another point in the orbit $\mathcal{O}$, the isotropy representations of $x$ and $y$ are linearly symplectically isomorphic. Therefore, the orbit $\mathcal{O}$ have a well defined isomorphism class of representations of its stabilizer, and we call it the \textbf{isotropy representation of $\mathcal{O}$}. Given a level set of $\Phi$, the finite unordered list of isotropy representations of the exceptional orbits is called the \textbf{isotropy data} of the level set. 
\\

The Guillemin-Sternberg-Marle \textbf{local normal form theorem} gives a description of the neighborhood of an orbit, under the presence of a Hamiltonian action of a compact Lie group (see~\cite{GS_local_normal_form},~\cite{marle_local_normal_form}). More precisely, it shows that the neighborhood is completely determined by the isotropy representation of the orbit. In this paper we only need the theorem for non-fixed orbits of the action of $S^1$ on a four-dimensional manifold. A good reference for the local normal form theorem for four-dimensional Hamiltonian $S^1$-manifolds is~\cite[Appendix A]{periodic_hamiltonians}. For simplicity, we will only state the theorem for our specific case.
\\

Let $\mathbb{Z}_n \subset S^1$ be the cyclic subgroup of $S^1$ of order $n$, and let $\nu = e^{\frac{2\pi i}{n}}$ be the standard generator of the subgroup. Let $\mathbb{Z}_n$ act effectively and linearly on $\mathbb{C}$. When the $n > 1$, the action can be described by the unique integer $0 < k < n$, relatively prime to $n$, that satisfies $\nu \cdot z = e^{\frac{2\pi ki}{n}} z$. Moreover, let $T^*S^1$ be the cotangent bundle of $S^1$, with coordinates $t \mod 1$ and $h$, and let $\mathbb{Z}_n$ act on $T^*S^1$ by $\nu \cdot (t, h) = (\frac{1}{n} + t, h)$.

We define the \textbf{local model} related to the action of $\mathbb{Z}_n$ on $\mathbb{C}$ by taking the quotient of $T^*S^1 \times \mathbb{C}$ by the diagonal action of $\mathbb{Z}_n$. It is a proper Hamiltonian $S^1$-manifold, denoted as $S^1 \times_{\mathbb{Z}_n} \mathbb{C} \times \mathbb{R}$. Its points are written in the form $[t, z, h]$, where we identify $[\frac{1}{n} + t, z, h]$ with $[t, e^{\frac{2\pi ki}{n}} z, h]$. The $S^1$ action of the local model is defined by $e^{2\pi i s} \cdot [t, z, h] = [s + t, z, h]$. For any $c \in \mathbb{R}$, this action is generated by the Hamiltonian $\Phi([t, z, h]) = c + h$.
\begin{theorem}\label{simple_local_normal_form}
    Let $(M, \omega, \Phi)$ be a four-dimensional Hamiltonian $S^1$-manifold, let $\mathcal{O}$ be a non-fixed orbit of the action, and let $\mathbb{Z}_n$ be its stabilizer, where $n$ is a positive integer.
    
    Then there exists an effective linear action of $\mathbb{Z}_n$ on $\mathbb{C}$, a neighborhood $U$ of $\mathcal{O}$, a neighborhood $V$ of the orbit $[t,0,0]$ in the local model $S^1 \times_{\mathbb{Z}_n} \mathbb{C} \times \mathbb{R}$ associated with the $\mathbb{Z}_n$ action, and an equivariant symplectomorphism $\varphi:V \rightarrow U$, such that:
    \begin{enumerate}
        \item $\varphi$ sends the orbit $[t,0,0]$ to the orbit $\mathcal{O}$.
        \item $(\Phi \circ \varphi)([t, z, h]) = \Phi(\mathcal{O}) + h$.
    \end{enumerate}
    If $n > 1$, the $\mathbb{Z}_n$ action is given by a unique integer $0 < k < n$ relatively prime to $n$.
\end{theorem}

Let $(M, \omega, \Phi, U)$ be a four-dimensional proper Hamiltonian $S^1$-manifold. Let $\mathcal{O}$ be an orbit in a regular level set of $\Phi$. By the normal local form theorem, $\mathcal{O}$ is exceptional if and only if it is not free. Its isotropy representation can be described by its stabilizer $\mathbb{Z}_n$, and the integer $k$ mentioned in the statement above that describes the representation of $\mathbb{Z}_n$ on the symplectic slice. We define a \textbf{coprime residue class} to be a pair of relatively prime numbers $(n, k) \in \mathbb{N}^2$ that satisfy $n > 1$ and $0 < k < n$. Hence, the isotropy representation of $\mathcal{O}$ corresponds to a coprime residue class, and the neighborhood of the orbit is in turn determined by this class by the local normal form theorem.
\\

The following Lemma was proved by McDuff in~\cite{symplectic_circle_actions} by examining the level sets of a circle-valued Hamiltonian:
\begin{lemma}[Proposition 2 in~\cite{symplectic_circle_actions}]\label{no_fixed_points}
    Let $S^1$ act effectively on a four-dimensional compact symplectic manifold. Assume that the action is symplectic but not Hamiltonian. Then the action has no fixed points.
\end{lemma}
To define a circle-valued Hamiltonian, McDuff first perturbed the symplectic form so that it will have a rational cohomology class. We will discuss this in more details in Section~\ref{section_circle_valued}, and we will need the following definition.
\begin{definition}
    Let $\eta$ be a closed $k$-dimensional differential form in a compact manifold $M$. The \textbf{group of periods} of $\eta$ is the following subgroup of $\mathbb{R}$:
    \begin{equation*}
        P(\eta) := \{\int_A \eta\ |\ A \in H_k(M, \mathbb{Z})\}
    \end{equation*}
\end{definition}
The closed form $\eta$ is exact if and only if $P(\eta) = \{0\}$. Therefore, if $S^1$ acts symplectically on a symplectic manifold $(M, \omega)$ with a generating vector field $X \in \mathfrak{X}(M)$, then the action is Hamiltonian if and only if the group of periods of $\iota_X \omega$ is trivially zero. Throughout this paper, we will refer to $P(\iota_X \omega)$ as the \textbf{group of periods of the space}, and often just denote it as $P$.
\\

Let $(M, \omega, \Phi, U)$ be a proper Hamiltonian $T$-manifold. The volume form $\frac{1}{n!}\omega^n$ defines a measure on $M$, which is called the \textbf{Liouville measure}. The pushforward of the Liouville measure through the momentum map $\Phi$ is called the \textbf{Duistermaat-Heckman measure}. The Duistermaat-Heckman measure can be written as $\rho_{\DuHe}\lambda$, where $\lambda$ is the Lesbegue measure on $U$, and $\rho_{\DuHe}: U \rightarrow \mathbb{R}_{\ge 0}$ is called the \textbf{Duistermaat-Heckman function}. By the Duistermaat-Heckman Theorem (see~\cite{duistermaat_heckman}), the Duistermaat-Heckman function is a piece-wise polynomial function whose degree is at most the complexity $k$ of the action. If $y \in U$ is a regular point of the momentum map $\Phi$, the Duistermaat-Heckman function at $y$ is equal to the symplectic volume of the reduced space at $y$:
\begin{equation*}
    \rho_{\DuHe}(y) = {(2\pi)}^k \int_{\Phi^{-1}(y)/T} \frac{\omega_{\reduced}^{n-k}}{(n-k)!}
\end{equation*}
For four-dimensional proper Hamiltonian $S^1$-manifolds, the Duistermaat-Heckman function is piecewise linear, and its derivative is given by the Euler/Chern class:
\begin{theorem}[Duistermaat-Heckman Theorem]\label{duistermaat_heckman_thm}
    Let $(M, \omega, \Phi, U)$ be a four-dimensional proper Hamiltonian $S^1$-manifold. Let $y \in \mathbb{R}$ be a regular value of $\Phi$. Then in a neighborhood $V$ of $y$, the Duistermaat-Heckman function $\rho_{\DuHe}: V \rightarrow \mathbb{R}_{\ge 0}$ is a linear function with derivative equals to $2\pi$ times the Seifert Euler number of the Seifert fibration $\Phi^{-1}(y) \rightarrow \Phi^{-1}(y)/S^1$.
\end{theorem}
More often, this theorem is stated with the generalized Chern number instead of the Seifert Euler number. Both notions coincide, and therefore the different statements are equivalent. We choose to work with the Seifert Euler number since its standard expression (Equation~\eqref{seifert_euler_equation}) is given in terms of the isotropy representations of the exceptional orbits in the level set $\Phi^{-1}(y)$. See~\cite{seifert_fibrations} and~\cite{orbifold_seifert} for more information about Seifert fibrations, and the definition of the Seifert Euler number (Equation~\eqref{seifert_euler_equation}). For the correspondence between the Seifert Euler class and the generalized Chern class see Section 1.2.1 in~\cite{orbifold_ham_s1}, and particularly Remark 1.67. For the formula of the generalized Chern class for complex bundles over 2-dimensional orbifolds, see Proposition 4.2.1 in~\cite{orbifold_cohomology}.

Let $\{g;(1,b),(n_1, a_1),\dots,(n_k, a_k)\}$ be the Seifert invariants of the Seifert fibration $\Phi^{-1}(y) \rightarrow \Phi^{-1}(y)/S^1$. Then the Seifert Euler number is given by:
\begin{equation}\label{seifert_euler_equation}
    e(E \rightarrow F) = -b - \sum_{j=1}^k{\frac{a_j}{n_j}}.
\end{equation}
The pairs $(n_1, a_1),\dots,(n_k, a_k)$ are precisely the coprime residue classes that correspond to the isotropy data of $\Phi^{-1}(y)$. Hence, the Seifert Euler number can be calculated from the isotropy data, up to an integer (since $b$ cannot be read from the isotropy data).
\\

Lastly, we give a short reminder about basic forms. Let a compact connected Lie group $G$ act on a compact connected smooth manifold $M$, and let $\mathfrak{g}$ be the Lie algebra of $G$. A differential form $\alpha$ on $M$ is called \textbf{basic} if it is $G$-invariant and if for every $\xi \in \mathfrak{g}$, the induced vector field $\xi_M$ on $M$ satisfies $\iota_{\xi_M}\alpha = 0$. The complex of basic forms is a subcomplex of the complex of differential forms. By~\cite{koszul_basic}, its cohomology is isomorphic to the Čech/singular cohomology of $M/G$, with real coefficients. If $G = S^1$, and $X$ is the generating vector field of the action, then the form $\alpha$ is basic if and only if it satisfies the following two properties:
\begin{itemize}
    \item $\mathcal{L}_X \alpha = 0$,
    \item $\iota_X \alpha = 0$.
\end{itemize}
Moreover, if every non-free orbit of the action has a finite stabilizer, then there is a correspondence between differential forms on the orbifold $M/S^1$, and basic forms on $M$, given by the pullback map (see for instance~\cite[Corollary B.36]{ggk_book}).

\section{Tight circle-valued Hamiltonians}\label{section_circle_valued}
In this section, we define circle-valued Hamiltonians for symplectic non-Hamiltonian circle actions with a discrete group of periods $P(\iota_X \omega)$. We prove that if a circle-valued Hamiltonian is defined as a map to $\mathbb{R}/P(\iota_X \omega)$, its level sets must be connected. Moreover, we show that if the manifold is four-dimensional, the Duistermaat-Heckman measure is constant throughout the level sets of the circle-valued Hamiltonian, and the Seifert Euler number of each level set is zero.
\\

Let $(M, \omega)$ be a compact connected symplectic manifold, with an effective symplectic non-Hamiltonian $S^1$ action. For the purpose of studying these spaces, McDuff introduced circle-valued Hamiltonians in~\cite{symplectic_circle_actions}. She assumed that the class $[\omega] \in H^2(M, \mathbb{R})$ is in the image of the inclusion of $H^2(M, \mathbb{Z})$ into $H^2(M, \mathbb{R})$, and deduced that the group of periods $P$ of the one-form $\iota_X \omega$ is a subgroup of $\mathbb{Z}$. This allowed her to define an $\mathbb{R}/\mathbb{Z}$-valued Hamiltonian that satisfies $-d \Phi = \iota_X \omega$. For more information on the definition of circle-valued Hamiltonians, see the expository note~\cite{pelayo_circle_valued_ham}.

In this paper, we slightly generalize McDuff's definition of circle-valued Hamiltonians by omitting the integrability assumption on the class $[\omega]$. Instead, we assume that the group of periods $P$ of the one-form $\iota_X \omega$ is a discrete subgroup of $\mathbb{R}$. This allows us to define an $\mathbb{R}/\Gamma$-valued Hamiltonian, for any discrete group $\Gamma \subset \mathbb{R}$ that contains $P$ as a subgroup. The case $\Gamma = \mathbb{Z}$ corresponds to McDuff's definition.

An $\mathbb{R}/\Gamma$-valued Hamiltonian always factors through an $\mathbb{R}/P$-valued Hamiltonian. We will see that $\mathbb{R}/P$-valued Hamiltonians exhibit a special property, namely that their level sets are connected. We refer to such circle-valued Hamiltonians as \textbf{tight}.
\begin{definition}\label{circle_valued_hamiltonian_definition}
    Let $(M, \omega)$ be a symplectic manifold, $X \in \mathfrak{X}(M)$ a symplectic vector field, and $P\subset \mathbb{R}$ the group of periods of $\iota_X \omega$. Assume that $P$ is discrete, and let $0 \neq \Gamma \subsetneq \mathbb{R}$ be a non-trivial discrete group containing $P$. A \textbf{circle-valued Hamiltonian} is a smooth map $\Phi:M \rightarrow \mathbb{R}/\Gamma$ such that $\iota_X \omega = -d \Phi$. If $\Gamma = P$, we call $\Phi$ a \textbf{tight} circle-valued Hamiltonian.
\end{definition}
\begin{example}\label{ham_ex}
    Let $X$ be a Hamiltonian vector field on $(M, \omega)$, generated by $\Phi:M \rightarrow \mathbb{R}$. The one-form $\iota_X \omega = -d \Phi$ is exact, and therefore its group of periods $P$ is trivial. Hence, for every discrete nonzero subgroup $0 \neq \Gamma \subset \mathbb{R}$, the induced map $\tilde \Phi: M \rightarrow \mathbb{R}/\Gamma$ is a circle-valued Hamiltonian. Note that it is not tight.
\end{example}
\begin{remark}\label{tight_means_non_hamiltonian}
    It follows from Example~\ref{ham_ex} that a symplectic vector field $X$ that admits a tight circle-valued Hamiltonian must be non-Hamiltonian. The requirement that $\Gamma$ is a non-trivial subgroup together with the definition of tight ($P = \Gamma$), implies that $P$ is non-trivial, and thus that $\iota_X \omega$ is not exact, and therefore that $X$ is not Hamiltonian.
\end{remark}
\begin{example}\label{torus_ex}
    Let $X=\frac{\partial}{\partial q}$ be a vector field on the torus $(\mathbb{T}^2 \cong \mathbb{R}^2/\mathbb{Z}^2, \omega=dp \wedge dq)$. The vector field $X$ is symplectic and the group of periods of $\iota_{\frac{\partial}{\partial q}} \omega$ is $\mathbb{Z}$. Let $k \in \mathbb{N}$ be a natural number and define $\Gamma$ to be $\frac{1}{k}\mathbb{Z}$. Then the function $\Phi:\mathbb{T}^2 \rightarrow \mathbb{R}/\Gamma$ defined by $\Phi(p, q) = p \mod \frac{1}{k}$ is a circle-valued Hamiltonian for $X$. It is tight when $k=1$.
\end{example}
The following is an example of a vector field with a non-discrete group of periods, which does not admit a circle-valued Hamiltonian.
\begin{example}
    Let $X=\frac{\partial}{\partial q}$ be a vector field on $\mathbb{T}^4 \cong \mathbb{R}^4/\mathbb{Z}^4$. It is symplectic with respect to the symplectic form $\omega = dp_1 \wedge dq_1 + \sqrt{2}dq_1 \wedge dp_2 + dp_2 \wedge dq_2$. The group of periods of $\iota_{\frac{\partial}{\partial q}} \omega$ is $P = \mathbb{Z} + \sqrt{2}\mathbb{Z}$, and it is not discrete. Note that the function $\Phi:\mathbb{T}^4 \rightarrow \mathbb{R}/P$ can be still defined as $\Phi(q_1, p_1, q_2, p_2) = p_1 + \sqrt{2}p_2 \mod P$, but $P$ is not a circle anymore, and each level set of $\Phi$ is dense in $M$, and not compact.
\end{example}
When the first Betti number of the quotient space $M/S^1$ is one, the space always admits a circle-valued Hamiltonian:
\begin{lemma}\label{betti_one_group_of_periods_discrete}
    Let $(M, \omega)$ be a four-dimensional compact connected symplectic manifold, equipped with a symplectic non-Hamiltonian $S^1$ action, with generating vector field $X$. Assume that $b_1(M/S^1) = 1$. Then the group of periods $P$ of $\iota_X \omega$ is discrete.
\end{lemma}
\begin{proof}
    Let $\beta$ be the one-form on $M/S^1$, given by integration of $\omega$ along fibers of the $S^1$-orbibundle $\pi:M \rightarrow M/S^1$. Then for every homology class $A \in H^1(M, \mathbb{Z})$, we have:
    \begin{equation}\label{eq_in_proof_of_discrete}
        \int_A \iota_X \omega = \frac{1}{2\pi}\int_{S^1 \cdot A} \omega = \frac{1}{2\pi} \int_{\pi_* A} \beta.
    \end{equation}
    By assumption, $b_1(M/S^1) = 1$, and therefore $H^1(M/S^1, \mathbb{Z}) \cong \mathbb{Z} \times T$, where $T$ is a finite group (Actually, Proposition~\ref{topological_model_proposition} shows that $T$ is always trivial). Let $B$ be a generator of the $\mathbb{Z}$ component, and let $c:= \int_B \beta$ be its integral. Then by Equation~\eqref{eq_in_proof_of_discrete}, the integral $\int_A \iota_X \omega$ is always an integer multiple of $\frac{c}{2\pi}$, and therefore the group of periods of $\iota_X \omega$ is discrete.
\end{proof}
Next, we want to generalize Hamiltonian $S^1$-manifolds to allow for circle-valued Hamiltonians.
\begin{definition}\label{hamiltonian_s1_manifold_definition}
    Let $(M, \omega)$ be a compact connected symplectic manifold with an effective $S^1$ action generated by a circle-valued Hamiltonian $\Phi:M \rightarrow \mathbb{R}/\Gamma$. Then $(M, \omega, \Phi)$ is called a \textbf{circle-valued Hamiltonian $S^1$-manifold}. If $\Phi$ is tight, we say that the space is \textbf{tight}.
    
    We say that two circle-valued Hamiltonian $S^1$-manifolds are isomorphic if $\Gamma = \Gamma'$, and there is an equivariant symplectomorphism $\varphi:M \rightarrow M'$ such that $\Phi = \Phi' \circ \varphi$.
\end{definition}
\begin{example}\label{ham_space_ex}
    Let $(M, \omega, \Phi)$ be a compact connected Hamiltonian $S^1$-manifold. Then for every discrete nonzero subgroup $\Gamma \subset \mathbb{R}$, the induced map $\tilde \Phi: M \rightarrow \mathbb{R}/\Gamma$ makes $(M, \omega, \tilde \Phi)$ a circle-valued Hamiltonian $S^1$-manifold. Note that it is not tight.
\end{example}
\begin{example}\label{torus_space_ex}
    Following Example~\ref{torus_ex}, let $S^1$ act on $(\mathbb{T}^2, dp \wedge dq)$ with generating vector field $X = \frac{1}{2\pi}\frac{\partial}{\partial q}$, and let $k \ge 1$ be a positive integer. Then the function $\Phi(p, q) = \frac{1}{2\pi}p \mod \frac{1}{2\pi k}$ is a circle-valued Hamiltonian that generates the circle action, and $(\mathbb{T}^2, dp \wedge dq, \Phi)$ is a circle-valued Hamiltonian $S^1$-manifold. Note that each level set has $k$ connected components. Hence, the level sets of $\Phi$ are connected if and only if $k=1$ if and only if $\Phi$ is tight. Lemma~\ref{fiber_connectivity_of_circle_actions} below shows that this is always the case for non-Hamiltonian circle actions.
\end{example}
\begin{example}
   More generally, let $(M, \omega)$ be a compact connected symplectic manifold, with a symplectic non-Hamiltonian circle action. By~\cite[Lemma 1]{symplectic_circle_actions}, there exists an $S^1$-invariant symplectic form $\omega'$ such that the action admits a circle-valued Hamiltonian $\Phi:M \rightarrow \mathbb{R}/\mathbb{Z}$. Then $(M, \omega', \Phi)$ is a circle-valued Hamiltonian $S^1$-manifold. Moreover, $\omega'$ can be normalized such that $\Phi$ is tight.
\end{example}

Let $(M, \omega, \Phi)$ be a circle-valued Hamiltonian $S^1$-manifold, and let $I \subsetneq \mathbb{R}/\Gamma$ be an open interval. Then the space $\Phi^{-1}(I)$ is a proper Hamiltonian $S^1$-manifold over $i(I)$ with the Hamiltonian $i \circ \Phi$, where $i:I \rightarrow \mathbb{R}$ is a lift of $I$ from $\mathbb{R}/\Gamma$ to $\mathbb{R}$. Leveraging this idea we can use techniques developed for proper Hamiltonian actions. For example, we now apply the standard Morse-Bott arguments used for real-valued Hamiltonians to prove Lemma~\ref{fiber_connectivity_of_circle_actions}:
\begin{replemma}{fiber_connectivity_of_circle_actions}[In the language of Hamiltonian $S^1$-manifolds]
    Let $(M, \omega, \Phi)$ be a tight circle-valued Hamiltonian $S^1$-manifold. Then the level sets of $\Phi:M \rightarrow \mathbb{R}/P$ are connected.
\end{replemma}
\begin{proof}
    A Hamiltonian function generating a circle action is Morse-Bott with even indices (see~\cite[Lemma 5.5.8]{intro_to_symp_topology}). This claim extends to circle-valued Hamiltonians by carrying the same arguments after restricting to a preimage of a small interval in $\mathbb{R}/P$.
    
    Choose a Riemannian metric, and look at the normalized gradient flow through regular level sets. The number of connected components in each level set stays unchanged. When we cross a critical value, the number of components can only change if the critical point is of index zero/one, or coindex zero/one. Because the indices must be even, critical points of index one or coindex one cannot appear. Moreover, there are no local minima or maxima points by McDuff's proof of Lemma~\ref{no_fixed_points} (see~\cite[Proposition 2]{symplectic_circle_actions}), and therefore there are also no critical points of index zero or coindex zero.
    
    Let $y\in \mathbb{R}/P$ be some value of $\Phi$. Its fiber has finitely many connected components $A_1, A_2,\dots,A_k$. Flowing with the normalized gradient flow (and crossing critical points), and circling once around the image $\mathbb{R}/P$, each connected component $A_i$ is mapped to some other connected component $A_{\sigma(i)}$, where $\sigma$ is some permutation in $S_k$. By the connectivity of $M$, the permutation $\sigma$ has exactly one cycle.

    By the definition of the group of periods $P(\iota_X \omega)$, there exists a loop $\gamma:S^1 \rightarrow M$ with integral $\int_\gamma \iota_X \omega = \tau$, where $\tau$ is the positive generator of~$P$. It defines a map $\Phi \circ \gamma: S^1 \rightarrow \mathbb{R}/P$. Because $P$ is nontrivial and discrete, $\mathbb{R}/P$ is isomorphic to $S^1$, and the map $\Phi \circ \gamma$ has a well defined degree. This degree is the amount of times the image of $\gamma$ through $\Phi$ rotated around $\mathbb{R}/P$, and it satisfies the relation
    \begin{equation*}
        \int_\gamma \iota_X \omega = \int_\gamma -d \Phi = -\tau \deg(\Phi \circ \gamma).
    \end{equation*}
    
    Without loss of generality, $\gamma$ is based at $x \in A_i$. Hence, $\deg(\Phi \circ \gamma)$ must be a multiple of the order of the cycle of $A_i$ in $\sigma$. Since $\sigma$ only has one cycle, and its order is $k$, we have that
    \begin{equation*}
        \int_\gamma \iota_X \omega = mk\tau,
    \end{equation*}
    for some integer $m$. Therefore, by the choice of $\gamma$, we have $\tau = mk\tau$, and since $\tau \ne 0$, then $k$ must be equal to $1$, hence the level sets are connected.
\end{proof}
\begin{remark}
    Using the formulation of covering spaces, one can alternatively deduce Lemma~\ref{fiber_connectivity_of_circle_actions} from~\cite[Theorem 3.1]{benoist}. In fact, his proof also applies for higher-dimensional torus actions generated by cylinder-valued momentum maps. See also~\cite{giacobbe}.
\end{remark}
The following lemma states that for every four-dimensional tight circle-valued Hamiltonian $S^1$-manifold, the isotropy data of the level sets, and the genus of the reduced spaces do not vary between the values of the circle-valued Hamiltonian. This was already implicit in McDuff's paper~\cite{symplectic_circle_actions}. Moreover, the part about the genus can be deduced from~\cite[Corollary 9.7]{centered_hamiltonians}. For completeness, we give a statement and a direct proof.
\begin{lemma}\label{same_genus_and_isotropy_data}
    Let $(M, \omega, \Phi)$ be a four-dimensional tight circle-valued Hamiltonian $S^1$-manifold. Then all the level sets of $\Phi:M \rightarrow \mathbb{R}/P$ have the same isotropy data, and all the reduced spaces are homeomorphic to the same compact oriented surface.
\end{lemma}
\begin{proof}
    Take an invariant compatible metric $g$ and a corresponding almost complex structure $J$. By Lemma~\ref{no_fixed_points}, the action has no fixed points, and therefore the gradient vector field $Y := -JX$ is non-zero on all of $M$. Therefore, we can normalize $Y$ such that $d \Phi(Y) = 1$, and take the flow in the direction of $Y$. This gives an equivariant diffeomorphism between each pair of level sets. Therefore, the isotropy data is the same in every level set of $\Phi$, and the reduced spaces are all homeomorphic. Since the reduced spaces are compact, and also connected by Lemma~\ref{fiber_connectivity_of_circle_actions}, then they are all homeomorphic to a compact oriented surface with the same genus.
\end{proof}
Lemma~\ref{same_genus_and_isotropy_data} introduces two invariants of four-dimensional tight circle-valued Hamiltonian $S^1$-manifolds: the \textbf{genus invariant}, and the \textbf{isotropy data invariant}. Because $\Phi$ is unique up to translation, these invariants are preserved by equivariant symplectomorphisms that do not necessarily preserve $\Phi$. (In fact, if there exists an equivariant symplectomorphism then there exists one that preserves $\Phi$. We show this in Proposition~\ref{invariance_of_translations}). It follows that they are also invariants of symplectic non-Hamiltonian circle actions on compact connected four-manifolds with discrete group of periods, up to equivariant symplectomorphisms.

Another concept that can be generalized from real-valued Hamiltonians to circle-valued Hamiltonians is the Duistermaat-Heckman theorem. The Duistermaat-Heckman measure can still be defined as the push forward of the Liouville measure, and since a preimage of a small interval in $\mathbb{R}/\Gamma$ is a proper Hamiltonian $S^1$-manifold, we can apply the Duistermaat-Heckman theorem near level sets. Previous works on the Duistermaat-Heckman measure for circle-valued Hamiltonians can be found in Section 6 of~\cite{ginzburg_symplectic_actions} and in~\cite{weitsman_dh_circle}. We now prove Lemma~\ref{duistermaat_heckman_constant}, which is closely related to the mentioned section in Ginzburg's paper, and to McDuff's proof of Lemma~\ref{no_fixed_points} (see~\cite[Proposition 2]{symplectic_circle_actions}).
\begin{replemma}{duistermaat_heckman_constant}[In the language of Hamiltonian $S^1$-manifolds]
    Let $(M, \omega, \Phi)$ be a four-dimensional tight circle-valued Hamiltonian $S^1$-manifold. The Duistermaat-Heckman function is constant, and the Seifert Euler number of each level set is zero.
\end{replemma}
\begin{proof}[Proof of Lemma~\ref{duistermaat_heckman_constant}]
    Let $y \in \mathbb{R}/P$ be a value of $\Phi$. Restricting to the preimage of a neighborhood $U$ of $y$, we get a proper Hamiltonian $S^1$-manifold over $U$. By the Duistermaat-Heckman Theorem, the Duistermaat-Heckman function is a linear function with derivative equals to $2\pi$ times the level set's Seifert Euler number, which only changes when crossing critical points (see Theorem~\ref{duistermaat_heckman_thm}). $(M, \omega, \Phi)$ is tight, and therefore the action is non-Hamiltonian (see Remark~\ref{tight_means_non_hamiltonian}). Moreover, By assumption, $M$ is four-dimensional. Hence, by Lemma~\ref{no_fixed_points}, there are no fixed points, and therefore the derivative of the Duistermaat-Heckman function is constant over $U$. But this is true for every point $y$ in $\mathbb{R}/P$, and therefore the derivative of the Duistermaat-Heckman function is constant all over $\mathbb{R}/P$, so the function is monotone. Because $\mathbb{R}/P$ is a circle, the Duistermaat-Heckman function must be constant. Its derivative is therefore zero, and so the Seifert Euler number is zero.
\end{proof}
Since the Duistermaat-Heckman function is constant, it can be described by a unique positive number. We call this number the \textbf{Duistermaat-Heckman constant} of the space and denote it by~$c_{\DuHe}$. Note that it is invariant under equivariant symplectomorphisms.


\section{Moser's trick and \texorpdfstring{$\Phi$}{Phi}-\texorpdfstring{$T$}{T}-diffeomorphisms}\label{moser_section}

In this section, we apply an adaptation of Moser's trick from~\cite{centered_hamiltonians} to isotope $\Phi$-$T$-diffeomorphisms of circle-valued Hamiltonian $S^1$-manifolds, to equivariant symplectomorphisms. In short, a $\Phi$-$T$-diffeomorphism is an orientation-preserving equivariant diffeomorphism that respects the Hamiltonians (see Definition~\ref{phi_t_diffeo}). We give a cohomological condition for when a $\Phi$-$T$-diffeomorphism can be isotoped to an isomorphism of circle-valued Hamiltonian $S^1$-manifolds (see Lemma~\ref{moser_with_same_cohomology}).

We deduce that two compact connected four-manifolds with effective symplectic non-Hamiltonian circle actions that admit circle-valued Hamiltonians, are equivariantly symplectomorphic if and only if they are isomorphic as tight circle-valued Hamiltonian $S^1$-manifolds (see Corollary~\ref{classifying_circle_valued_spaces_instead}). This means that the classification of circle-valued Hamiltonian $S^1$-manifolds up to isomorphism, is equivalent to the classification of non-Hamiltonian $S^1$ actions up to equivariant symplectomorphisms.

We also show that if $b_1(M/S^1)=1$, then two spaces are isomorphic if and only if they are $\Phi$-$T$-diffeomorphic and have the same Duistermaat-Heckman constant (see Corollary~\ref{phi_t_diffeo_and_eq_symp_betti_one}). For complexity one spaces, Karshon and Tolman proved a similar proposition, without any assumptions on the Betti numbers (see~\cite[Proposition 3.3]{centered_hamiltonians}). Surprisingly, we show that this is not true for non-Hamiltonian actions without the assumption on the Betti numbers -- We construct $\Phi$-$T$-diffeomorphic spaces with the same Duistermaat-Heckman constant, that are not equivariantly symplectomorphic (see Example~\ref{ex_different_cohomologies}). See also Section~\ref{de_rham_invariant_section}, where we formalize an invariant to determine when two $\Phi$-$T$-diffeomorphic spaces are also equivariantly symplectomorphic.
\\

Let's recall the definition of a $\Phi$-$T$-diffeomorphism (see~\cite[Definition 3.1]{centered_hamiltonians}). We will slightly generalize it, to allow for circle-valued Hamiltonians.
\begin{definition}\label{phi_t_diffeo}
    Let a torus $T$ act on oriented manifolds $M$ and $M'$ with $T$-invariant maps $\Phi:M \rightarrow C$ and $\Phi':M' \rightarrow C$, for some set $C$. A \textbf{$\Phi$-$T$-diffeomorphism} from $(M, \Phi)$ to $(M', \Phi')$ is an orientation preserving equivariant diffeomorphism $F:M \rightarrow M'$ that satisfies $F^* \Phi' = \Phi$.
\end{definition}
In the original definition in~\cite{centered_hamiltonians}, the target of the invariant map was the dual algebra $\mathfrak{t}^*$ of the torus $T$, in accordance to the definition of a momentum map. Since we are working with circle-valued Hamiltonians, whose targets are quotients of $\mathfrak{t}^*$, we generalize the definition to allow for invariant maps with general targets.
We use the name $\Phi$-$T$-diffeomorphism to stay consistent with the works of Karshon and Tolman, even though in our paper the map $\Phi$ is circle-valued (instead of $\mathfrak{t}^*$-valued), and the acting group $T$ is always $S^1$.
\\

In the setting of complexity one spaces, while not stated explicitly, the following two lemmas follow from the proofs of Lemmas 3.3 and 3.6 in~\cite{centered_hamiltonians}. For completeness, we prove them in the setting of circle-valued Hamiltonians.
\begin{lemma}\label{omega_difference_is_basic}
    Let $(M, \omega, \Phi)$ and $(M', \omega', \Phi')$ be four-dimensional tight circle-valued Hamiltonian $S^1$-manifolds that have the same group of periods $P$. Let $F: M\rightarrow M'$ be a $\Phi$-$T$-diffeomorphism. Then the two-form $F^*\omega' - \omega$ on $M$ is a basic form, and in particular $F$ gives a well-defined basic cohomology class ${[F^*\omega' - \omega]}_{\basic}$ in $H^2_{\basic}(M)$.
\end{lemma}
\begin{proof}
    First, $\mathcal{L}_X(F^*\omega' - \omega) = 0$ since the forms are invariant and the map is equivariant. Second, by the assumption that $\Phi' \circ F = \Phi$:
    \begin{equation*}
            \iota_X (F^*\omega' - \omega) = d \Phi - d \Phi = 0.
    \end{equation*}
\end{proof}
\begin{lemma}\label{moser_with_same_cohomology}
    Let $(M, \omega, \Phi)$ and $(M', \omega', \Phi')$ be four-dimensional tight circle-valued Hamiltonian $S^1$-manifolds that have the same group of periods $P$.
    Let $F: M\rightarrow M'$ be a $\Phi$-$T$-diffeomorphism. Assume that the basic cohomology class ${[F^*\omega' - \omega]}_{\basic}$ vanishes. Then $F$ is smoothly isotopic to an equivariant symplectomorphism that satisfies $F^*\Phi' = \Phi$.
\end{lemma}
\begin{proof}
    The forms $\omega$ and $F^*\omega'$ generate the same action with the same circle-valued Hamiltonian, and induce the same orientation. Therefore, by~\cite[Lemma 3.6]{centered_hamiltonians}, the two-form 
    \begin{equation*}
        \omega_t = (1-t)\omega + t F^*\omega'
    \end{equation*}
    is non-degenerate for every $t$. Note that while the mentioned lemma assumes that the Hamiltonian is real-valued, it is also true for circle-valued Hamiltonians since it can be applied to preimages of small intervals in $\mathbb{R}/P$ which are proper Hamiltonian $S^1$-manifolds.
    
    By the assumption that the basic cohomology class ${[F^*\omega' - \omega]}_{\basic}$ vanishes, we have that $d\beta = F^*\omega' - \omega$ for some basic one-form $\beta$.

    We now use Moser's trick. By non-degeneracy, there exists a unique vector field $Z_t$ such that $\iota_{Z_t} \omega_t = -\beta$. The flow $\Psi_t$ generated by $Z_t$ is $S^1$-invariant because $\omega_t$ and $\beta$ are $S^1$-invariant, and it preserves $\Phi$ since $d \Phi(Z_t) = \omega_t(X, Z_t) = -\iota_X \beta = 0$.
    
    Taking the derivative of $\Psi_t^*\omega_t$ by $t$, we get:
    \begin{align*}
        \frac{d}{dt}(\Psi_t^*\omega_t) &= \Psi_t^*(\mathcal{L}_{Z_t}\omega_t + {\dot \omega}_t) = \\
        &= \Psi_t^*(d\iota_{Z_t}\omega_t + F^*\omega' - \omega) \\
        &= \Psi_t^*(-d\beta + F^*\omega' - \omega) = 0,
    \end{align*}
    where in the second equality we use Cartan's formula and $\omega_t$ being closed. We deduce that $\Psi_t^*\omega_t = \omega$ for every $0 \le t \le 1$, and therefore that $\omega = \Psi_1^*\omega_1 = \Psi_1^*(F^* \omega')$, so $F \circ \Psi_1$ is an equivariant symplectomorphism such that ${(F \circ \Psi_1)}^*\Phi' = \Phi$.
\end{proof}
Now, we prove the first key result of this section: the isomorphism class of a four-dimensional tight circle-valued Hamiltonian $S^1$-manifold is independent of the choice of the tight circle-valued Hamiltonian. 
\begin{proposition}\label{invariance_of_translations}
    Let $(M, \omega, \Phi)$ be a four-dimensional tight circle-valued Hamiltonian $S^1$-manifold. Let $c \in \mathbb{R}/P$ be a constant.
    Then there exists an equivariant symplectomorphism $\Psi:M \rightarrow M$ which is smoothly isotopic to the identity map, that satisfies $\Psi^* \Phi = \Phi + c$.
\end{proposition}
\begin{proof}
    We start as in the proof of Lemma~\ref{same_genus_and_isotropy_data}.
    Take an equivariant compatible metric $g$ and a corresponding almost complex structure $J$. The gradient vector field $Y := -JX$ is non zero on all of $M$ because there are no fixed points by Lemma~\ref{no_fixed_points}. Therefore, we can normalize $Y$ such that $d \Phi(Y) = 1$, and take the time-$s$ flow map $F_s$, where $s$ is in the coset $c + P \subset \mathbb{R}$. Then $\Phi \circ F_s = \Phi + c$. Moreover, $F_s$ is an equivariant orientation-preserving diffeomorphism, i.e., a $\Phi$-$T$-diffeomorphism between $(M, \omega, \Phi)$ and $(M, \omega, \Phi - c)$. Furthermore, the basic cohomology class ${[F_s^*\omega - \omega]}_{\basic}$ vanishes because $F_s$ is smoothly isotopic to the identity map through equivariant diffeomorphisms. Therefore, by Lemma~\ref{moser_with_same_cohomology}, $F$ can be isotoped to an equivariant symplectomorphism $\Psi$ that satisfy $\Psi^*(\Phi - c) = \Phi$, or equivalently $\Psi^*\Phi = \Phi + c$.
\end{proof}
\begin{corollary}\label{classifying_circle_valued_spaces_instead}
    Let $(M, \omega)$ and $(M', \omega')$ be two four-dimensional connected compact symplectic manifolds with effective symplectic non-Hamiltonians $S^1$ actions, generated by the vector fields $X$ and $X'$. Assume that the groups of periods of $\iota_X\omega$, $\iota_X'\omega'$ are discrete. Let $\Phi$ and $\Phi'$ be tight circle-valued Hamiltonians for the spaces.
    
    Then $(M, \omega)$ and $(M', \omega')$ are equivariantly symplectomorphic if and only if $(M, \omega, \Phi)$ and $(M', \omega', \Phi')$ are isomorphic as tight circle-valued Hamiltonian $S^1$-manifolds.
\end{corollary}
\begin{remark}
Proposition~\ref{invariance_of_translations} and Corollary~\ref{classifying_circle_valued_spaces_instead} are no longer true if one omits the assumption on the dimension of the manifold. When the dimension is larger than $4$, there exist non-Hamiltonian circle actions with fixed points (see~\cite{symplectic_circle_actions} and~\cite{tolman_isolated_points}). In this case, not all level sets are equivariantly diffeomorphic, hence it is impossible to find an equivariant symplectomorphisms that shift the circle-valued Hamiltonian by arbitrary translations.
\end{remark}

Next, we introduce a technical condition for tight circle-valued Hamiltonian $S^1$-manifolds:
\begin{equation}\label{condition}
    \begin{aligned}
    & \text{The restriction map } H^2(M/S^1, \mathbb{Z}) \rightarrow H^2(\Phi^{-1}(y)/S^1, \mathbb{Z}) \text{ is } \\
    & \text{one-to-one for some regular value }y\text{ of }\Phi:M \rightarrow \mathbb{R}/P\text{.}
    \end{aligned}
\end{equation}
For complexity one spaces, there is a similar condition with $S^1$ replaced by $T$, and the circle-valued Hamiltonian $\Phi$ replaced by a momentum map. Karshon and Tolman showed that this condition is always satisfied for tall complexity one spaces. Moreover, they showed that two spaces that satisfy this condition, and have the same Duistermaat-Heckman measure, are isomorphic if and only if they are $\Phi$-$T$-diffeomorphic (see~\cite[Proposition 3.3]{centered_hamiltonians}). For tight circle-valued Hamiltonian $S^1$-manifolds, the analogous proposition is:
\begin{proposition}\label{phi_t_diffeo_and_eq_symp}
Let $(M, \omega, \Phi)$ and $(M', \omega', \Phi')$ be four-dimensional tight circle-valued Hamiltonian $S^1$-manifolds that satisfy Condition~\eqref{condition} and have the same Duistermaat-Heckman constant and the same group of periods $P$.
Then every $\Phi$-$T$-diffeomorphism from $M$ to $M'$ can be smoothly isotoped to an isomorphism of $(M, \omega, \Phi)$ and $(M', \omega', \Phi')$.
\end{proposition}
Before proving Proposition~\ref{phi_t_diffeo_and_eq_symp}, we will show the following lemma, which is an adaptation of Lemma 3.5 in~\cite{centered_hamiltonians}:
\begin{lemma}\label{condition_to_cohomology}
    Let $(M, \omega, \Phi)$ and $(M', \omega', \Phi')$ be four-dimensional tight circle-valued Hamiltonian $S^1$-manifolds that satisfy Condition~\eqref{condition}. Assume that they have the same group of periods $P$ and the same Duistermaat-Heckman constant.
    
    Then for every $\Phi$-$T$-diffeomorphism $F:M \rightarrow M'$, the basic cohomology class ${[F^*\omega' - \omega]}_{\basic}$ vanishes.
\end{lemma}
\begin{proof}
    Let $F:M \rightarrow M'$ be a $\Phi$-$T$-diffeomorphism. As shown in Lemma~\ref{omega_difference_is_basic}, the two-form $F^*\omega' - \omega$ is basic, and therefore it is the pullback of some two-form $\Omega$ on the orbifold $M/S^1$. Moreover, the class $[\Omega]$ in $H^2(M/S^1, \mathbb{R})$ vanishes if and only if the basic cohomology class ${[F^*\omega' - \omega]}_{\basic}$ vanishes. By Condition~\eqref{condition}, the restriction map $H^2(M/S^1, \mathbb{Z}) \rightarrow H^2(\Phi^{-1}(y)/S^1, \mathbb{Z})$ is one-to-one for some $y$. It follows that it's enough to show that the class $[\Omega|_{\Phi^{-1}(y)/S^1}]$ vanishes in $H^2(\Phi^{-1}(y)/S^1, \mathbb{R})$, to deduce that $[\Omega]$ vanishes and subsequently that ${[F^*\omega' - \omega]}_{\basic}$ vanishes.
    
    The integral of $\Omega|_{\Phi^{-1}(y)/S^1}$ over the reduced space $\Phi^{-1}(y)/S^1$ is $c_{\DuHe}'~-~c_{\DuHe}$, where $c_{\DuHe}$ and $c_{\DuHe}'$ are the Duistermaat-Heckman constants of $(M, \omega, \Phi)$ and $(M', \omega', \Phi')$, respectively. By the assumption that the Duistermaat-Heckman constants are the same, this integral is zero, hence $[\Omega|_{\Phi^{-1}(y)/S^1}]$ vanishes, and therefore the basic cohomology class ${[F^*\omega' - \omega]}_{\basic}$ vanishes.
\end{proof}

\begin{proof}[Proof of Proposition~\ref{phi_t_diffeo_and_eq_symp}]
    Let $F:M \rightarrow M'$ be a $\Phi$-$T$-diffeomorphism. By assumption, both spaces have the same Duistermaat-Heckman constant, and satisfy Condition~\eqref{condition}. Thus, by Lemma~\ref{condition_to_cohomology}, the basic cohomology class ${[F^*\omega' - \omega]}_{\basic}$ vanishes. Hence, by Lemma~\ref{moser_with_same_cohomology}, $F$ is smoothly isotopic to an equivariant symplectomorphism $\Psi$ that satisfy $\Psi^*\Phi' = \Phi$, i.e., an isomorphism of circle-valued Hamiltonians $S^1$-manifolds.
\end{proof}
An obstacle to applying Proposition~\ref{phi_t_diffeo_and_eq_symp} is that Condition~\eqref{condition} isn't always satisfied for four-dimensional tight circle-valued Hamiltonian $S^1$-manifolds. The following example shows that the Condition~\eqref{condition} is not satisfied for the circle action that rotates the first coordinate of the $4$-torus.
\begin{example}\label{ex_condition_not_satisfied}
    Let $S^1$ act on the four-dimensional torus $\mathbb{T}^4 \cong \mathbb{R}^4/\mathbb{Z}^4$, supplied with the standard symplectic form $\omega:=dp_1\wedge dq_1 + dp_2\wedge dq_2$, by rotating the coordinate $q_1$. Let $X = \frac{1}{2\pi}\frac{\partial}{\partial q_1}$ be the generating vector field of the action. The action is symplectic, and the group of periods of $\iota_X\omega = -\frac{1}{2\pi}dp_1$ is $\frac{1}{2\pi}\mathbb{Z}$. The action is generated by the tight circle-valued Hamiltonian $\Phi(p_1, q_1, p_2, q_2) = \frac{1}{2\pi}p_1$. Hence, the triple $(\mathbb{T}^4, \omega, \Phi)$ is a four-dimensional tight circle-valued Hamiltonian $S^1$-manifold. Since $H^2(\mathbb{T}^3, \mathbb{Z}) \cong \mathbb{Z}^3$ and $H^2(\mathbb{T}^3/S^1 \cong \mathbb{T}^2, \mathbb{Z}) \cong \mathbb{Z}$, the restriction map cannot be one-to-one, and Condition~\eqref{condition} is not satisfied.    
\end{example}
  The next example shows that the assumption on Condition~\eqref{condition} cannot be omitted from Proposition~\ref{phi_t_diffeo_and_eq_symp}. More precisely, the example shows that there exist $\Phi$-$T$-diffeomorphic four-dimensional tight circle-valued Hamiltonian $S^1$-manifolds, with the same Duistermaat-Heckman constants, that are not isomorphic. This is a striking difference between our spaces and tall complexity one spaces; the latter always satisfy the (analogous) condition, and so $\Phi$-$T$-diffeomorphisms of tall complexity one spaces can always be upgraded to isomorphisms, provided that they have the same Duistermaat-Heckman measure.
\begin{example}\label{ex_different_cohomologies}
     Let $\omega_{A,B}$ be a family of symplectic forms on the four-dimensional torus $\mathbb{T}^4 \cong \mathbb{R}^4/\mathbb{Z}^4$ defined by
     \begin{equation*}
         \omega_{A,B} := dp_1\wedge dq_1 + dp_2\wedge dq_2 + Adp_1\wedge dq_2 + Bdp_1\wedge dp_2.
     \end{equation*}
     Then for every $A,B \in \mathbb{R}$, the circle action that rotates the coordinate $q_1$ is symplectic with respect to $\omega_{A,B}$, and the group of periods of $\iota_{\frac{1}{2\pi}\frac{\partial}{\partial q_1}}\omega_{A,B} = -\frac{1}{2\pi}dp_1$ is $\frac{1}{2\pi}\mathbb{Z}$. Therefore, $(\mathbb{T}^4, \omega_{A,B}, \Phi = \frac{1}{2\pi}p_1)$ is a family of tight circle-valued Hamiltonian $S^1$-manifolds, parameterized by $A$ and $B$. Every two spaces in this family are $\Phi$-$T$-diffeomorphic by the identity map $\id:\mathbb{T}^4 \rightarrow \mathbb{T}^4$. Furthermore, the Duistermaat Heckman constant of these spaces is $2\pi$, regardless of $A$ and $B$.
     \\
     For every $A, B \in \mathbb{R}$, define the evaluation map $\evaluation_{A,B} : H_2(M, \mathbb{Z}) \rightarrow \mathbb{R}$ that sends each homology class $\alpha \in H_2(M, \mathbb{Z})$ to the integral $\int_\alpha \omega_{A, B}$. Then if $(\mathbb{T}^4, \omega_{A,B})$ is symplectomorphic to $(\mathbb{T}^4, \omega_{A',B'})$ then the maps $\evaluation_{A,B}$, $\evaluation_{A', B'}$ must have the same image, but this is only the case when $A\mathbb{Z} + B\mathbb{Z} + \mathbb{Z} = A'\mathbb{Z} + B'\mathbb{Z} + \mathbb{Z}$, and therefore there are infinitely many pairs of non symplectomorphic spaces in this family.
\end{example}

The following proposition shows that Condition~\eqref{condition} is satisfied precisely when the first (or second) Betti number of the quotient space is $1$:
\begin{proposition}\label{condition_for_betti_one}
    Let $(M, \omega, \Phi)$ be a four-dimensional tight circle-valued Hamiltonian $S^1$-manifold. Then Condition~\eqref{condition} is satisfied if and only if $b_1(M/S^1) = 1$ if and only if $b_2(M/S^1) = 1$.
\end{proposition}
\begin{proof}
    By Proposition~\ref{topological_model_proposition}, the topological bundle $\overline{\Phi}: M/S^1 \rightarrow \mathbb{R}/P$ is isomorphic to the mapping torus $\pi: M_f \rightarrow \mathbb{R}/P$, for some orientation-preserving homeomorphism $f:\Sigma_g \rightarrow \Sigma_g$. Let $U, V$ be two open intervals that cover $\mathbb{R}/P$. By applying the Mayer Vietoris sequence for the cover $\pi^{-1}(U)$ and $\pi^{-1}(V)$, we get the short exact sequence
    \begin{equation} \label{long_exact_sequence_equation}
        0 \rightarrow H^{n-1}(\Sigma_g) / \image(f_{n-1} - \id) \rightarrow H^n(M_f) \rightarrow \ker(f_n - \id) \rightarrow 0,
    \end{equation}
    where $f_k:H^k(\Sigma_g) \rightarrow H^k(\Sigma_g)$ is the induced map on cohomology. Setting $n=1$ in the short exact sequence, the group $H^0(\Sigma_g) / \image(f_0 - \id)$ is isomorphic to $\mathbb{Z}$, because $f_0$ is the identity map. Hence, we have
    \begin{equation}\label{eq1_mayer}
        H^1(M_f) \cong \mathbb{Z} \Longleftrightarrow \ker(f_1 - \id) = 0.
    \end{equation}
    Setting $n=2$, the group $\ker(f_2 - \id)$ is isomorphic to $\mathbb{Z}$, because $f_2$ is the identity map. Therefore, we have:
    \begin{equation}\label{eq2_mayer}
        H^2(M_f) \cong \mathbb{Z} \Longleftrightarrow H^1(\Sigma_g)/\image(f_1-\id) = 0 \Longleftrightarrow \ker(f_1 - \id) = 0.
    \end{equation}
    
    Now, let $y$ be in $\mathbb{R}/P$. The kernel of the restriction map $H^2(M/S^1) \rightarrow H^2(\Phi^{-1}(y))$ is isomorphic to the term $H^1(\Sigma_g)/\image(f_1 - \id)$ in the short exact sequence (for $n=2$). Therefore, the restriction map is injective if and only if $H^1(\Sigma_g)/\image(f_1 - \id)$ is trivial. By~\eqref{eq2_mayer}, this is true if and only if $b_2(M/S^1) = b_2(M_f) = 0$ if and only if $\ker(f_1 - \id) = 0$. By~\eqref{eq1_mayer}, this is true if and only if $b_1(M/S^1) = b_1(M_f) = 1$.
\end{proof}
\begin{corollary}\label{phi_t_diffeo_and_eq_symp_betti_one}
    Let $(M, \omega, \Phi)$ and $(M', \omega', \Phi')$ be four-dimensional tight circle-valued Hamiltonian $S^1$-manifolds. Assume that they have the same group of periods, the same Duistermaat-Heckman constant, and that $b_1(M/S^1) = b_1(M'/S^1) = 1$. Then they are isomorphic if and only if they are $\Phi$-$T$-diffeomorphic.
\end{corollary}
\begin{proof}
    In one direction, every isomorphism between $(M, \omega, \Phi)$ and $(M', \omega', \Phi')$ is in particular a $\Phi$-$T$-diffeomorphism. For the other direction, let $F:M \rightarrow M'$ be a $\Phi$-$T$-diffeomorphism. By Proposition~\ref{condition_for_betti_one}, Both $(M, \omega, \Phi)$ and $(M', \omega', \Phi')$ satisfy Condition~\eqref{condition}. Therefore, by Proposition~\ref{phi_t_diffeo_and_eq_symp}, $F$ can be isotoped to an isomorphism of $(M, \omega, \Phi)$ and $(M', \omega', \Phi')$.
\end{proof}


\section{From \texorpdfstring{$\Phi$}{Phi}-\texorpdfstring{$T$}{T}-diffeomorphisms to \texorpdfstring{$\Phi$}{Phi}-diffeomorphisms} \label{phi_diffeo_section}
In this section, we describe when a diffeomorphism between the quotients of four-dimensional tight circle-valued Hamiltonian $S^1$-manifolds can be lifted to a $\Phi$-$T$-diffeomorphism. For diffeomorphisms between the quotients, we use the notion of $\Phi$-diffeomorphisms introduced in~\cite{centered_hamiltonians} (see Definition~\ref{phi_diffeo}).

Like in the previous section, we show that if $b_1(M/S^1) = 1$, then every $\Phi$-diffeomorphism can be lifted to a $\Phi$-$T$-diffeomorphism (see Corollary~\ref{phi_diffeo_and_phi_t_diffeo_betti_one}). For complexity one spaces, Karshon and Tolman proved a similar proposition, without any assumptions on the Betti numbers (see~\cite[Proposition 4.2]{centered_hamiltonians}). In our case, such a general claim is not true: in Appendix~\ref{examples_appendix}, we describe spaces which have $\Phi$-diffeomorphic quotients but are not $\Phi$-$T$-diffeomorphic. See also Section~\ref{fibration_invariant_section}, where we formalize an invariant to determine when two spaces with $\Phi$-diffeomorphic quotients are also $\Phi$-$T$-diffeomorphic.
\\

We start by introducing the definition of a $\Phi$-diffeomorphism. As in the previous section, we slightly generalize the definition to allow general targets for the $T$-invariant maps, instead of only allowing $\mathfrak{t}^*$, so we can support the more general setting of circle-valued Hamiltonians.
\begin{definition}\label{phi_diffeo}
    Let a torus $T$ act on oriented manifolds $M$ and $M'$ with $T$-invariant maps $\Phi:M \rightarrow C$ and $\Phi':M' \rightarrow C$, for some set $C$. A \textbf{$\Phi$-diffeomorphism} from $M/T$ to $M'/T$ is an orientation preserving diffeomorphism $\overline{F}:M/T\rightarrow M'/T$ such that
    \begin{enumerate}
        \item $\overline{F}$ preserves the maps to $C$, i.e., ${\overline{F}}^* \Phi' = \Phi$.
        \item Each of $\overline{F}$ and ${\overline{F}}^{-1}$ lifts to a $\Phi$-$T$-diffeomorphism in a neighborhood of each exceptional orbit.
    \end{enumerate}
\end{definition}
Note that Definition~\ref{phi_diffeo} mentions a diffeomorphism between the quotient spaces, which are not necessarily smooth manifolds. In this context, a real-valued function on $M/T$ is called smooth if its pullback to $M$ is smooth, and a map between the quotients $M/T$ and $M'/T$ is called smooth if it pull backs every smooth function to a smooth function. See the discussion before Definition 4.1 in~\cite{centered_hamiltonians} for more details.

Let $(M, \omega, \Phi)$ be a four-dimensional tight circle-valued Hamiltonian $S^1$-manifold, and let $\pi: M \rightarrow M/S^1$ be the quotient map. By~\cite[Proposition 4.2]{haefliger_salem}, the isomorphism classes of orbibundles over $M/S^1$ which are locally isomorphic to $\pi:M \rightarrow M/S^1$ are in one-to-one correspondence with $H^2(M/S^1, \mathbb{Z})$. This gives a criterion for when a $\Phi$-diffeomorphism can be lifted to a $\Phi$-$T$-diffeomorphism:
\begin{lemma}\label{lifting_phi_diffeo_condition}
    Let $(M, \omega, \Phi)$ and $(M', \omega', \Phi')$ be four-dimensional tight circle-valued Hamiltonian $S^1$-manifolds, and $\overline{F}: M/S^1 \rightarrow M'/S^1$ be a $\Phi$-diffeomorphism. Moreover, let $\alpha \in H^2(M/S^1, \mathbb{Z})$ be the cohomology class that corresponds to the pullback bundle $\overline{F}^*M' \rightarrow M/S^1$ by Haefliger-Salem. Then $\overline{F}$ can be lifted to a $\Phi$-$T$-diffeomorphism if and only if $\alpha = 0$.
\end{lemma}
See also Lemma~\ref{lifts_by_fintushel} for an equivalent result in the language of Fintushel's papers~\cite{fintushel,fintushel2}. Karshon and Tolman used~\cite{haefliger_salem} to prove that if two complexity one spaces satisfy Condition~\eqref{condition}, and have the same Duistermaat-Heckman measure, then every $\Phi$-diffeomorphism between their quotients lift to a $\Phi$-$T$-diffeomorphism (see~\cite[Proposition 4.2]{centered_hamiltonians}). For tight circle-valued Hamiltonian $S^1$-manifolds, the analogous proposition is:
\begin{proposition}\label{phi_diffeo_and_phi_t_diffeo}
Let $(M, \omega, \Phi)$ and $(M', \omega', \Phi')$ be four-dimensional tight circle-valued Hamiltonian $S^1$-manifolds that satisfy Condition~\eqref{condition}, and have the same group of periods.
Then every $\Phi$-diffeomorphism from $M/S^1$ to $M'/S^1$ lifts to a $\Phi$-$T$-diffeomorphism from $M$ to $M'$.
\end{proposition}
\begin{proof}
    The proof given by Karshon and Tolman for the case of complexity one spaces (see~\cite[Proposition 4.2]{centered_hamiltonians}) can be restated without any significant modifications to prove Proposition~\ref{phi_diffeo_and_phi_t_diffeo}. In short, by Lemma~\ref{lifting_phi_diffeo_condition}, and Condition~\eqref{condition}, it is enough to show that the Chern classes of restrictions to the reduced spaces of $M$ and $M'$ agree. This follows since their Duistermaat-Heckman derivatives, that correspond to the Chern classes, are always zero by Lemma~\ref{duistermaat_heckman_constant}.
\end{proof}
Note that for complexity one spaces,~\cite[Proposition 4.2]{centered_hamiltonians} had an additional assumption that the Duistermaat-Heckman measures of the spaces agree, which is not needed here since the Duistermaat-Heckman's derivatives always agree.

\begin{corollary}\label{phi_diffeo_and_phi_t_diffeo_betti_one}
Let $(M, \omega, \Phi)$ and $(M', \omega', \Phi')$ be four-dimensional tight circle-valued Hamiltonian $S^1$-manifolds. Assume that they have the same group of periods, and that $b_1(M/S^1) = b_1(M'/S^1) = 1$. Then they are $\Phi$-$T$-diffeomorphic if and only if their quotients are $\Phi$-diffeomorphic.
\end{corollary}
\begin{proof}
    By Lemma~\ref{condition_for_betti_one}, Both $(M, \omega, \Phi)$ and $(M', \omega', \Phi')$ satisfy Condition~\eqref{condition}. Thus, the Corollary is immediate by Proposition~\ref{phi_diffeo_and_phi_t_diffeo}.
\end{proof}


\section{From \texorpdfstring{$\Phi$}{Phi}-diffeomorphisms to painted surface bundles} \label{sheaves_section}
In this section, we reduce the problem of determining whether two four-dimensional tight circle-valued Hamiltonian $S^1$-manifolds have $\Phi$-diffeomorphic quotients to a problem in differential topology. We define for each four-dimensional tight circle-valued Hamiltonian $S^1$-manifold a unique up to isomorphism painted surface bundle over $\mathbb{R}/P$, in the same sense as in~\cite{tall_uniqueness}. Intuitively, an associated painted surface bundle of a tight circle-valued Hamiltonian $S^1$-manifold is a ``smooth version'' of the (singular) bundle $M/S^1 \rightarrow \mathbb{R}/P$, together with ``paint'' that labels the exceptional orbits in $M/S^1$. We show that that two spaces have $\Phi$-diffeomorphic quotients if and only if their associated painted surface bundles are isomorphic.

Throughout the section, the approach that we use is based on Karshon and Tolman's proof of a similar claim for proper tall complexity one spaces, see parts II, III and IV of~\cite{tall_uniqueness} for more details. While some of the definitions in~\cite{tall_uniqueness} can be used without change, others must be adapted to the setting of circle-valued Hamiltonians.
\\

In the following paragraphs, we give a brief introduction to sheaves of groupoids. We refer the reader to~\cite[Part II]{tall_uniqueness} for a more detailed treatment of this topic.

A \textbf{presheaf of groupoids} $F$ over a topological space $X$ consists of a groupoid $F(U)$ for every open subset $U \subset X$, and a homomorphism $\iota_U^V:F(U) \rightarrow F(V)$ for every open subset $V \subset U$, such that the following two axioms are satisfied:
\begin{enumerate}
    \item For every open set $U$, the map $\iota_U^U$ is the identity map $\id_{F(U)}$.
    \item If $U, V, W$ are open sets satisfying $W \subset V \subset U$, then $\iota_U^V \circ \iota_V^W = \iota_U^W$.
\end{enumerate}
If $A$ is an object in $F(U)$, and $V$ is a subset of $U$, we denote $\iota_U^V(A)$ by $A|_V$ for simplicity. Similarly, if $f:A \rightarrow A'$ is an arrow in $F(U)$, we denote $\iota_U^V(f)$ by $f|_V$. Objects in the groupoid $F(X)$ are called \textbf{global objects}.

A presheaf of groupoids $F$ is called a \textbf{sheaf of groupoids} if for every collection $W_{\alpha}$ of open subsets of $X$, and objects $A, A'$ in the groupoid $F(\bigcup W_{\alpha})$, the following two axioms are satisfied:
\begin{enumerate}
    \item Let $f:A \rightarrow A'$ and $g:A \rightarrow A'$ be arrows that satisfy $f|_{W_\alpha} = g|_{W_\alpha}$ for every $W_{\alpha}$. Then~$f=g$.
    \item Let $f_{\alpha}:A|_{W_{\alpha}} \rightarrow A'|_{W_{\alpha}}$ be a collection of arrows that agree on intersections, i.e., that satisfy $f_{\alpha_1}|_{W_{\alpha_1} \cap W_{\alpha_2}} = f_{\alpha_2}|_{W_{\alpha_1} \cap W_{\alpha_2}}$ for every $W_{\alpha_1}$ and $W_{\alpha_2}$. Then there exists an arrow $g:A \rightarrow A'$ such that $g|_{W_{\alpha}} = f_{\alpha}$ for every $W_{\alpha}$.
\end{enumerate}
We say that a sheaf of groupoids $F$ has \textbf{gluable objects}, if for every collection $W_{\alpha}$ of open sets in $X$, objects $A_{\alpha}$ in every groupoid $F(W_{\alpha})$, and arrows $g_{\beta, \alpha}:A_{\alpha}|_{W_{\alpha} \cap W_{\beta}} \rightarrow A_{\beta}|_{W_{\alpha} \cap W_{\beta}}$ for every pair $W_{\alpha}$ and $W_{\beta}$, satisfying that $g_{\gamma, \beta} \circ g_{\beta, \alpha} = g_{\gamma, \alpha}$ on $W_{\alpha} \cap W_{\beta} \cap W_{\gamma}$ for every $\alpha, \beta, \gamma$, then there exists an object $B$ in the groupoid $F(\bigcup W_{\alpha})$, and arrows $f_{\alpha}:A_{\alpha} \rightarrow B|_{W_{\alpha}}$, such that $f_{\beta}|_{W_{\alpha} \cap W_{\beta}} \circ g_{\beta, \alpha} \circ f_{\alpha}|_{W_{\alpha} \cap W_{\beta}}^{-1}$ is the identity map $\id_{B|_{{W_{\alpha} \cap W_{\beta}}}}$ in $F({W_{\alpha} \cap W_{\beta}})$.

Let $F$ be a sheaf of groupoids over a topological space $X$, and let $\mathfrak{U}$ be an open cover of $X$. A \textbf{zero cochain} $a \in C^0(\mathfrak{U}, F)$ is a choice of arrow $a_U$ in $F(U)$ for every open set $U$ in the cover $\mathfrak{U}$. A \textbf{one cochain} $b \in C^1(\mathfrak{U}, F)$ is a choice of object $O_U$ in $F(U)$ for every open set $U$ in the cover $\mathfrak{U}$, together with a choice of arrow $b_{V,W}:O_W|_{V \cap W} \rightarrow O_V|_{V \cap W}$ in $F(V \cap W)$ for every $V$ and $W$ in the cover $\mathfrak{U}$. A one cochain $b$ is called a \textbf{one cocycle} if it is \textbf{closed}, i.e., if it satisfies the following two properties:
\begin{enumerate}
    \item For every $U$ in $\mathfrak{U}$, the arrow $b_{U,U}$ is the identity arrow $\id_{U}$ in $F(U)$.
    \item For every $W \subset V \subset U$ in $\mathfrak{U}$, the arrows $b_{V, U}$, $b_{W, V}$ and $b_{W, U}$ satisfy $b_{W, V} \circ b_{V, U} = b_{W, U}$.
\end{enumerate}
The group of zero cochains act on the set of one cocycles by conjugation of the arrows. More explicitly, if $a$ is a zero cochain, and $b$ is a one cocycle, then for every $U$ and $V$ in $\mathfrak{U}$, the arrow $(a \cdot b)_{V, U}$ is given by
\begin{equation*}
    a_V|_{U \cap V} \circ b_{V, U} \circ (a_U|_{U \cap V})^{-1}.
\end{equation*}
The \textbf{first cohomology} $H^1(F, \mathfrak{U})$ is defined as the set of equivalence classes of one cocycles under this action. The first cohomologies defined with different covers, and the refinement maps, define a direct system. The \textbf{first Čech cohomology} $\check H^1(X, F)$ is defined to be the direct limit of this system.

The following two lemmas will be useful for dealing with isomorphism classes of spaces:
\begin{lemma}[Lemma 5.4 in~\cite{tall_uniqueness}] \label{abstract_nonsense_lemma}
    Let $F$ be a sheaf of groupoids over $X$. A global object in $F$ naturally determines a class $[A] \in \check H^1(X , F)$. Two objects $A$ and $A'$ are isomorphic if and only if $[A] = [A']$. If the sheaf has gluable objects, every class in $\check H^1(X , F)$ arises in this way.
\end{lemma}
\begin{lemma}[Lemma 6.4 in~\cite{tall_uniqueness}]\label{other_abstract_nonsense_lemma}
    Let $F$ and $F'$ be sheaves of groupoids over $X$, and let $i : F \rightarrow F'$ be a map of sheaves such that:
    \begin{enumerate}
        \item For any open subset $U \subset X$ and objects $A, A' \in F(U)$, the map $i : \hom_F(A, A') \rightarrow hom_{F'}(i(A), i(A'))$ is a bijection.
        \item For any open subset $U \subset X$ and object $B \in F'(U)$, every point in $U$ has a neighborhood $V \subset U$ and an object $A \in F(V)$ so that $i(A)$ is isomorphic to $B|_V$.
    \end{enumerate}
    Then $i$ induces an isomorphism $i_* : \check H^1(X , F) \rightarrow \check H^1(X , F')$.
\end{lemma}

Let $P$ be a non-trivial discrete subgroup of $\mathbb{R}$. We wish to define a sheaf of groupoids $\mathcal{Q}_{\mathbb{R}/P}$ over $\mathbb{R}/P$, such that the objects of each groupoid $\mathcal{Q}_{\mathbb{R}/P}(U)$ will model restrictions of four-dimensional tight circle-valued Hamiltonian $S^1$-manifolds to $U$, and the arrows are $\Phi$-diffeomorphisms between their quotients. We first introduce an intrinsic definition for the restrictions of four-dimensional tight circle-valued Hamiltonian $S^1$-manifolds to preimages of intervals in $\mathbb{R}/P$. Note that for tall proper complexity one spaces, there was no need for such a definition since the restriction of a tall proper complexity one space to a preimage of an open convex set is again a tall proper complexity one space.
\\
\begin{definition} \label{piece_definition}
    A \textbf{four-dimensional circle-valued Hamiltonian $S^1$-piece} over an open subset $U$ of $\mathbb{R}/P$ is a four-dimensional symplectic manifold $(M, \omega)$ with a symplectic circle action, effective on every connected component, that is generated by a circle-valued Hamiltonian $\Phi:M \rightarrow \mathbb{R}/P$, such that the following properties are satisfied:
\begin{enumerate}
    \item the image of $\Phi$ is $U$,
    \item $\Phi$ is proper as a map to $U$,
    \item the circle action has no fixed points,
    \item for every $y \in U$, the level set $\Phi^{-1}(y)$ is connected, and
    \item for every $y \in U$, the Seifert Euler number of the Seifert fibration $\Phi^{-1}(y) \rightarrow \Phi^{-1}(y)/S^1$ is $0$.
\end{enumerate}
\end{definition}

We now define the sheaf $\mathcal{Q}_{\mathbb{R}/P}$ over $\mathbb{R}/P$. For every $U\subset \mathbb{R}/P$, the objects of $\mathcal{Q}_{\mathbb{R}/P}(U)$ are four-dimensional circle-valued Hamiltonian $S^1$-pieces over $U$, and the arrows are $\Phi$-diffeomorphisms between their quotients.
\\
\begin{claim}\label{global_objects_of_q_claim}
    The global objects of $\mathcal{Q}_{\mathbb{R}/P}$ correspond to four-dimensional tight circle-valued Hamiltonian $S^1$-manifolds with group of periods $P$.
\end{claim}
\begin{proof}
    Let $(M, \omega, \Phi)$ be a four-dimensional tight circle-valued Hamiltonian $S^1$-manifold with group of periods $P$. The map $\Phi$ is surjective and proper as a map to $\mathbb{R}/P$. By Lemma~\ref{no_fixed_points}, $(M, \omega, \Phi)$ has no fixed points. By Lemma~\ref{fiber_connectivity_of_circle_actions} its level sets are connected. By Lemma~\ref{duistermaat_heckman_constant}, the Seifert Euler number of each level set is $0$. Therefore, $(M, \omega, \Phi)$ is indeed a four-dimensional circle-valued Hamiltonian $S^1$-piece over $\mathbb{R}/P$, i.e., a global object of $\mathcal{Q}_{\mathbb{R}/P}$.
    
    Conversely, let $(M, \omega)$ be a four-dimensional circle-valued Hamiltonian $S^1$-piece over $\mathbb{R}/P$, i.e., a global object of $\mathcal{Q}_{\mathbb{R}/P}$. It is a symplectic four-manifold with an effective $S^1$ symplectic action, generated by a map $\Phi:M \rightarrow \mathbb{R}/P$. By definition, $\Phi$ is proper, hence $M$ is compact. By definition, the action has no fixed points, hence it is a non-Hamiltonian action. The group of periods $P'$ of the one-form $\iota_X \omega$ is a subgroup of $P$. We can lift the map $\Phi$ to a map $\overline{\Phi}: M \rightarrow \mathbb{R}/{P'}$ which is tight, and thus by~\ref{fiber_connectivity_of_circle_actions} the level sets of $\overline{\Phi}$ are connected. The number of connected components of the level sets of $\Phi$ is the index $[P' : P]$. By definition, the level sets of $\Phi$ are connected, hence $[P' : P] = 1$, and therefore $P'$ must be equal to $P$. Choosing a Riemannian metric, and flowing one of the level sets with the gradient flow of $\Phi$, we see that $M$ is connected. Thus, $(M, \omega, \Phi)$ is indeed a four-dimensional tight circle-valued Hamiltonian $S^1$-manifold with group of periods $P$.
\end{proof}

\begin{corollary}\label{phi_diffeomorphic_iff_cohomologic}
    Let $(M, \omega, \Phi)$ and $(M', \omega', \Phi')$ be four-dimensional tight circle-valued Hamiltonian $S^1$-manifolds with the same group of periods $P$. Then they have $\Phi$-diffeomorphic quotients if and only if they represent the same class in the first Čech cohomology $\check H^1(\mathbb{R}/P, \mathcal{Q}_{\mathbb{R}/P})$.
\end{corollary}
\begin{proof}
    This is immediate by Lemma~\ref{abstract_nonsense_lemma} and Claim~\ref{global_objects_of_q_claim}.
\end{proof}

In Section~\ref{preliminaries_section}, we stated the local normal form theorem for non-fixed orbits in four-dimensional Hamiltonian $S^1$-manifolds (see Theorem~\ref{simple_local_normal_form}). A four-dimensional circle-valued Hamiltonian $S^1$-piece over $U\subsetneq \mathbb{R}/P$ has no fixed points, and can be regarded as a proper four-dimensional Hamiltonian $S^1$-manifold by embedding $U$ into $\mathbb{R}$. Therefore, we can apply the local normal form theorem as it's stated. We now define \textbf{grommets} for four-dimensional circle-valued Hamiltonian $S^1$-pieces. See~\cite[Section 8]{centered_hamiltonians} for the definition of grommets for complexity one spaces.
\begin{definition} \label{grommet_definition}
    Let $(M, \omega, \Phi)$ be a four-dimensional circle-valued Hamiltonian $S^1$-piece over $U \subset \mathbb{R}/P$. A \textbf{grommet} is a $\Phi$-$T$-diffeomorphism $\Psi:D \rightarrow M$ from an invariant open subset $D$ of a local
    model $Y = S^1 \times_{\mathbb{Z}_n} \mathbb{C} \times \mathbb{R}$ onto an open subset of $M$. 
\end{definition}
\begin{definition}
    Let $(M, \omega, \Phi)$ be a four-dimensional circle-valued Hamiltonian $S^1$-piece over $U \subset \mathbb{R}/P$. We say that $(M, \omega, \Phi)$ is \textbf{grommeted} if it is equipped with grommets whose images are pairwise disjoint, and cover all of the exceptional orbits.
\end{definition}

We now define the sheaf $\hat{\mathcal{Q}}_{\mathbb{R}/P}$ over $\mathbb{R}/P$. For every $U\subsetneq \mathbb{R}/P$, the objects of $\hat{\mathcal{Q}}_{\mathbb{R}/P}(U)$ are grommeted four-dimensional circle-valued Hamiltonian $S^1$-pieces over $U$, and the arrows are $\Phi$-diffeomorphisms between their quotients.
For $U = \mathbb{R}/P$, we define $\hat{\mathcal{Q}}_{\mathbb{R}/P}(\mathbb{R}/P)$ to have no objects.

Since we required that $\hat{\mathcal{Q}}_{\mathbb{R}/P}(\mathbb{R}/P)$ has no objects, we have that for every $U \subset \mathbb{R}/P$, all of the objects in the groupoid $\hat{\mathcal{Q}}_{\mathbb{R}/P}(U)$ can be regarded as complexity one spaces by embedding $U$ into $\mathbb{R}$ (which is always possible when $U \ne \mathbb{R}/P$).
\\
By Lemma~\ref{other_abstract_nonsense_lemma}, We have that:
\begin{lemma} \label{grommet_sheaves_isomorphism}
    $\check H^1(\mathbb{R}/P, \hat{\mathcal{Q}}_{\mathbb{R}/P}) \cong \check H^1(\mathbb{R}/P, \mathcal{Q}_{\mathbb{R}/P})$.
\end{lemma}

Let $Y = S^1 \times_{\mathbb{Z}_n} \mathbb{C} \times \mathbb{R}$ be a local model. By~\cite[Definition 5.12, Lemma 6.2]{centered_hamiltonians}, there is a map $\overline{P}:Y/S^1 \rightarrow \mathbb{C}$ called the \textbf{defining monomial}, such that the map
\begin{equation*}
    F := (\overline{\Phi}, \overline{P}): Y/S^1 \rightarrow \mathbb{R} \times \mathbb{C}
\end{equation*}
is a homeomorphism onto $\image(\Phi) \times \mathbb{C}$. The map $F$ is called the \textbf{trivializing homeomorphism}. By~\cite[Corollary 7.2]{centered_hamiltonians}, it is a diffeomorphism off the set of exceptional orbits.

\begin{definition}[cf. {\cite[Definition 9.5]{tall_uniqueness}}]
    Let $Y = S^1 \times_{\mathbb{Z}_n} \mathbb{C} \times \mathbb{R}$ be a local model, let $\Psi: D \rightarrow M$ be a grommet with $D \subset Y$, and let $F$ be the trivializing homeomorphism. Let $B \subset \image(\Phi) \times \mathbb{C}$ be the image of $D$ through $F$. Define the map $G:B \rightarrow M/S^1$ by $G := \Psi \circ F^{-1}$.
    
    Then $G$ is a homeomorphism onto an open subset of $M/S^1$. We call it the \textbf{surface bundle grommet} associated to the grommet $\Psi$.
\end{definition}
Every surface bundle grommet induces a smooth structure on its image, and it agrees with the smooth structure of $M/S^1$ outside the set of exceptional orbits. We now introduce sb-diffeomorphisms between quotients of grommeted four-dimensional circle-valued Hamiltonian $S^1$-pieces. See~\cite[Definition 11.1]{tall_uniqueness} for the definition of sb-diffeomorphisms for complexity one spaces.
\begin{definition} \label{sb_diffeo_definition}
    Let $(M, \omega, \Phi)$ and $(M', \omega', \Phi')$ be grommeted four-dimensional circle-valued Hamiltonian $S^1$-pieces. A homeomorphism $f : M/S^1 \rightarrow M'/S^1$ is called an \textbf{sb-diffeomorphism} if for every pair of associated surface bundle grommets $G:B \rightarrow M/S^1$ and $G':B' \rightarrow M'/S^1$, the map $(G')^{-1} \circ f \circ G$ is a diffeomorphism.
\end{definition}

We now define the sheaf $\hat{\mathcal{P}}_{\mathbb{R}/P}$ over $\mathbb{R}/P$. For every $U\subsetneq \mathbb{R}/P$, the objects of $\hat{\mathcal{P}}_{\mathbb{R}/P}(U)$ are grommeted four-dimensional circle-valued Hamiltonian $S^1$-pieces over $U$, and the arrows are sb-diffeomorphisms between their quotients.
For $U = \mathbb{R}/P$, we define $\hat{\mathcal{P}}_{\mathbb{R}/P}(\mathbb{R}/P)$ to have no objects.

\begin{lemma} \label{main_sheaves_isomorphism}
    $\check H^1(\mathbb{R}/P, \hat{\mathcal{P}}_{\mathbb{R}/P}) \cong \check H^1(\mathbb{R}/P, \hat{\mathcal{Q}}_{\mathbb{R}/P})$
\end{lemma}
\begin{proof}
    We follow the proof of the analogous claim for proper tall complexity one spaces, given in~\cite{tall_uniqueness}. We define 3 more sheaves, each sheaf with the same objects as in $\hat{\mathcal{Q}}_{\mathbb{R}/P}$ and $\hat{\mathcal{P}}_{\mathbb{R}/P}$, but with different arrows (see sections 7-11 in~\cite{tall_uniqueness} for the definitions of these new notions of maps):
    \begin{itemize}
        \item $\mathcal{RQ}_{\mathbb{R}/P}$ whose arrows are locally rigid $\Phi$-homeomorphisms
        \item $\mathcal{E}_{\mathbb{R}/P}$ whose arrows are local stretch maps
        \item $\mathcal{RP}_{\mathbb{R}/P}$ whose arrows are locally sb-rigid $\Phi$-homeomorphisms
    \end{itemize}
    We will show the following chain of isomorphisms:
    \begin{align*}
        \check H^1(\mathbb{R}/P, \hat{\mathcal{Q}}_{\mathbb{R}/P}) &\cong \check H^1(\mathbb{R}/P, \mathcal{RQ}_{\mathbb{R}/P}) \cong \check H^1(\mathbb{R}/P, \mathcal{E}_{\mathbb{R}/P}) 
        \\ &\cong \check H^1(\mathbb{R}/P, \mathcal{RP}_{\mathbb{R}/P}) \cong \check H^1(\mathbb{R}/P, \hat{\mathcal{P}}_{\mathbb{R}/P})
    \end{align*}

    In~\cite{tall_uniqueness}, these isomorphisms are given for tall complexity one spaces in Lemmas 8.16. 9.10, 10.4 and 11.3, respectively. These lemmas were proved using the following sheaf-theoretic lemma:
    \begin{lemma}[Lemma 12.1 in~\cite{tall_uniqueness}]\label{sheaf_isomorphic_arrows_lemma}
        Let $A$ and $B$ be sheaves of groupoids on $\mathcal{T}$ with the same objects and such that the arrows in $A$ are subsets of the arrows in $B$. Suppose that for every open cover $U = \{U_i\}$ and cocycle $\beta \in \check{Z}^1(U, B)$ there exists an open cover $\{U_i'\}$ such that $U_i' \subset U_i$ for all $i$ and such that the following holds:
        \begin{quote}
            Take any $U_i$ and $U_j$ in $U$; let $N$ and $N'$ be the restriction to $Z := U_i \cap U_j$ of the objects associated to $U_i$ and $U_j$ by $\beta$. Let $X$ and $Y$ be any open sets such that $X \cap Y = \empty$, let $W = U_1' \cup ... \cup U_{p-1}'$ for any integer $p$, and let
            $f: N \rightarrow N'$ be any $B$ arrow. Then there exists a $B$-arrow $f':N \rightarrow N'$ with the following properties:
            \begin{enumerate}
                \item The restriction of $f'$ to $X \cap Z$ is in $A$.
                \item The restrictions of $f'$ and $f$ to $Y \cap Z$ coincide.
                \item If the restriction of $f$ to $W \cap Z$ lies in $A$, then the restriction of $f'$ to $W \cap Z$ also lies in $A$.
            \end{enumerate}
        \end{quote}
        Then the inclusion map $i : A \rightarrow B$ induces an isomorphism in cohomology
        \begin{equation*}
            i^* : \check H^1(\mathcal{T} , A) \cong \check H^1(\mathcal{T} , B).
        \end{equation*}
    \end{lemma}
    
    Conceptually, the lemma claims that given two sheaves of groupoids, which have the same objects and the second has fewer maps, if the arrows of the first sheaf can be locally modified to also be arrows of the second sheaf, then both sheaves have the same first cohomology.
    
    In~\cite{tall_uniqueness}, the conditions of Lemma 12.1 were shown to be satisfied for each pair of sheaves in Propositions 13.19, 14.9, 14.12 and 15.10, respectively. These propositions show that for each pair of sheaves, the arrows of one of the sheaves can be locally isotoped to also be arrows of the other sheaf. This is done in the setting of proper tall complexity one space.
    
    Since the objects of $\hat{\mathcal{Q}}_{\mathbb{R}/P}$, $\mathcal{RQ}_{\mathbb{R}/P}$, $\mathcal{E}_{\mathbb{R}/P}$, $\mathcal{RP}_{\mathbb{R}/P}$ and $\hat{\mathcal{P}}_{\mathbb{R}/P}$ can be regarded as proper tall grommeted complexity one spaces, the propositions also apply to our case. We deduce that the cohomology groups are isomorphic by Lemma~\ref{sheaf_isomorphic_arrows_lemma} above.
\end{proof}

We now introduce painted surface bundles over subsets of $\mathbb{R}/P$.
\begin{definition}
    A \textbf{painted surface bundle} over an open set $U \subset \mathbb{R}/P$ is an oriented smooth surface bundle over $U$, with a subset $A$ whose points are labeled by isotropy representations of subgroups of $S^1$. We refer to $A$ as \textbf{the painted subset}, and to each point in $A$ as a \textbf{labeled point}.

    An isomorphism between two painted surface bundles $\pi:E \rightarrow U$ and $\pi':E' \rightarrow U$ is an orientation-preserving diffeomorphism $f:E \rightarrow E'$ such that:
    \begin{enumerate}
        \item $\pi' \circ f = \pi$,
        \item both $f$ and $f^{-1}$ send each labeled point to a labeled point with the same label.
    \end{enumerate}
\end{definition}
\begin{remark}
    The monodromy invariant introduced in the next section replaces the painting invariant defined in~\cite{tall_uniqueness}. For the definition of the painting invariant, Karshon and Tolman introduced the notion of a \textit{skeleton}, that models the painted subset of a painted surface bundle. Since we don't need the notion of a skeleton in this paper, we slightly simplified the definition of painted surface bundles.
\end{remark}
\begin{definition}
     Let $(M, \omega, \Phi)$ be a grommeted four-dimensional circle-valued Hamiltonian $S^1$-piece over $U \subset \mathbb{R}/P$. \textbf{The associated painted surface bundle} of $(M, \omega, \Phi)$ is the painted surface bundle given by the following data:
     \begin{enumerate}
         \item The oriented topological manifold $M/S^1$, with the bundle map $\overline{\Phi}:M/S^1 \rightarrow U$ induced by the Hamiltonian $\Phi$.
         \item The manifold structure that is given by the following charts: Choose charts that cover the complement of the set of exceptional orbits, together with the associated surface bundles grommets for each given grommet.
         \item The painted subset that consists of the set of exceptional orbits, each labeled by its isotropy representation.
     \end{enumerate}
\end{definition}

\begin{definition}
    Let $\pi:E \rightarrow U$ be a painted surface bundle over $U \subset \mathbb{R}/P$. We say that $E$ is \textbf{legal} if there exists a cover $\{V_j\}$ of $U$ such that for every $V_j$ there exists a grommeted four-dimensional circle-valued Hamiltonian $S^1$-piece over $V_j$, whose associated painted surface bundle is isomorphic to $E|_{\pi^{-1}(V_j)}$.
\end{definition}

We define the sheaf $\mathcal{P}_{\mathbb{R}/P}$ of painted surface bundles. For every $U\subset \mathbb{R}/P$ the objects of $\mathcal{P}_{\mathbb{R}/P}(U)$ are legal painted surface bundles over $U$, and the arrows are isomorphisms of painted surface bundles.

There is a natural map $\hat{\mathcal{P}}_{\mathbb{R}/P} \rightarrow \mathcal{P}_{\mathbb{R}/P}$, taking each grommeted four-dimensional circle-valued Hamiltonian $S^1$-piece over $U$ to its associated painted surface bundle over $U$, and each sb-diffeomorphism to itself, noting that each sb-diffeomorphism is an isomorphism of the associated painted surface bundles.
By Lemma~\ref{other_abstract_nonsense_lemma}, We have that:
\begin{lemma}\label{painted_sheaves_isomorphism}
    The map $\hat{\mathcal{P}}_{\mathbb{R}/P} \rightarrow \mathcal{P}_{\mathbb{R}/P}$ induces an isomorphism 
    \begin{equation*}
        \check H^1(\mathbb{R}/P, \hat{\mathcal{P}}_{\mathbb{R}/P}) \cong \check H^1(\mathbb{R}/P, \mathcal{P}_{\mathbb{R}/P})
    \end{equation*}
\end{lemma}

Finally, we get the long-awaited isomorphism.
\begin{corollary}\label{cohomology_isomorphism}
    \begin{equation*}
        \check H^1(\mathbb{R}/P, \mathcal{Q}_{\mathbb{R}/P}) \cong \check H^1(\mathbb{R}/P, \mathcal{P}_{\mathbb{R}/P})
    \end{equation*}
\end{corollary}
\begin{proof}
    This is immediate by Lemma~\ref{grommet_sheaves_isomorphism}, Lemma~\ref{main_sheaves_isomorphism}, and Lemma~\ref{painted_sheaves_isomorphism}.
\end{proof}
\begin{definition}
    Let $(M, \omega, \Phi)$ be a four-dimensional tight circle-valued Hamiltonian $S^1$-manifold, with group of periods $P$. \textbf{An associated painted surface bundle} to $(M, \omega, \Phi)$ is a painted surface bundle over $\mathbb{R}/P$ whose cohomology class in $\check H^1(\mathbb{R}/P, \mathcal{P}_{\mathbb{R}/P})$ corresponds to the cohomology class of $(M, \omega, \Phi)$ in $\check H^1(\mathbb{R}/P, \mathcal{Q}_{\mathbb{R}/P})$ under the isomorphism of Corollary~\ref{cohomology_isomorphism}.
\end{definition}
The sheaf $\mathcal{P}_{\mathbb{R}/P}$ over $\mathbb{R}/P$ has glueable objects. Therefore, by Lemma~\ref{abstract_nonsense_lemma}, the global objects in $\check H^1(\mathbb{R}/P, \mathcal{P}_{\mathbb{R}/P})$ correspond to isomorphism classes of legal painted surface bundles over $\mathbb{R}/P$. We formalize this in the following proposition, along with the assertion that topologically, an associated painted surface bundle is isomorphic to the quotient space together with the information on the non-free orbits.

\begin{proposition}\label{associated_painted_surface_bundle_proposition}
    Let $(M, \omega, \Phi)$ be a four-dimensional tight circle-valued Hamiltonian $S^1$-manifold. Then there exists a unique up to isomorphism associated painted surface bundle $N$ over $\mathbb{R}/P$. Moreover, If $(M', \omega', \Phi')$ is another four-dimensional tight circle-valued Hamiltonian $S^1$-manifold, with the same group of periods $P$, and $N'$ is a painted surface bundle associated to $M'$, then $M/S^1$ and $M'/S^1$ are $\Phi$-diffeomorphic if and only if $N$ and $N'$ are isomorphic.

    Furthermore, there is an orientation-preserving homeomorphism $\Psi: M/S^1 \rightarrow N$ that respects the projections to $\mathbb{R}/P$, and induces a bijection between the set of non-free orbits in $M/S^1$ and the painted subset of $N$, such that every non-free orbit is sent to a point that is labeled by its isotropy representation.
\end{proposition}
\begin{proof}
    The proof is identical to the proof of Proposition 17.5 in~\cite{tall_uniqueness}.
\end{proof}


\section{Classification of painted surface bundles} \label{painted_surface_bundle_section}
In this section, we classify legal painted surface bundles over $\mathbb{R}/P$ up to isomorphism. This completes the uniqueness part of the classification of four-dimensional tight circle-valued Hamiltonian $S^1$-manifolds up to $\Phi$-diffeomorphisms between their quotients (see Proposition~\ref{uniqueness_theorem_phi_diffeo}). 

Intuitively, we think of a legal painted surface bundle as a surface bundle over $\mathbb{R}/P$, with colored paths in the total space that loop around the base $\mathbb{R}/P$. Locally, these paths look like sections (see Lemma~\ref{trivialize_painted_surface_bundle}), but globally they might circle the base $\mathbb{R}/P$ more than once, and have interesting topology. We show that every legal painted surface bundle is isomorphic to a mapping torus constructed by some smooth gluing map (see Lemma~\ref{construct_as_mapping_torus}), and deduce that legal painted surface bundles are classified by the genus of their level sets, the list of coprime residue classes that corresponds to the painted points in each level set, and the class of the gluing map (See Proposition~\ref{classification_of_painted_surface_bundles}). Subsequently, we show that the class of the gluing map can be calculated topologically, and does not depend on the smooth structure of the painted surface bundle, and we deduce Proposition~\ref{topological_model_proposition}, and a classification theorem for our spaces, up to $\Phi$-diffeomorphisms (Proposition~\ref{uniqueness_theorem_phi_diffeo}).
\\

Let $C_1,...,C_k$ be an unordered list of coprime residue classes. Fix a smooth surface of genus $g$, and label $k$ distinct points with $C_1,...,C_k$. We denote the labeled surface by $\Sigma_{g,C_1,...,C_k}$. For every interval $I$, we denote by $\Sigma_{g, C_1,...,C_k} \times I$ the trivial painted surface bundle $\Sigma_g \times I \rightarrow I$, whose painted subset consists of $k$ constant sections, painted with the isotropy representations that correspond to the coprime residue classes $C_1,...,C_k$.

First, we show that the isomorphism class of a painted surface bundle associated to a grommeted four-dimensional circle-valued Hamiltonian $S^1$-piece is determined by the genus of the reduced spaces, and the isotropy data of any one of the level sets.
\begin{lemma}\label{trivialize_painted_surface_bundle}
    Let $P$ be a non-trivial discrete subgroup of $\mathbb{R}$.
    Let $(M, \omega, \Phi)$ be a grommeted four-dimensional circle-valued Hamiltonian $S^1$-piece over an open interval $U \subsetneq \mathbb{R}/P$. Let $g$ be the genus of one of the reduced sets, and $C_1,...,C_k$ be the coprime residue classes that correspond to the isotropy data of one of the level sets. Then the associated painted surface bundle of $(M, \omega, \Phi)$ is isomorphic to the painted surface bundle $\Sigma_{g, C_1,...,C_k} \times U$.
\end{lemma}
\begin{proof}
    By definition, the associated painted surface bundle is a surface bundle over $U$, whose fiber is $\Sigma_g$, with some painted subset $S$. Since $U$ is an interval, the bundle is trivializable.
    
    By the definition of a four-dimensional circle-valued Hamiltonian $S^1$-piece, it has no fixed points. Hence, by the local normal form, every grommet has a complexity one model of the form $Y := S^1 \times_{\mathbb{Z}_n} \mathbb{C} \times \mathbb{R}$, with $\mathbb{Z}_n$ corresponding to the stabilizer given by $C_j$.
    
    The image of the exceptional orbits of $Y/S^1$ under the trivializing homeomorphism $F:Y/S^1 \rightarrow \mathbb{R} \times \mathbb{C}$ is the subset $\mathbb{R} \times \{0\}$. Moreover, all of the exceptional orbits in the grommet have the isotropy representation that corresponds to $C_j$. Thus, every point $y$ in the painted subset $S$ of the associated painted surface bundle has a small neighborhood of the form $W \times U$, such that the intersection $S \cap (W \times U)$ is the image of some smooth section $s_j:U \rightarrow W \times U$. Therefore, the painted subset of the associated painted surface bundle has $k$ connected components, each one a section of $s_j:U \rightarrow \Sigma_g \times U$, painted by the isotropy representation that corresponds to $C_j$.
    
    Using a partition of unity, we choose an Ehresmann connection such that all of these sections are horizontal. Using the parallel transport of one of the fibers, we get a new trivialization of the painted surface bundle. In the new trivialization, the painted subset consists of the chosen $k$ constant sections, painted by the isotropy representations that correspond to $C_1$,..., $C_k$, respectively.
\end{proof}
\begin{corollary}\label{painted_bundle_same_isotropy_data}
    Let $E$ be a legal painted surface bundle over $\mathbb{R}/P$. Then all of its level sets have the same unordered list of isotropy representations corresponding to their painted points. We call the corresponding list of coprime residue classes the \textbf{isotropy data invariant} of $E$.
\end{corollary}
\begin{proof}
    Since $E$ is legal, we have a cover $\{V_j\}$ of $\mathbb{R}/P$, and four-dimensional circle-valued Hamiltonian $S^1$-pieces over each $V_j$, whose associated painted surface bundles are isomorphic to $E|_{\pi^{-1}(V_j)}$. By Lemma~\ref{trivialize_painted_surface_bundle}, for every $V_j$, each level set of $E|_{\pi^{-1}(V_j)}$ has the same list of isotropy representations. Because these lists agree on intersections of different $V_j$s, the unordered list of isotropy representations does not change through the level sets of $E$.

    Note that every isotropy representation corresponds to a non-free non-fixed orbit, and therefore given by a coprime residue class (see the discussion after Theorem~\ref{simple_local_normal_form}).
\end{proof}

We now give a reminder about the mapping class group of a surface with marked points, and define subgroups of the mapping class group that only permute certain marked points, according to a labeling of the points. These subgroups will be used for defining the monodromy invariant.

Let $g$ and $k$ be non-negative integers. We denote by $\Sigma_{g,k}$ the surface of genus $g$, with $k$ marked points. We denote by $\Diff_+(\Sigma_{g, k})$ the group of orientation-preserving diffeomorphisms that preserve the marked points. The \textbf{smooth mapping class group} of $\Sigma_{g,k}$ is defined by
\begin{equation*}
    \mcg^\infty(\Sigma_{g,k}) := \Diff_+(\Sigma_{g,k})/{\sim},
\end{equation*}
where $f \sim h$ if they are smoothly isotopic relative to the marked points.

Similarly, we denote by $\Homeo_+(\Sigma_{g, k})$ the group of orientation-preserving homeomorphisms that preserve the marked points. The \textbf{topological mapping class group} of $\Sigma_{g,k}$ is defined by
\begin{equation*}
    \mcg(\Sigma_{g,k}) :=\Homeo_+(\Sigma_{g,k})/{\sim},
\end{equation*}
where $f \sim h$ if they are isotopic relative to the marked points.

There is a natural forgetful homomorphism
\begin{equation*}
    \mcg^\infty(\Sigma_{g,k}) \xrightarrow{\Forgetful} \mcg(\Sigma_{g,k}).
\end{equation*}
It is a classical result that this forgetful map is an isomorphism, i.e., that the two definitions of the mapping class group coincide. We weren't able to find a reference that proves this isomorphism for surfaces with marked points, and therefore we give a proof in Appendix~\ref{mcg_appendix}, see Theorem~\ref{mcg_forgetful_isomorphism}.
\\

We denote by $\Diff_+(\Sigma_{g, C_1,...,C_k})$ the group of orientation-preserving diffeomorphisms of $\Sigma_{g, C_1,...,C_k}$ that send each labeled point to a labeled point with the same label. We define a subgroup of $\mcg^\infty(\Sigma_{g,k})$ that consists of classes of diffeomorphisms that respect the labels:
\begin{definition}
    The \textbf{(smooth) mapping class group} of $\Sigma_{g, C_1,...,C_k}$ is defined by
    \begin{equation*}
        \mcg^\infty(\Sigma_{g, C_1,...,C_k}) := \Diff_+(\Sigma_{g, C_1,...,C_k}) /{\sim},
    \end{equation*}
    where $f \sim h$ if they are smoothly isotopic relative to the labeled points.
\end{definition}

Similarly, we denote by $\Homeo_+(\Sigma_{g, C_1,...,C_k})$ the group of orientation-preserving homeomorphisms of $\Sigma_{g, C_1,...,C_k}$ that send each labeled point to a labeled point with the same label, and define a subgroup of $\mcg(\Sigma_{g,k})$ that consists of classes of homeomorphisms that respect the labels:
\begin{definition}
    The \textbf{(topological) mapping class group} of $\Sigma_{g, C_1,...,C_k}$ is defined by
    \begin{equation*}
        \mcg(\Sigma_{g, C_1,...,C_k}) := \Homeo_+(\Sigma_{g, C_1,...,C_k}) /{\sim},
    \end{equation*}
    where $f \sim h$ if they are isotopic relative to the labeled points.
\end{definition}
The $\Forgetful$ isomorphism between $\mcg^\infty(\Sigma_{g,k})$ and $\mcg(\Sigma_{g,k})$ restricts to an isomorphism between $\mcg^\infty(\Sigma_{g, C_1,...,C_k})$ and $\mcg(\Sigma_{g, C_1,...,C_k})$, hence the two definitions coincide.

For example, when all of the labels are the same, then $\mcg(\Sigma_{g, C_1,...,C_k})$ is the same group as $\mcg(\Sigma_{g, k})$ by definition. When all of the labels are pairwise different, then $\mcg(\Sigma_{g, C_1,...,C_k})$ is the subgroup of $\mcg(\Sigma_{g, k})$ that fixes each marked point individually, which is known as the pure mapping class group of $\Sigma_{g,k}$.
\\
Now, we show that every legal painted surface bundle is isomorphic to a mapping torus glued by some map in $\Diff_+(\Sigma_{g, C_1,...,C_k})$. Recall that the smooth structure on a mapping torus $M \times [0,1] / {(x,0) \sim (f(x), 1)}$ depends on a choice of collar neighborhoods for $M \times \{0\}$ and $M \times \{1\}$. Here, we use the canonical collar neighborhoods $M \times [0, \varepsilon]$, $M \times [1 - \varepsilon, 1]$.
\begin{lemma}\label{construct_as_mapping_torus}
    Let $E$ be a legal painted surface bundle over $\mathbb{R}/P$. Let $g$ be its genus, let $C_1,...,C_k$ be its isotropy data invariant, and let $\tau$ be the positive generator of $P$.
    
    Then there exists a map $f$ in $\Diff_+(\Sigma_{g, C_1,...,C_k})$ such that $E$ is isomorphic to the painted surface bundle defined by
    \begin{equation*}
        \bigslant{\Sigma_{g, C_1,...,C_k} \times [0, \tau]}{(x,0) \sim (f(x), \tau)}.
    \end{equation*}
\end{lemma}
\begin{proof}
    By Lemma~\ref{trivialize_painted_surface_bundle}, we can cover $\mathbb{R}/P$ with two intervals $U$ and $V$ such that $E|_{\pi^{-1}(U)}$ and $E|_{\pi^{-1}(V)}$ have global trivializations over $U$ and $V$. By gluing both trivializations on one of the connected components of $U \cap V$, by the transition map, we get a painted surface bundle of the form $\Sigma_{g, C_1,...,C_k} \times (-\varepsilon, \tau + \varepsilon)$.
    
    The transition map in the second connected component of $U \cap V$ induces a map $F:\Sigma_{g, C_1,...,C_k} \times (-\varepsilon, \varepsilon) \rightarrow \Sigma_{g, C_1,...,C_k} \times (\tau - \varepsilon, \tau + \varepsilon)$ that satisfies $F(x, t) = (f_t(x), t+\tau)$ for a smooth family of maps $f_t \in \Diff_+(\Sigma_{g, C_1,...,C_k})$, parameterized by $t \in (-\varepsilon, \varepsilon)$. We can change the trivialization of $\Sigma_{g, C_1,...,C_k} \times (-\varepsilon, \tau + \varepsilon)$ such that $f_t$ is constant near $t = 0$. Then, by restricting to $[0, \tau]$, we get that the map $f_0$ glues the boundaries of $\Sigma_{g, C_1,...,C_k} \times [0, \tau]$ to get a painted surface bundles that is isomorphic to $E$.
\end{proof}
\begin{lemma}\label{surface_bundles_isomorphic_if_conjugate}
    Let $f$ and $g$ be maps in $\Diff_+(\Sigma_{g, C_1,...,C_k})$. Let $M_f$ and $M_g$ be the painted surface bundles over $\mathbb{R}/P$, constructed as the mapping tori of $f$ and $g$. Then $M_f$ and $M_g$ are isomorphic if and only if $[f]$ and $[g]$ are conjugate in $\mcg^\infty(\Sigma_{g, C_1,...,C_k})$.
\end{lemma}
\begin{proof}
    We follow the proof of the classical claim that surface bundles over a circle are classified by a conjugacy class in the mapping class group of the surface. Here, we also have labeled points in each fiber, hence the transition maps live in $\Diff_+(\Sigma_{g, C_1,...,C_k})$ instead of $\Diff_+(\Sigma_g)$, and therefore the appropriate mapping class group is $\mcg^\infty(\Sigma_{g, C_1,...,C_k})$.

    Let $M_f$ and $M_g$ be isomorphic. By their definition as mapping tori, an isomorphism between $M_f$ and $M_g$ defines a smooth family of maps $F_t:\Sigma_{g, C_1,...,C_k} \rightarrow \Sigma_{g, C_1,...,C_k}$ in $\Diff_+(\Sigma_{g, C_1,...,C_k})$, parameterized by $t \in [0, \tau]$, such that the following diagram commutes:
    \begin{equation*}
        \begin{tikzcd}\label{diagram_ef_eg}
            \Sigma_{g, C_1,...,C_k} \arrow{r}{f} \arrow[swap]{d}{F_0} & \Sigma_{g, C_1,...,C_k} \arrow{d}{F_\tau} \\%
            \Sigma_{g, C_1,...,C_k} \arrow{r}{g}& \Sigma_{g, C_1,...,C_k}
        \end{tikzcd}
    \end{equation*} 
    By definition, $F_\tau$ and $F_0$ are smoothly isotopic, relative to the labeled points. Hence, $f$ and $g$ are indeed in the same conjugacy class in $\mcg^\infty(\Sigma_{g, C_1,...,C_k})$.
    \\
    
    On the other, let $[f]$ and $[g]$ be two elements in the same conjugacy class in $\mcg^\infty(\Sigma_{g, C_1,...,C_k})$. Then, there exists maps $\tilde f$, $\tilde g$, $h$, and $\tilde h$ in $\Diff_+(\Sigma_{g, C_1,...,C_k})$ that satisfy:
    \begin{enumerate}
        \item $f$ and $\tilde f$ are smoothly isotopic relative to the labeled points,
        \item $g$ and $\tilde g$ are smoothly isotopic relative to the labeled points,
        \item $h$ and $\tilde h$ are smoothly isotopic relative to the labeled points, and
        \item $h^{-1} \circ \tilde g \circ \tilde h = \tilde f$.
    \end{enumerate}
    We can choose smooth isotopies $F_t$ and $G_t$, relative to the labeled points, from the identity map to $\tilde f \circ f^{-1}$ and $\tilde g \circ g^{-1}$. We then have 
    \begin{equation*}
        h^{-1} \circ G_1 \circ g \circ \tilde h = F_1 \circ f,
    \end{equation*}
    and by setting $h' := G_1^{-1} \circ h \circ F_1$, we get that
    \begin{equation*}
        h'^{-1} \circ g \circ \tilde h = f
    \end{equation*}
    where $h'$ is smoothly isotopic, relative to the labeled points, to $h$, and therefore to $\tilde h$ as well.
    
    The isotopy between $h'$ and $\tilde h$ induces a smooth family of maps $H_t:\Sigma_{g, C_1,...,C_k} \rightarrow \Sigma_{g, C_1,...,C_k}$, parameterized by $t \in [0, \tau]$. Without loss of generality we can assume that $H_t$ is constant near $t = 0$ and $t = \tau$. Then $H_t$ defines a diffeomorphism of $\Sigma_{g, C_1,...,C_k} \times [0, \tau]$ that induces an isomorphism between $M_f$ and $M_g$.
\end{proof}
\begin{proposition}\label{classification_of_painted_surface_bundles}
    Every legal painted surface bundle $E$ over $\mathbb{R}/P$ with genus $g$, and isotropy data invariant $C_1,...,C_k$, has a unique well-defined conjugacy class in $\mcg^\infty(\Sigma_{g, C_1,...,C_k})$, which we call the \textbf{surface bundle monodromy} of $E$, that corresponds to the class of a gluing map from which an isomorphic mapping torus can be constructed.
    
    Moreover, two legal painted surface bundles over $\mathbb{R}/P$ are isomorphic if and only if they have the same genus, the same isotropy data invariant, and the same surface bundle monodromy.
\end{proposition}
\begin{proof}
    Let $E$ be a painted surface bundle with genus $g$, and isotropy data invariant $C_1,...,C_k$. Then by Lemma~\ref{construct_as_mapping_torus}, it is isomorphic to the mapping torus of some map $f$ in $\Diff_+(\Sigma_{g, C_1,...,C_k})$. By Lemma~\ref{surface_bundles_isomorphic_if_conjugate}, the conjugacy class of $[f]$ in $\mcg^\infty(\Sigma_{g, C_1,...,C_k})$ is independent of the choice of $f$, and we call it the surface bundle monodromy of $E$.

    Let $E$ and $E'$ be two legal painted surface bundles over $\mathbb{R}/P$. For $E$ and $E'$ to be isomorphic, they must have the same fiber, i.e., the genus of their surfaces has to be same. Moreover, they must have the same isotropy data invariant, otherwise the bundle isomorphism cannot respect the labels of the painted subset. By Lemma~\ref{construct_as_mapping_torus}, $E$ and $E'$ are isomorphic to mapping tori $M_f$ and $M_g$ constructed by some maps $f$ and $g$ in $\Diff_+(\Sigma_{g, C_1,...,C_k})$, respectively. By Lemma~\ref{surface_bundles_isomorphic_if_conjugate}, $M_f$ and $M_g$ are isomorphic if and only the conjugacy classes of $[f]$ and $[g]$ in $\mcg^\infty(\Sigma_{g, C_1,...,C_k})$ agree. Therefore, $E$ and $E'$ are isomorphic if and only if they have the same genus, isotropy data invariant, and surface bundle monodromy.
\end{proof}
\begin{definition}\label{monodromy_invariant_definition}
    Let $(M, \omega, \Phi)$ be a four-dimensional tight circle-valued Hamiltonian $S^1$-manifold with genus invariant $g$, and isotropy data invariant $C_1,...,C_k$. Let $E$ be an associated painted surface bundle, and let the surface bundle monodromy of $E$ be the conjugacy class of $[f]$, where $[f]$ is an element in $\mcg^\infty(\Sigma_{g, C_1,...,C_k})$.
    
    We define the \textbf{monodromy invariant} of $(M, \omega, \Phi)$ to be the conjugacy class of $\Forgetful([f])$ in $\mcg(\Sigma_{g, C_1,...,C_k})$.
\end{definition}
\begin{proposition}[Uniqueness up to $\Phi$-diffeomorphisms]\label{uniqueness_theorem_phi_diffeo}
    Let $(M, \omega, \Phi)$ and $(M', \omega', \Phi')$ be four-dimensional tight circle-valued Hamiltonian $S^1$-manifolds. Assume that they have the same group of periods $P$. Then they have $\Phi$-diffeomorphic quotients if and only if they have the same genus invariant, isotropy data invariant, and monodromy invariant.
\end{proposition}
\begin{proof}
    By Proposition~\ref{associated_painted_surface_bundle_proposition}, we can choose associated painted surface bundles $E$ and $E'$ over $\mathbb{R}/P$, for $(M, \omega, \Phi)$ and $(M', \omega', \Phi')$. Moreover by the same proposition, $(M, \omega, \Phi)$ and $(M', \omega', \Phi')$ have $\Phi$-diffeomorphic quotients if and only if $E$ and $E'$ are isomorphic. By Proposition~\ref{classification_of_painted_surface_bundles}, the painted surface bundles $E$ and $E'$ are isomorphic if and only if the have the same genus, isotropy data invariant, and surface bundle monodromy.
    
    The genera of the associated painted surface bundles $E$ and $E'$ correspond to the genus invariants of $(M, \omega, \Phi)$ and $(M', \omega', \Phi')$. The isotropy data invariants of $E$ and $E'$ correspond to the isotropy data invariants of $(M, \omega, \Phi)$ and $(M', \omega', \Phi')$. Lastly, since the $\Forgetful$ map is an isomorphism (see Theorem~\ref{mcg_forgetful_isomorphism}), the surface bundle monodromies of the associated painted surface bundles are the same if and only if the monodromy invariants of the $(M, \omega, \Phi)$ and $(M', \omega', \Phi')$ are the same.
\end{proof}
Given a four-dimensional tight circle-valued Hamiltonian $S^1$-manifold, while it is possible to calculate its monodromy invariant by constructing an associated painted surface bundle, it is a rather cumbersome task, as it involves splitting the space to pieces, and choosing grommets. Instead, we prove Proposition~\ref{topological_model_proposition}, which shows that the class can be calculated directly from the topology of the topological fiber bundle $M/S^1 \rightarrow \mathbb{R}/P$ and its painted subset.
\begin{proof}[Proof of Proposition~\ref{topological_model_proposition}]
    Let $E \rightarrow \mathbb{R}/P$ be an associated painted surface bundle. By Proposition~\ref{associated_painted_surface_bundle_proposition}, it is homeomorphic to the topological bundle $M/S^1 \rightarrow \mathbb{R}/P$, with an orientation-preserving homeomorphism $\Psi:M/S^1 \rightarrow E$ that respects the projections to $\mathbb{R}/P$, and induces a bijection between the set of non-free orbits in $M/S^1$ and the painted subset of $E$, such that every non-free is sent to a point that is labeled by its isotropy representation.
    
    By Proposition~\ref{classification_of_painted_surface_bundles}, $E$ is isomorphic to the mapping torus $M_f$, where $[f]$ is the surface bundle monodromy of $E$. By composing the homeomorphism $\Psi$ with this isomorphism, we get an orientation-preserving homeomorphism between $M/S^1$ and $M_f$, that respects the projections to $\mathbb{R}/P$, and induces a bijection between the set of non-free orbits in $M/S^1$ and the painted subset of $M_f$, such that every non-free is sent to a point that is labeled by its isotropy representation.

    Hence, we are left with showing that the monodromy invariant, which is defined as the conjugacy class of $\Forgetful([f])$ in $\mcg(\Sigma_{g, C_1,...,C_k})$, is independent of the choice of $f$. This holds because the conjugacy class of $[f]$ in $\mcg^\infty(\Sigma_{g, C_1,...,C_k})$ is independent of $f$ by Proposition~\ref{classification_of_painted_surface_bundles}, and because the $\Forgetful$ map is an isomorphism by Theorem~\ref{mcg_forgetful_isomorphism}.
\end{proof}
\begin{remark}\label{monodromy_definitions_agree}
    Note that in the introduction, for brevity, we define the monodromy invariant using Proposition~\ref{topological_model_proposition}. The proof of Proposition~\ref{topological_model_proposition}, given before this remark, shows that this definition agrees with Definition~\ref{monodromy_invariant_definition}.
\end{remark}

\section{Construction of spaces}\label{construction_section}
In this section, we construct a four-dimensional tight circle-valued Hamiltonian $S^1$-manifold for every choice of values for the group of periods, genus invariant, isotropy data invariant, Dusitermaat-Heckman invariant, and the monodromy invariant. Together with Proposition~\ref{uniqueness_theorem_phi_diffeo}, this gives a complete classification of four-dimensional tight circle-valued Hamiltonian $S^1$-manifolds up to $\Phi$-diffeomorphisms between their quotients.

The main theorem of this section is the following:
\begin{theorem}\label{existence_theorem_phi_diffeo}
    Suppose that $P \subset \mathbb{R}$ is a non-trivial discrete subgroup, $c_{\DuHe}$ is a positive real, and $g$ is a non-negative integer. Furthermore, let $(C_1,...,C_k):=((n_1,a_1),...,(n_k,a_k))$ be an unordered list of coprime residue classes, such that the sum $\sum_{j=1}^k \frac{a_j}{n_j}$ is an integer, and let $[f]$ be an element in the mapping class group $\mcg(\Sigma_{g, C_1,...,C_k})$.
    
    Then, there exists a four-dimensional tight circle-valued Hamiltonian $S^1$-manifold with group of periods $P$, Duistermaat-Heckman constant $c_{\DuHe}$, genus invariant $g$, isotropy data invariant $C_1,...,C_k$, and monodromy invariant the conjugacy class of $[f]$.
\end{theorem}
We begin by constructing a proper Hamiltonian $S^1$-manifold for every choice of the invariants, to model the neighborhood of a level set.
\begin{lemma}[Local Existence]\label{local_existence_lemma}
    Let $c_{\DuHe} \in \mathbb{R}_{>0}$ be a positive constant, let $g \ge 0$ be a non-negative integer, and let $C_1,...,C_k$ be a finite unordered list of coprime residue classes. Assume that the sum $\sum_{j=1}^k \frac{a_j}{n_j}$ is an integer, where each $C_j$ is given by the pair $(n_j, a_j)$.
    
    Then there exists a positive number $\varepsilon > 0$, and a proper Hamiltonian $S^1$-manifold over $(-\varepsilon, \varepsilon) \subset \mathbb{R}$ such that for every $y \in (-\varepsilon, \varepsilon)$:
    \begin{enumerate}
        \item the genus of the reduced space is $g$,
        \item the value of the Duistermaat-Heckman function is $c_{\DuHe}$,
        \item the isotropy data of the level set $\Phi^{-1}(y)$ corresponds to the coprime residue classes $C_1,...,C_k$, and
        \item $y$ is a regular value.
    \end{enumerate}
\end{lemma}
\begin{proof}
    This is a special case of Proposition 9.10 in~\cite{tall_existence}. We set the skeleton mentioned in their statement to be a union of $k$ intervals, each one painted by the isotropy representation that corresponds to the coprime residue class $C_j$, and we set the Duistermaat-Heckman function in their statement to be the constant function $f = c_{\DuHe}$ which is compatible with the skeleton by Theorem~\ref{duistermaat_heckman_thm}.
\end{proof}
Since we only consider spaces of dimension four with no fixed points, Lemma~\ref{local_existence_lemma} can also be proved directly without quoting the general statement from~\cite{tall_existence}. We give a simple proof below, constructing the space by taking the product of a Seifert fibration and an interval. For a comprehensive review of Seifert fibrations, see~\cite{seifert_fibrations}.
\begin{proof}[Alternative Proof of Lemma~\ref{local_existence_lemma}]
    Let $c_{\DuHe}$, $g$, and $(n_j, a_j)$ be given as above. Let $b$ be given by
    \begin{equation}\label{definition_of_integer_in_seifert}
        b = - \sum_{j=1}^k \frac{a_j}{n_j}.
    \end{equation}
    By the assumption on the sum, $b$ is an integer. 
    
    Take a Seifert fibration $\pi: E \rightarrow F$ whose Seifert invariant is
    \begin{equation*}
        \{g;(1,b),(n_1, a_1),...,(n_k, a_k)\}.
    \end{equation*}
    Then by Equation~\eqref{seifert_euler_equation} and Equation~\eqref{definition_of_integer_in_seifert}, the Seifert Euler number $e(E \rightarrow F)$ is zero.
    \\
    Next, we want to construct a symplectic form on $E \times (-\varepsilon, \varepsilon)$, such that $\iota_X \omega = dt$, where $X$ is the generating vector field of the $S^1$ action of the fibration, and $t$ is the coordinate of $(-\varepsilon, \varepsilon)$. Let $\omega_F$ be a symplectic form on the orbifold $F$, with total area $\frac{1}{2\pi}c_{\DuHe}$. Choose a connection one-form $\alpha$ on the orbi-bundle $E \rightarrow F$, i.e., a one-form that satisfies $\alpha(X) = 1$ and $L_X \alpha = 0$. We define a two-form on $E \times (-\varepsilon, \varepsilon)$ by
    \begin{equation*}
        \omega := (\pi \circ \projection_E)^*\omega_F + d(t\projection_E^*\alpha),
    \end{equation*}
    where $\projection_E:E\times (-\varepsilon, \varepsilon) \rightarrow E$ is the projection map to $E$. It is a symplectic form for small enough $\varepsilon$. Furthermore, $\omega$ and $dt$ indeed generate the $S^1$ action of the fibration since $\iota_X \omega = \alpha(X)dt = dt$.
    
    Moreover, $2\pi \int_{(E \times \{0\})/S^1} \omega_{\reduced} = c_{\DuHe}$ by the choice of $\omega_F$. Since the Seifert Euler number of each level set is $0$, the Duistermaat-Heckman function has derivative $0$, and thus it must be constant with value $c_{\DuHe}$. Hence, $(E\times (-\varepsilon, \varepsilon), \omega, t)$ is a proper Hamiltonian $S^1$-manifold that satisfy the needed properties.
\end{proof}
\begin{remark}
    Another possible approach for proving Lemma~\ref{local_existence_lemma} is to realize the space as a restriction of a compact Hamiltonian $S^1$-manifold. For spaces with empty isotropy data, one can take the restriction of an appropriate ruled manifold with a Hamiltonian circle action (see~\cite{periodic_hamiltonians}, and in particular Definition 6.13 and Figure 30). Similarly, spaces with non-empty isotropy data can always be constructed as restrictions of a ruled manifold with a finite series of equivariant blowups (see Section 7 in~\cite{periodic_hamiltonians} for a description of which equivariant blowups can be performed, and how they change the decorated graph).
\end{remark}

For gluing pieces constructed by Lemma~\ref{local_existence_lemma}, we will use the local uniqueness theorem for complexity one spaces, proved in~\cite{centered_hamiltonians}. The original statement of the theorem is:
\begin{lemma}[Theorem 1 in~\cite{centered_hamiltonians}] \label{local_uniqueness_lemma}
    Let $(M, \omega, \Phi, U)$ and $(M', \omega', \Phi', U)$ be complexity one spaces. Assume that their Duistermaat-Heckman measures are the same, and that their genus
    and isotropy data over a point $y \in U$ are the same. Then there exists a neighborhood of the point $y$ over which the spaces are isomorphic.
\end{lemma}
Recall that an isomorphism of complexity one spaces is an equivariant symplectomorphism that respects the momentum maps.

The following stronger result for tall spaces follows easily from their proof:
\begin{lemma}[Strengthened version of Theorem 1 in~\cite{centered_hamiltonians}] \label{local_uniqueness_lemma_with_lifting}
    Let $(M, \omega, \Phi, U)$ and $(M', \omega', \Phi', U)$ be grommeted tall complexity one spaces. Assume that their Duistermaat-Heckman measures are the same, and that their genus and isotropy data over a point $y \in U$ are the same. Let $\Sigma$ and $\Sigma'$ be the fibers at $y$ of the associated painted surface bundles. Let $f: \Sigma \rightarrow \Sigma'$ be an orientation-preserving diffeomorphism that sends each labeled point to a labeled point with the same label.
    
    Then there exists a neighborhood $V$ of the point $y$, and an isomorphism $\varphi:M|_V \rightarrow M'|_V$, such that the induced map between $\Sigma$ and $\Sigma'$ is isotopic to $f$, relative to the labeled points.
\end{lemma}
\begin{proof}
    First, we sketch the proof of Lemma~\ref{local_uniqueness_lemma} from~\cite{centered_hamiltonians}. Choose some rigid map $g: \Sigma \rightarrow \Sigma'$, that is, an orientation-preserving diffeomorphism that induces a bijection between the labeled points of $\Sigma$ and the labeled points of $\Sigma'$, and sends each labeled point to a labeled point with the same label, and such that the induced map on each grommet is a rotation of $\mathbb{C}$ (see Definition 10.2 and Lemma 10.3 in~\cite{centered_hamiltonians}). By Proposition 11.1 in~\cite{centered_hamiltonians}, the rigid map $g$ extends to a $\Phi$-diffeomorphism $G:M/S^1 \rightarrow M'/S^1$. By Corollary 9.9, Proposition 3.3, and Proposition 4.2 in~\cite{centered_hamiltonians}, the map $G$ lifts to a $\Phi$-$T$-diffeomorphism, which is smoothly isotopic to an equivariant symplectomorphism $F:M \rightarrow M'$, through $\Phi$-$T$-diffeomorphisms.

    Now, let $f: \Sigma \rightarrow \Sigma'$ be an orientation-preserving diffeomorphism that sends each labeled point to a labeled point with the same label. By applying Lemma~\ref{milnor_linearization} to each labeled point, we get a rigid map $f': \Sigma \rightarrow \Sigma'$, maybe after shrinking our grommets, and it is smoothly isotopic to $f$, relative to the labeled points. Thus, we can apply the proof above to $f'$, and get an equivariant symplectomorphism $F:M \rightarrow M'$, that induces a map between $\Sigma$ and $\Sigma'$ which is isotopic to $f'$, and therefore to $f$, as needed.
\end{proof}
We are ready to proof Theorem~\ref{existence_theorem_phi_diffeo}.
\begin{proof}[Proof of Theorem~\ref{existence_theorem_phi_diffeo}]
    Throughout the proof, whenever $f:X \rightarrow \mathbb{R}$ is a function, and $V$ is an open interval in $\mathbb{R}$, we will denote $X|_{f^{-1}(V)}$ by $X|_V$.

    Let $P$, $c_{\DuHe}$, $g$, $C_1,...,C_k$, and $[f]$ be as in the statement of the theorem. We will show that there exists a four-dimensional tight circle-valued Hamiltonian $S^1$-manifold with group of periods $P$, Duistermaat-Heckman constant $c_{\DuHe}$, genus invariant $g$, isotropy data invariant $C_1,...,C_k$, and monodromy invariant equals to the conjugacy class of $[f]$.
    
    By Lemma~\ref{local_existence_lemma}, we have a proper Hamiltonian $S^1$-manifold $M$ over the interval $I := (-\varepsilon, \varepsilon) \subset \mathbb{R}$, with genus $g$, isotropy data $C_1,...,C_k$, no fixed points, and Duistermaat-Heckman function equals to $c_{\DuHe}$ at every level set.

    Let $M'$ be a copy of $M$, whose Hamiltonian is shifted by $\frac{2}{3}\varepsilon$. It is defined over the interval $I' = (-\frac{1}{3}\varepsilon, \frac{5}{3}\varepsilon)$. By Lemma~\ref{local_uniqueness_lemma}, there exists a small neighborhood $V:=(a, b)$ of $\frac{2}{3}\varepsilon$, and an equivariant symplectomorphism $\varphi$ between $M|_V$ and $M'|_V$. By gluing $M|_{(-\varepsilon, b)}$ and $M'|_{(a, \frac{5}{3}\varepsilon)}$ with the map $\varphi$, we get a proper Hamiltonian $S^1$-manifold over the interval $(-\varepsilon, \frac{5}{3}\varepsilon)$, with the same genus, isotropy data, and Duistermaat-Heckman measure as above.

    By repeatedly gluing copies of $M$, as described in the last paragraph, we construct a proper Hamiltonian $S^1$-manifold $N$ over $J := (-\varepsilon, \tau + \varepsilon)$, where $\tau$ is the positive generator of $P$. Let $E \rightarrow J$ be the associated painted surface bundle to $N$. It follows from Lemma~\ref{trivialize_painted_surface_bundle} that $E$ is isomorphic to the painted surface bundle $\Sigma_{g, C_1,...,C_k} \times J$. Let $F:E \rightarrow \Sigma_{g, C_1,...,C_k} \times J$ be an isomorphism.
    
    Choose grommets around every non-free orbit at the level sets of $N$ at $0$ and at $\tau$. Then there exists some $\delta > 0$ such that $N|_{(-\delta, \delta)}$ and $N|_{(\tau -\delta, \tau + \delta)}$ are grommeted.
    
    Let $f \in \Diff_+(\Sigma_{g, C_1,...,C_k})$ be a diffeomorphism in the given class $[f]$.
    By Lemma~\ref{local_uniqueness_lemma_with_lifting}, there exists a small number $0 < c < \delta$, and an equivariant symplectomorphism $\Psi$ between $N|_{(-c, c)}$ and $N|_{(\tau - c, \tau + c)}$, such that the induced map
    \begin{equation*}
        \pi_1 \circ F \circ \overline{\Psi} \circ F^{-1}|_{\Sigma_{g, C_1,...,C_k} \times \{0\}} : \Sigma_{g, C_1,...,C_k} \rightarrow \Sigma_{g, C_1,...,C_k}
    \end{equation*}
    is isotopic to $f$, relative to the labeled points, where $\pi_1: \Sigma_{g, C_1,...,C_k} \times J \rightarrow \Sigma_{g, C_1,...,C_k}$ is the projection map, and $\overline{\Psi}:E|_{(-c, c)} \rightarrow E|_{(\tau - c, \tau + c)}$ is given the induced map by $\Psi$.
    
    By gluing $N|_{(-c, \tau + c)}$ to itself with $\Psi$, we get a four-dimensional tight circle-valued Hamiltonian $S^1$-manifold $N_{\Psi}$, whose group of periods is $P$. It has genus invariant $g$, isotropy data invariant $C_1,...,C_k$, and Duistermaat-Heckman constant $c_{\DuHe}$. Lastly, by the choice of $\Psi$, the monodromy invariant of $N_\Psi$ is the conjugacy class of $[f]$.
\end{proof}


\section{The fibration invariant}\label{fibration_invariant_section}
In this section, we formalize an additional invariant. For two spaces with $\Phi$-diffeomorphic quotients, the invariant determines whether they are $\Phi$-$T$-diffeomorphic. Intuitively it captures the isomorphism type of the fibration $M \rightarrow M/S^1$, and therefore we call it the \textbf{Fibration Invariant}. Our construction of the fibration invariant is based on a characteristic class introduced by Fintushel in~\cite{fintushel, fintushel2} for circle actions on four-manifolds. In the end of the section, we show which values of the fibration invariant can be attained, by constructing a space for every valid value (see Lemma~\ref{fibration_existence_lemma}). Together with the previous invariants, this gives a full classification of our spaces up to $\Phi$-$T$-diffeomorphisms.
\\

We start by a short discussion on Fintushel's characteristic class. In his paper, he allows fixed points, which results in complications that are not relevant to our paper. We will only present his construction in the special case of circle actions with no fixed points.

Let $M$ be a closed and connected smooth manifold, with a locally-free circle action, and let $\pi:M \rightarrow M/S^1$ be the quotient map. Let $L_i$ be the connected components of the set of exceptional orbits. It follows from the slice theorem that each $L_i$ is a loop of exceptional orbits, all with the same isotropy representation. Let $Q_i$ be a small invariant neighborhood of $L_i$, and let $Q_i^*$ be its image $\pi(Q_i)$ in $M/S^1$. Moreover, let $Q^*$ be the union of the images $Q_i^*$, let $N^* := \left (M/S^1\right ) \setminus \overline{Q^*}$ be the complement of its closure in $M/S^1$, and let $N$ be its preimage $\pi^{-1}(N^*)$ in $M$.

By~\cite[(9.1) and (9.2)]{fintushel2}, the principal $S^1$-bundle $\pi: \pi^{-1}(N) \rightarrow N$ is trivial, and can be extended to all of $M/S^1$ by filling each $\partial Q_i^*$ by the trivial bundle $Q_i^* \times S^1$. The isomorphism type of the resulting bundle depends on the gluing map between $N^*$ and each bundle $Q_i^* \times S^1$. However, in~\cite[(9.1)]{fintushel2}, Fintushel shows how to choose these gluing maps such that they give a principal $S^1$-bundle which is unique up to isomorphism.
To be more precise, choose a meridinal loop ${m_i}^*$ and a complementary loop ${l_i}^*$ in $\partial Q_i^* \cong T^2$. Following \cite[Section 2]{fintushel}, we can choose a lift $m_i$ of $m_i^*$ to $\partial Q_i$, satisfying some relations with respect to the Seifert invariant of the exceptional orbits in $L_i$, whose homology class is independent of the choice of lift. Furthermore, choose an arbitrary lift $l_i$ of $l_i^*$ to $\partial Q_i$.
Then Fintushel shows that if one chooses the gluing maps $\phi_i: \partial Q_i^* \times S^1 \rightarrow \partial Q_i$ such that $\phi_i({m_i}^* \times 0) \sim m_i$ and $\phi_i({l_i}^* \times 0) \sim l_i$, then the resulted principal $S^1$-bundle is independent of the chosen loops, lifts, and gluing maps.
The Chern class of the resulting principal $S^1$-bundle is a class $c$ in the cohomology group $H^2(M/S^1, \mathbb{Z})$. We call $c$ the \textbf{Fintushel class} of the $S^1$ action on $M$. Of course, when the action is free, $c$ is just the usual Chern class of the principal $S^1$-bundle $M \rightarrow M/S^1$.

The following lemma follows from the proof of Theorem 9.6 in~\cite{fintushel2}, or alternatively, in the language of Čech cohomology, from Proposition 4.2 in~\cite{haefliger_salem}. See also Lemma~\ref{lifting_phi_diffeo_condition}, for a similar statement in the formulation of~\cite{haefliger_salem}.
\begin{lemma}\label{lifts_by_fintushel}
    Let $(M, \omega, \mu)$ and $(M', \omega', \mu')$ be four-dimensional tight circle-valued Hamiltonian $S^1$-manifolds, with discrete group of periods $P$. Moreover, let $F:M/S^1 \rightarrow M'/S^1$ be a $\Phi$-diffeomorphism. Then $F$ can be lifted to a $\Phi$-$T$-diffeomorphism if and only if the Fintushel classes satisfy $F^*c' = c$.
\end{lemma}
We will now formalize the fibration invariant, using the Fintushel invariant, and Lemma~\ref{lifts_by_fintushel}.
Let $\mathcal{J}: = (P, g, C_1,\dots,C_k, [f])$ be a choice of invariants, such that the sum $\sum_{j=1}^k \frac{a_j}{n_j}$ is an integer, where $(n_j, a_j)$ are the coprime residue classes corresponding to $C_1,\dots,C_k$. Let $(M, \omega, \Phi)$ be a tight circle-valued Hamiltonian $S^1$-manifold with invariants $\mathcal{J}$, which exists by Theorem~\ref{existence_theorem_phi_diffeo}. The Fintushel class of this space is an element $c$ of $H^2(M/S^1, \mathbb{Z})$. Let $(M', \omega', \Phi')$ be some other space with invariants $\mathcal{J}$, and Fintushel invariant $c' \in H^2(M'/S^1, \mathbb{Z})$. Additionally, let $\overline{\Phi}:M/S^1 \rightarrow \mathbb{R}/P$ and $\overline{\Phi'}:M'/S^1 \rightarrow \mathbb{R}/P$ be the orbital circle-valued momentum maps. We say that the Fintushel invariants $c$ and $c'$ of the spaces $(M, \omega, \Phi)$ and $(M', \omega', \Phi')$ are \textbf{equivalent} if there exists an orientation-preserving homeomorphism $F:M/S^1 \rightarrow M'/S^1$ between their quotients, which respects the isotropy-representations $C_1,\dots,C_k$ of the orbits, and satisfy both $F^*c' = c$ and $\overline{\Phi'} \circ F = \overline{\Phi}$. Now, we define the fibration invariant.
\begin{definition}\label{fibration_invariant_definition}
    Let $(M, \omega, \Phi)$ be a four-dimensional tight circle-valued Hamiltonian $S^1$-manifold, and let $c$ be its Fintushel invariant. The \textbf{Fibration Invariant} of $(M, \omega, \Phi)$ is represented by $c$, and two spaces are said to have the same fibration invariant if their Fintushel invariants are equivalent.
\end{definition}
The following lemma consists of the uniqueness part of the fibration invariant. In other words, it shows that the fibration invariant determines whether two circle valued Hamiltonian $S^1$-manifolds with $\Phi$-diffeomorphic quotients are $\Phi$-$T$-diffeomorphic.
\begin{lemma}\label{fibration_invariant_uniqueness}
    Let $(M, \omega, \Phi)$ and $(M', \omega', \Phi')$ be four-dimensional tight circle-valued Hamiltonian $S^1$-manifolds, with invariants $\mathcal{J}: = (P, g, C_1,\dots,C_k, [f])$. Then they are $\Phi$-$T$-diffeomorphic if and only if they have the same fibration invariant. 
\end{lemma}
Note that the equivalences in the definition of the fibration invariant were defined using homeomorphisms that are not necessarily $\Phi$-diffeomorphisms, hence we cannot immediately deduce Lemma~\ref{fibration_invariant_uniqueness} from Lemma~\ref{lifts_by_fintushel}. The choice to define the equivalences using the more general notion of homeomorphisms will be useful later for reducing the definition of the fibration invariant to a topological one. To prove Lemma~\ref{fibration_invariant_uniqueness}, we need the following technical lemma, which helps us pass from sb-diffeomorphisms to $\Phi$-diffeomorphisms.
\begin{lemma}\label{isotopy_lemma}
    Let $(M, \omega, \Phi)$ and $(M', \omega', \Phi')$ be four-dimensional tight circle-valued Hamiltonian $S^1$-manifolds, and let $E \rightarrow \mathbb{R}/P$ and $E' \rightarrow \mathbb{R}/P$ be associated painted surface bundles, and let $\Psi:M/S^1 \rightarrow E$ and $\Psi':M'/S^1 \rightarrow E'$ be the homeomorphisms given by Proposition~\ref{associated_painted_surface_bundle_proposition}. Furthermore, let $f:E \rightarrow E'$ be an isomorphism of painted surface bundles. Then there is an isotopy $F_t:E \rightarrow E'$ such that
    \begin{enumerate}
        \item $F_0 = f$.
        \item $\Psi'^{-1} \circ F_1 \circ \Psi$ is a $\Phi$-diffeomorphism.
        \item For every $t \in [0,1]$, the map $F_t:E \rightarrow E'$ is a homeomorphism preserving the projection to $\mathbb{R}/P$, and respecting the paint.
    \end{enumerate}
\end{lemma}
We first prove Lemma~\ref{fibration_invariant_uniqueness} using the technical Lemma~\ref{isotopy_lemma}.
\begin{proof}[Proof of Lemma~\ref{fibration_invariant_uniqueness}]
    For the first direction, let $G:M \rightarrow M'$ be a $\Phi$-$T$-diffeomorphism. Then by Lemma~\ref{lifts_by_fintushel}, the induced $\Phi$-diffeomorphism $\overline{G}:M/S^1 \rightarrow M'/S^1$ satisfies $F^*c' = c$, so the Fintushel invariants of the spaces are equivalent, and thus by definition their fibration invariants agree.

    For the other direction, assume that the spaces $(M, \omega, \Phi)$ and $(M', \omega', \Phi')$ have the same fibration invariant. Then by definition, there exists an orientation-preserving homeomorphism $F:M/S^1 \rightarrow M'/S^1$ between their quotients, which respects the isotropy-representations of the orbits, and satisfies both $F^*c' = c$ and $\overline{\Phi'} \circ F = \overline{\Phi}$, where $c$ and $c'$ are the Fintushel invariants, and $\overline{\Phi}$ and $\overline{\Phi'}$ are the orbital momentum maps.
    
    Let $\pi:E \rightarrow \mathbb{R}/P$ and $\pi':E' \rightarrow \mathbb{R}/P$ be associated painted surface bundles, and let $\Psi:M/S^1 \rightarrow E$ and $\Psi':M'/S^1 \rightarrow E'$ be the homeomorphisms given by Proposition~\ref{associated_painted_surface_bundle_proposition}. Then the map $H:E \rightarrow E'$, which completes the diagram
    \begin{equation*}
        \begin{tikzcd}
            M/S^1 \arrow[r, "F"] \arrow[d, "\Psi"'] & M'/S^1 \arrow[d, "\Psi'"] \\
            E \arrow[r, "H"] & E' ,
        \end{tikzcd}
    \end{equation*}
    is an orientation-preserving homeomorphism, which respects the paint, and satisfies both $H^*({(\Psi'^{-1})}^*c') = {(\Psi^{-1})}^*c$ and $\pi' \circ H = \pi$.

    The map $H$ can be identified with a family of paint-respecting homeomorphisms of $\Sigma_g$, parametrized by $\mathbb{R}/P$, by looking at its restrictions to each fiber. This defines an element in the fundamental group $\pi_1(\Homeo(\Sigma_{g,C_1,\dots,C_k}), \id)$ of the space of color-preserving orientation-preserving homeomorphisms. It is unique up to conjugacy. By Theorem~\ref{mcg_forgetful_isomorphism}, the forgetful map $\Diff(\Sigma_{g,C_1,\dots,C_k}) \rightarrow \Homeo(\Sigma_{g,C_1,\dots,C_k})$ induces an isomorphism of mapping class groups. Hence, if we show that the forgetful map $\Diff(\Sigma_{g,C_1,\dots,C_k}) \rightarrow \Homeo(\Sigma_{g,C_1,\dots,C_k})$ induces an isomorphism of fundamental groups when restricted to the identity component, we will deduce that it induces an isomorphism of fundamental groups. For surfaces with no marked points, this follows from~\cite{hamstrom, earle_eels, luke_mason}, where they prove that the forgetful map $\Diff(\Sigma_g) \rightarrow \Homeo(\Sigma_g)$ restricted on the identity component is a homotopy equivalence, and thus induces an isomorphism of fundamental groups. See the remark at the end of~\cite{luke_mason}. We explain here how to use their results to prove by induction that indeed the forgetful map $\Diff(\Sigma_{g,C_1,\dots,C_k}) \rightarrow \Homeo(\Sigma_{g,C_1,\dots,C_k})$ induces an isomorphism between the fundamental groups. We ignore the colors, and only consider the case where all of the special points are fixed, for simplicity. At each step of the induction, we add one fixed point. The above discussion gives the base case, where there are no fixed points.  Let $\Sigma$ be a surface with $x_1,\dots,x_n$ as fixed points, and let $\overline{\Sigma} := \Sigma\setminus\{x_1,\dots,x_n\}$ be the punctured surface. Assume by induction that the forgetful map $\Diff_{x_1,\dots,x_n}(\Sigma_g) \rightarrow \Homeo_{x_1,\dots,x_n}(\Sigma_g)$ induces an isomorphism of fundamental groups, where $\Diff_{x_1,\dots,x_n}(\Sigma_g)$ and $\Homeo_{x_1,\dots,x_n}(\Sigma_g)$ are the groups of diffeomorphisms and homeomorphisms preserving the points $x_1,\dots,x_n$. Let $x \in \overline{\Sigma}$ be an additional point. We have the following Serre fibrations, with the forgetful maps as vertical arrows
    \[
    \begin{tikzcd}
        \Diff_{x, x_1,\dots,x_n}(\Sigma) \arrow[r, hook] \arrow[d]
        & \Diff_{x_1,\dots,x_n}(\Sigma) \arrow[r] \arrow[d]
        & \overline{\Sigma} \arrow[d] \\
        \Homeo_{x, x_1,\dots,x_n}(\Sigma) \arrow[r, hook]
        & \Homeo_{x_1,\dots,x_n}(\Sigma) \arrow[r]
        & \overline{\Sigma},
    \end{tikzcd}
    \]
    where the maps $\Diff_{x_1,\dots,x_n}(\Sigma) \rightarrow \overline{\Sigma}$ and $\Homeo_{x_1,\dots,x_n}(\Sigma) \rightarrow \overline{\Sigma}$ are the evaluation maps at $x$. When the punctured surface $\overline{\Sigma}$ is aspherical, the long exact sequences of homotopy groups gives the following commutative diagram 
    \[
    \begin{tikzcd}
        1 \arrow[r] 
        & \pi_1(\Diff_{x, x_1,\dots,x_n}(\Sigma), \id) \arrow[r] \arrow[d]
        & \pi_1(\Diff_{x_1,\dots,x_n}(\Sigma), \id) \arrow[r] \arrow[d]
        & \pi_1(\overline{\Sigma}, x) \arrow[r] \arrow[d]
        & 1\\
        1 \arrow[r] 
        & \pi_1(\Homeo_{x, x_1,\dots,x_n}(\Sigma), \id) \arrow[r]
        & \pi_1(\Homeo_{x_1,\dots,x_n}(\Sigma), \id) \arrow[r]
        & \pi_1(\overline{\Sigma}, x) \arrow[r]
        & 1,
    \end{tikzcd}
    \]
    and since by induction the middle vertical arrow is an isomorphism, it follows that the vertical arrow $\pi_1(\Diff_{x, x_1,\dots,x_n}(\Sigma), \id) \rightarrow \pi_1(\Homeo_{x, x_1,\dots,x_n}(\Sigma), \id)$ is also an isomorphism as needed. In the case $\overline{\Sigma}$ is not aspherical, i.e., if $\overline{\Sigma}$ is a sphere with no punctures, because $\Diff(S^2) \rightarrow \Homeo(S^2)$ is a homotopy equivalence, then the maps $\pi_2(\Homeo(S^2), \id) \rightarrow \pi_2(S^2, x)$ and $\pi_2(\Diff(S^2), \id) \rightarrow \pi_2(S^2, x)$ have the same images, and it follows again by the long exact sequence of homotopy groups that the map $\Diff_x(S^2) \rightarrow \Homeo_x(S^2)$ induces an isomorphism of fundamental groups. This finishes the induction argument.

    Now, since the forgetful map $\Diff(\Sigma_{g,C_1,\dots,C_k}) \rightarrow \Homeo(\Sigma_{g,C_1,\dots,C_k})$ induces an isomorphism of fundamental groups, we deduce that there exists a family of paint-respecting diffeomorphisms of $\Sigma_g$, parametrized by $\mathbb{R}/P$, whose corresponding class in the fundamental group agrees with the class corresponding to the family of $H$. If chosen smoothly with respect to the parameter, this family of diffeomorphisms defines a diffeomorphism $\tilde H:E \rightarrow E'$, which respects the paint, is homotopic to $H$, and such that $\pi' \circ \tilde H = \pi$. Hence, $\tilde H$ is an isomorphism of painted surface bundles, so by Lemma~\ref{isotopy_lemma} it can be isotoped through paint-preserving homeomorphisms, which respects the projections $\pi$ and $\pi'$, to a homeomorphism $\tilde H':E \rightarrow E'$ such that the map $\overline{G} :M/S^1 \rightarrow M'/S^1$, which completes the diagram
    \begin{equation*}
        \begin{tikzcd}
            M/S^1 \arrow[r, "\overline{G}"] \arrow[d, "\Psi"'] & M'/S^1 \arrow[d, "\Psi'"] \\
            E \arrow[r, "\tilde H'"] & E' ,
        \end{tikzcd}
    \end{equation*}
    is a $\Phi$-diffeomorphism. Since $\tilde H'$ is isotopic to $\tilde H$, which is homotopic to $H$, it follows that $\overline{G}$ satisfies $\overline{G}^*c' = c$. Thus, by Lemma~\ref{lifts_by_fintushel}, $\overline{G}$ can be lifted to a $\Phi$-$T$-diffeomorphism $G:M \rightarrow M'$, and the lemma follows.
\end{proof}
We now prove Lemma~\ref{isotopy_lemma}, which follows directly from technical results in~\cite{tall_uniqueness}.
\begin{proof}[Proof of Lemma~\ref{isotopy_lemma}]
    We choose a finite cover $\{U_j\}$ of $\mathbb{R}/P$ with grommets in each $\Phi^{-1}(U_j)$. Then, it's enough to show that for every open subsets $\overline{V_j} \subset U_j$ and $W_j \subset \overline{V_j}$, we can isotope $f$ to be a $\Phi$-diffeomorphism in $\Phi^{-1}(W_j)$, without changing it outside of $\Phi^{-1}(V_j)$. By~\cite[Proposition 15.10]{tall_uniqueness}, there exists an isotopy $H_t:\Phi^{-1}(U_j) \rightarrow \Phi^{-1}(U_j)$ such that
    \begin{itemize}
        \item $H_0 = f$.
        \item $H_1$ is locally sb-rigid in $\Phi^{-1}(U_j)$ (see~\cite[Definition 10.1]{tall_uniqueness}).
        \item If $f$ is locally sb-rigid in a neighborhood of $\Phi^{-1}(y)$, for $y \in V_j$, then so is $H_t$ for every $t \in [0,1]$.
    \end{itemize}
    Now, choose a smooth bump function $\rho:U_j \rightarrow \mathbb{R}$ that is $1$ on $W_j$ and $0$ outside of $V_j$. Then the isotopy given by
    \begin{equation*}
        \tilde H_t(x) = H_{\rho(\Phi(x))t}(x)
    \end{equation*}
    is supported in $\Phi^{-1}(V_j)$, and satisfies the following
    \begin{itemize}
        \item $\tilde{H}_0 = f$.
        \item $\tilde{H}_1$ is locally sb-rigid in $\Phi^{-1}(W_j)$.
        \item If $f$ is locally sb-rigid in a neighborhood of $\Phi^{-1}(y)$, for $y \in V_j$, then so is $\tilde{H}_t$ for every $t \in [0,1]$.
    \end{itemize}
    Similarly, by~\cite[Proposition 14.9]{tall_uniqueness}, there exists an isotopy $G_t:\Phi^{-1}(U_j) \rightarrow \Phi^{-1}(U_j)$ such that
    \begin{itemize}
        \item $G_0 = f$.
        \item $G_1$ is locally rigid in $\Phi^{-1}(U_j)$ (see~\cite[Definition 8.12]{tall_uniqueness}).
        \item If $f$ is locally rigid in a neighborhood of $\Phi^{-1}(y)$, for $y \in V_j$, then so is $G_t$ for every $t \in [0,1]$.
    \end{itemize}
    Thus, using the same trick with the bump function, and composing $\tilde{H}_t$ with $\tilde{G}_t$, we get an isotopy $F_t:\Phi^{-1}(U_j) \rightarrow \Phi^{-1}(U_j)$ that satisfies
    \begin{itemize}
        \item $F_0 = f$.
        \item $F_1$ is locally rigid in $\Phi^{-1}(W_j)$.
        \item If $f$ is locally rigid in a neighborhood of $\Phi^{-1}(y)$, for $y \in V_j$, then so is $F_t$ for every $t \in [0,1]$.
    \end{itemize}
    Since every locally rigid map is also a $\Phi$-diffeomorphism, this finishes the proof by composing such isotopies for each $U_j$.
\end{proof}

Next, we wish to determine which values are possible for the fibration invariant. Let $(M, \omega, \Phi)$ be a space with invariants $\mathcal{J}: = (P, g, C_1,\dots,C_k, [f])$. By Proposition~\ref{topological_model_proposition}, the quotient space $M/S^1$ is homeomorphic to the mapping torus
\begin{equation}
    M_f := \bigslant{\Sigma_{g, C_1,\dots,C_k} \times [0, \tau]}{(x,0) \sim (f(x), \tau)},
\end{equation}
where $[f]$ is the monodromy invariant and $\tau$ is the positive generator of the group of periods $P$. We can use such a homeomorphism $\Psi:M_f \rightarrow M/S^1$ to pullback the Fintushel class $c$ to $H^2(M_f, \mathbb{Z})$. The resulted element in $H^2(M_f, \mathbb{Z})$ depends on $\Psi$, but it is unique up to equivalence, where as before, $A$ and $B$ in $H^2(M_f, \mathbb{Z})$ are \textbf{equivalent} if there exists an orientation-preserving homeomorphism $F:M_f \rightarrow M_f$, which respects the paint, and satisfy both $F^*B = A$ and $\pi \circ F = \pi$, where $\pi:M_f \rightarrow \mathbb{R}/P$ is the bundle map. Hence, we get the following description for the fibration invariant, as an equivalence class of elements in $H^2(M_f, \mathbb{Z})$.
\begin{lemma}
    For every space with invariants $\mathcal{J}$, its Fintushel class can be identified with a class in $H^2(M_f, \mathbb{Z})$, which is unique up to equivalence. Furthermore, two spaces with invariants $\mathcal{J}$ have the same fibration invariant if and only if their corresponding classes in $H^2(M_f, \mathbb{Z})$ are equivalent.
\end{lemma}
We wish to determine which elements of $H^2(M_f, \mathbb{Z})$ can represent the fibration invariant of a tight circle-valued Hamiltonian $S^1$-manifold with invariants $\mathcal{J}$. Let $(M, \omega, \Phi)$ be a space with invariants $\mathcal{J}$, and let $c$ be its corresponding element in $H^2(M_f, \mathbb{Z})$. By assumption, the sum $\sum_{j=1}^k \frac{a_j}{n_j}$ is an integer, and by Lemma~\ref{duistermaat_heckman_constant}, the Seifert Euler number of each level set is zero. The Seifert Euler number is given by $e = -b - \sum_{j=1}^k \frac{a_j}{n_j}$, see Equation~\eqref{seifert_euler_equation}, hence we have
\begin{equation}\label{sum_euler_equation}
    b = -\sum_{j=1}^k \frac{a_j}{n_j},
\end{equation}
where $b$ is the integer of the Seifert invariants $\{g;(1,b),(n_1, a_1),\dots,(n_k, a_k)\}$ of the level set.

Let $f_1: H^1(\Sigma_g, \mathbb{Z}) \rightarrow H^1(\Sigma_g, \mathbb{Z})$ be the map induced on cohomology by $f$. By the Mayer--Vietoris sequence, we have the following short exact sequence (see also Equation~\eqref{long_exact_sequence_equation})
\begin{equation}\label{short_exact_sequence_fibration_invariant}
    0 \rightarrow H^1(\Sigma_g, \mathbb{Z})/\image(f_1 - \id) \overset{p_1}{\rightarrow} H^2(M_f, \mathbb{Z}) \overset{p_2}{\rightarrow} H^2(\Sigma_g, \mathbb{Z}) \rightarrow 0,
\end{equation}
where $p_2$ is given by restriction to a fiber. Using Equation~\eqref{sum_euler_equation}, it follows from the construction of the Fintushel class that the restriction $p_2(c)$ of the class $c$ to $H^2(\Sigma_g, \mathbb{Z})$ must satisfy
\begin{equation}
    p_2(c) = b [\Sigma_g],
\end{equation}
where $[\Sigma_g]$ is the fundamental class of $\Sigma_g$ in $H^2(\Sigma_g, \mathbb{Z})$. Hence, if $c$ represents the fibtraion invariant of a tight circle-valued Hamiltonian $S^1$-manifold with invariants $\mathcal{J}$, then $c$ is in the affine subspace
\begin{equation}\label{eq_definition_affine_subspace}
    \mathcal{H}_{\mathcal{J}} := p_2^{-1}(b [\Sigma_g]) \subset H^2(M_f, \mathbb{Z}),
\end{equation}
which is modeled on $H^1(\Sigma_g, \mathbb{Z})/\image(f_1 - \id)$ by Equation~\eqref{short_exact_sequence_fibration_invariant}. Note that if the action is free, i.e., if the isotropy data invariant is trivial, then $b = 0$, and the affine subspace $\mathcal{H}_{\mathcal{J}}$ can be canonically identified with $H^1(\Sigma_g, \mathbb{Z})/\image(f_1 - \id)$. The following lemma shows that every class in $\mathcal{H}_{\mathcal{J}}$ represents the fibration invariant of some tight circle-valued Hamiltonian $S^1$-manifold.
\begin{lemma}\label{fibration_existence_lemma}
    Let $A$ be a cohomology class in $\mathcal{H}_{\mathcal{J}}$. Then there exists a four-dimensional tight circle-valued Hamiltonian $S^1$-manifold with invariants $\mathcal{J}$, whose fibration invariant is represented by $A$.
\end{lemma}
\begin{proof}
    Let $A$ be an element of the subspace $\mathcal{H}_{\mathcal{J}} \subset H^2(M_f, \mathbb{Z})$. By Theorem~\ref{existence_theorem_phi_diffeo}, there exists a circle-valued Hamiltonian $S^1$-manifold $(M, \omega, \Phi)$ whose invariants are $\mathcal{J}$. Using a homeomorphism $\Psi:M/S^1 \rightarrow M_f$ given by Proposition~\ref{topological_model_proposition}, we can identify $A$ with a class $c$ in $H^2(M/S^1, \mathbb{Z})$. Moreover, let $c'$ in $H^2(M/S^1, \mathbb{Z})$ be the Fintushel invariant of $(M, \omega, \Phi)$, and let $\Delta := c - c'$ be the difference between $c$ and $c'$ in $H^2(M/S^1, \mathbb{Z})$. As in Haefliger-Salem~\cite[Proposition 4.2]{haefliger_salem}, we will construct a fibration over $M/S^1$, which is locally isomorphic to $M \rightarrow M/S^1$, and such that the difference between its Fintushel class and the Fintushel class of $M \rightarrow M/S^1$ is exactly $\Delta$. Then we will show that it admits an invariant symplectic form that generates the action with the same circle-valued Hamiltonian, and deduce that it is indeed a tight circle-valued Hamiltonian $S^1$-manifold with fibration invariant represented by $A$.

    Choose local trivializations over a cover $\mathcal{U} = {\{U_i\}}_{i \in I}$ of $M/S^1$, and transition maps $f_{i,j}$ representing the class $c$.
    Let $W \subset \mathbb{R}/P$ be a small open interval, and let $y \in W$ be its middle point. We may choose the cover $\mathcal{U}$ such that there is a subcover $\mathcal{V} = {\{V_j\}}_{j \in J}$ whose union $\bigcup_{j \in J} V_j$ is exactly $\overline{\Phi}^{-1}(W)$. By possibly shrinking some open sets in $\mathcal{U}\setminus \mathcal{V}$, we can assume that every set in $\mathcal{U}\setminus \mathcal{V}$ is disjoint from $\overline{\Phi}^{-1}(y)$. We denote by $W_+$ and $W_-$ the two parts of $W$, below and above $y$. We may also assume that every set in $\mathcal{U} \setminus \mathcal{V}$ is either disjoint from the subcover $\mathcal{V}$, intersects the subcover only in $W_+$, or intersects the subcover only in $W_+$.

    Take the cover $\mathcal{U}$, and add to it an additional copy of every open set in the subcover $\mathcal{V}$, which we will denote as $V'$s in $\mathcal{V}'$. We get a larger cover, and choose the transition map $f_{i, j}$ for every one of the pairs $U_i$ and $U_j$, $U_i$ and $V'_j$, or $V'_i$ and $V'_j$. These transition maps satisfy the cocycle condition, and represent the class $c$.

    We glue by all the transition maps, except for the following ones:
    \begin{itemize}
        \item $f_{i, j}$ between $U_i \in \mathcal{U}\setminus \mathcal{V}$ and $V_j \in \mathcal{V}$, if they have non-trivial intersection in $W_+$.
        \item $f_{i, j}$ between $U_i \in \mathcal{U}\setminus \mathcal{V}$ and $V'_j \in \mathcal{V}'$, if they have non-trivial intersection in $W_-$.
        \item $f_{i, j}$ between $V_i \in \mathcal{V}$ and $V'_j \in \mathcal{V}'$.
    \end{itemize}
    The map $\overline{\Phi}$ from the glued space to $\mathbb{R}/P$ can be regarded as a map to $(-\frac{\varepsilon}{2}, \tau + \frac{\varepsilon}{2}) \subset \mathbb{R}$, where $\tau$ is the positive generator of the group of periods $P$, and $\varepsilon$ is the length of $W$. Moreover, every level set has a neighborhood which is $\Phi$-$T$-diffeomorphic to a circle-valued Hamiltonian $S^1$-piece. Following the proof of Theorem~\ref{existence_theorem_phi_diffeo}, the glued space admits a symplectic form such that $\Phi$ generates the circle action, and such that the Duistermaat-Heckman measure is constant. By gluing the sets in $\mathcal{V}$ to the sets in $\mathcal{V}'$ with the identity maps, we get a fibration over $M/S^1$ which is locally isomorphic to $M \rightarrow M/S^1$, with Fintushel class $c$. However, the $\id$ map does not respect the symplectic structure, so we want to modify it. Since the $\id$ map is a $\Phi$-$T$-diffeomorphism between $\Phi^{-1}(-\frac{\varepsilon}{2}, \frac{\varepsilon}{2})$ and $\Phi^{-1}(\tau - \frac{\varepsilon}{2}, \tau + \frac{\varepsilon}{2})$, we can use~\cite[Proposition 3.3]{centered_hamiltonians} to isotope it to an equivariant symplectomorphism (while not explicitly stated in the proposition, their proof constructs an isotopy). By gluing with the equivariant symplectomorphism instead of the $\id$ map, we get a tight circle-valued Hamiltonian $S^1$-manifold, and its Fintushel invariant is unaffected by the isotopy and thus it is $c$, and in particular the fibration invariant of this space is represented by $c$ in $H^2(M/S^1, \mathbb{Z})$, and thus by $A$ in $H^2(M_f, \mathbb{Z})$.
\end{proof}
Lemma~\ref{fibration_existence_lemma}, together with Lemma~\ref{fibration_invariant_uniqueness}, finishes the classification of spaces with $\Phi$-diffeomorphic quotients up to $\Phi$-$T$-diffeomorphisms.
\\

Determining whether two spaces have the same fibration invariant, involves determining whether their corresponding classes in  $\mathcal{H}_{\mathcal{J}}$ are equivalent. The characterization of these equivalence classes corresponds to the study of an action of the infinite dimensional group of isomorphisms of the mapping torus $M_f$ on $\mathcal{H}_{\mathcal{J}} \subset H^2(M_f, \mathbb{Z})$. Fortunately, one does not need to study this group directly for explicitly describing the orbits of the action. This is because the orbits of the action only depend on the induced maps on cohomology, which can be described as finite dimensional matrices. However, in most cases, this calculation becomes very complicated, and we did not find a nice combinatorial description for the fibration invariant from this point of view.

Nevertheless, we will prove that for free actions, whose monodromy invariant induces the identity map on first cohomology, the fibration invariant is determined by a single natural number which corresponds to the smallest positive integral of the Chern class $c_1$ of the principal circle bundle $M \rightarrow M/S^1$. To formalize this, if $(M, \omega, \Phi)$ is a space with a free action, then we define $K_M \in \mathbb{N}$ by
\begin{equation}\label{min_chern_eq}
    K_M := \min \left\{\int_A c_1\ |\ A \in H_2(M/S^1, \mathbb{Z}), \int_A c_1 > 0 \right\},
\end{equation}
where $c_1$ is the Chern class of the bundle $M \rightarrow M/S^1$. If $c_1$ is trivial, we define $K_M$ to be $0$.
\begin{lemma}\label{lemma_fibration_invariant_for_free_trivial_monodromy}
    Let $\mathcal{J}: = (P, g, \{\}, [f])$ be a choice of invariants with trivial isotropy data invariant, and such that $f:\Sigma_g \rightarrow \Sigma_g$ induces the identity map on $H^1(\Sigma_g, \mathbb{Z})$. Moreover, let $(M, \omega, \Phi)$ and $(M', \omega', \Phi')$ be four-dimensional tight circle-valued Hamiltonian $S^1$-manifolds, with invariants $\mathcal{J}$. Then their fibration invariants are the same if and only if $K_M = K_{M'}$.
\end{lemma}
\begin{proof}
    Since the isotropy data invariant is trivial, the number $b$ in Equation~\eqref{sum_euler_equation} is zero, and therefore $\mathcal{H}_{\mathcal{J}}$ can be identified with $H^1(\Sigma_g, \mathbb{Z})/\image(f_1 - \id)$, where $f_1:H^1(\Sigma_g, \mathbb{Z}) \rightarrow H^1(\Sigma_g, \mathbb{Z})$ is the map induced by $f$. By assumption, $f_1 = \id$, and thus $\image(f_1 - \id) = 0$, and $\mathcal{H}_{\mathcal{J}}$ can be identified with $H^1(\Sigma_g, \mathbb{Z}) \cong \mathbb{Z}^{2g}$. Hence, to describe the pullback of the Fintushel class $c$ by an orientation-preserving homeomorphism $\Psi$ of $M_f$ which preserve the projection to $\mathbb{R}/P$, it is enough to describe the pullback of the corresponding element in $\mathbb{Z}^{2g}$ by the restriction of $\Psi$ on a fiber $\Sigma_g$. This restriction gives an element of the mapping class group $\mcg(\Sigma_g)$. Elements of the mapping class group can be described as symplectic matrices in $\Sp_{2n}(\mathbb{Z})$, where the symplectic structure on $\mathbb{Z}^{2g}$ is defined with algebraic intersection numbers, and it is well known that the map $\mcg(\Sigma_g) \rightarrow \Sp_{2n}(\mathbb{Z})$ is surjective. For example, see~\cite[Chapter six]{farb_margalit}.
    
    It follows that two elements of $H^1(\Sigma_g, \mathbb{Z}) \cong \mathbb{Z}^{2g}$ are related by an orientation-preserving homeomorphism which preserves the projection to $\mathbb{R}/P$, if and only if, they are in the same orbit of the action of $\Sp_{2n}(\mathbb{Z})$ on $\mathbb{Z}^{2g}$. These orbits are easily described as follows: two elements $u, v$ of $\mathbb{Z}^{2g}$ are in the same orbit of the action of $\Sp_{2n}(\mathbb{Z})$ if and only if their greatest common divisors $\gcd(u_1,\dots,u_{2g})$ and $\gcd(v_1,\dots,v_{2g})$ agree. See for example Lemma 1 in Section 5 of~\cite{benoist_symplectic_group}. In particular, these integers correspond precisely to the minimum positive integral of the Chern class of the fibration, and so the lemma follows. 
\end{proof}
The following natural question follows from the proof of Lemma~\ref{lemma_fibration_invariant_for_free_trivial_monodromy}. Is the number $K_M$ given by Equation~\eqref{min_chern_eq} equivalent to the fibration invariant, even when the monodromy invariant is cohomologically non-trivial? In Appendix~\ref{examples_appendix}, we show that the answer to this question is negative, and so in general the description of the fibration invariant is more complicated then just the ``minimal positive integral of the Chern class''. In particular, we show in Subsection~\ref{kodaira_like_subsection}, that for genus invariant $g = 1$, if the monodromy invariant induces the matrix 
\begin{equation*}
    \begin{bmatrix}
    1 & k \\
    0 & 1
  \end{bmatrix}
\end{equation*}
on the first cohomology of $\Sigma_1 \cong T^2$, and the minimal positive integral of the Chern class is $m$, then there are $\gcd(m, k)$ different possible values for the fibration invariant. Hence, for every choice of $k$ and $m$ such that $\gcd(m, k) > 1$, there exist $\Phi$-diffeomorphic spaces with minimal positive integral of the Chern class equal to $m$, that are not $\Phi$-$T$-diffeomorphic. For more information see Subsection~\ref{kodaira_like_subsection} in Appendix~\ref{examples_appendix}.
\begin{remark}\label{b1_trivial_fibration_invariant_remark}
    When $b_1(M/S^1) = 1$, the fibration invariant is trivial (i.e., it only has one possible value). This follows since $\mathcal{H}_{\mathcal{J}}$ is an affine space modeled on $H^1(\Sigma_g, \mathbb{Z})/\image(f_1 - \id)$, and $H^1(\Sigma_g, \mathbb{Z})/\image(f_1 - \id)$ is trivial whenever $b_1(M/S^1) = 1$ (see the proof of Proposition~\ref{condition_for_betti_one}). See also Remark~\ref{b1_trivial_de_rham_invariant_remark}.
\end{remark}


\section{The de Rham invariant}\label{de_rham_invariant_section}
In this section, we formalize the last invariant in this work, the de Rham invariant, which determines whether two $\Phi$-$T$-diffeomorphic spaces are also isomorphic as circle-valued Hamiltonian $S^1$-manifolds. Intuitively, it encodes information on the cohomology class of the symplectic form. We describe the different values of the invariant (see Proposition~\ref{de_rham_invariant_description}), and show which are attainable (see Lemma~\ref{de_rham_invariant_existence_lemma}), by constructing an appropriate symplectic form for every valid choice of value for the de Rham invariant. Together with the previous invariants, this gives a full classification of our spaces up to isomorphism (see Theorem~\ref{existence_theorem} and Theorem~\ref{uniqueness_theorem}).
\\

Let $\mathcal{J}: = (P, g, C_1,\dots,C_k, [f])$ be a choice of invariants as in the previous section, let $c$ be an element in the affine subspace $\mathcal{H}_{\mathcal{J}} \subset H^2(M_f, \mathbb{Z})$ given in Equation~\eqref{eq_definition_affine_subspace}, and let $c_{\DuHe}$ be a positive real number. We denote this choice of invariants by $\mathcal{K} := (P, g, c_{\DuHe}, C_1,\dots,C_k, [f], c)$. Let $(M, \omega, \Phi)$ and $(M', \omega', \Phi')$ be four-dimensional tight circle-valued Hamiltonian $S^1$-manifolds, with invariants $\mathcal{K}$. For every $\Phi$-$T$-diffeomorphism $\Psi:M \rightarrow M'$, the two-form $\Psi^*\omega' - \omega$ is a basic form, which defines a basic cohomology class ${[\Psi^*\omega' - \omega]}_{\basic}$ in $H_{\basic}^2(M, \mathbb{R})$, which induces an element ${[\Psi^*\omega' - \omega]}_{\reduced}$ in $H^2(M/S^1, \mathbb{R})$ (see Lemma~\ref{omega_difference_is_basic}). We say that the symplectic forms of $(M, \omega, \Phi)$ and $(M', \omega', \Phi')$ are \textbf{equivalent} if there exists a $\Phi$-$T$-diffeomorphism $\Psi:M \rightarrow M'$ satisfying ${[\Psi^*\omega' - \omega]}_{\reduced} = 0$ in $H^2(M/S^1, \mathbb{R})$.
\begin{definition}\label{de_rham_invariant_definition}
    The \textbf{de Rham invariant} of a space with invariants $\mathcal{K}$ is the equivalence class of its symplectic form, under the above equivalences.
\end{definition}
\begin{lemma}\label{de_rham_invariant_uniqueness}
    Let $(M, \omega, \Phi)$ and $(M', \omega', \Phi')$ be four-dimensional tight circle-valued Hamiltonian $S^1$-manifolds, with invariants $\mathcal{K}$. Then they are isomorphic if and only if their de Rham invariants agree.
\end{lemma}
\begin{proof}
    If they are isomorphic, then an isomorphism $\Psi:(M, \omega, \Phi) \rightarrow (M', \omega', \Phi')$ is in particular a $\Phi$-$T$-diffeomorphism satisfying $\Psi^*\omega' = \omega$, and thus ${[\Psi^*\omega' - \omega]}_{\reduced} = 0$.

    On the other hand, let $\Psi:M \rightarrow M'$ be a $\Phi$-$T$-diffeomorphism, satisfying
    \begin{equation*}
        {[\Psi^*\omega' - \omega]}_{\reduced} = 0.
    \end{equation*}
    Then by Lemma~\ref{moser_with_same_cohomology} it can be isotoped to an isomorphism.
\end{proof}
Next, we describe the values of the de Rham invariant explicitly. By Theorem~\ref{existence_theorem_phi_diffeo} and Lemma~\ref{fibration_existence_lemma}, there exists a circle-valued Hamiltonian $S^1$-manifold $(M_{\mathcal{K}}, \omega_{\mathcal{K}}, \Phi_{\mathcal{K}})$ whose invariants are $\mathcal{K}$. We will use this space to parameterize the possible equivalence classes in a simple way. This parametrization will depend on the chosen space $(M_{\mathcal{K}}, \omega_{\mathcal{K}}, \Phi_{\mathcal{K}})$. We do not know how to formalize this description without this choice.

Let $\omega$ be an invariant two-form on $M_{\mathcal{K}}$, such that $\Phi_{\mathcal{K}}$ generates the same circle action, and such that the space $(M_{\mathcal{K}}, \omega, \Phi_{\mathcal{K}})$ also has invariants $\mathcal{K}$. Then as mentioned above, $\omega - \omega_{\mathcal{K}}$ is a basic form, and it defines an element ${[\omega - \omega_{\mathcal{K}}]}_{\reduced}$ in $H^2(M_{\mathcal{K}}/S^1, \mathbb{R})$. Let $\overline{\Phi}:M_{\mathcal{K}}/S^1 \rightarrow \mathbb{R}/P$ be the induced map by $\Phi$ on the quotient. Then we have the Mayer--Vietoris sequence associated to the cover ${\overline{\Phi}}^{-1}(U), {\overline{\Phi}}^{-1}(V)$, where $U$ and $V$ cover $\mathbb{R}/P$:
\begin{equation}\label{mayer_vietoris_raw}
    \begin{aligned}
        \dots\ & \overset{\Delta}{\rightarrow}H^1({\overline{\Phi}}^{-1}(U) \cap {\overline{\Phi}}^{-1}(V), \mathbb{R}) \overset{d^*}{\rightarrow}
        H^2(M_{\mathcal{K}}/S^1, \mathbb{R}) \overset{\rho}{\rightarrow} \\
        & \overset{\rho}{\rightarrow} H^2({\overline{\Phi}}^{-1}(U), \mathbb{R}) \oplus H^2({\overline{\Phi}}^{-1}(V), \mathbb{R}) \overset{\Delta}{\rightarrow}\dots
    \end{aligned}
\end{equation}
We can then extract the following short exact sequence from it (see also Equation~\eqref{long_exact_sequence_equation}):
\begin{equation}\label{mayer_vietoris_short_exact_sequence}
    0 \rightarrow H^1(\Sigma_g, \mathbb{R})/{\image(f_1-\id)} \overset{\overline{d^*}}{\rightarrow} H^2(M_{\mathcal{K}}/S^1, \mathbb{R}) \overset{\rho}{\rightarrow} H^2(\Sigma_g, \mathbb{R}) \rightarrow 0,
\end{equation}
where $\overline{d^*}$ is induced by $d^*$. Because both spaces $(M_{\mathcal{K}}, \omega_{\mathcal{K}}, \Phi_{\mathcal{K}})$ and $(M_{\mathcal{K}}, \omega, \Phi_{\mathcal{K}})$ have the same Duistermaat-Heckman constant, the class ${[\omega - \omega_{\mathcal{K}}]}_{\reduced}$ restricts to zero on the fibers, and therefore it projects to zero through the map $\rho$. Hence, the cohomology class ${[\omega - \omega_{\mathcal{K}}]}_{\reduced}$ can be regarded as a class in the subgroup $\ker(\rho) \cong \image(\overline{d^*})$, which is isomorphic to $H^1(\Sigma_g, \mathbb{R})/{\image(f_1-\id)}$.

It follows from the above discussion that every invariant symplectic form $\omega$ on $M_{\mathcal{K}}$, such that $(M_{\mathcal{K}}, \omega, \Phi_{\mathcal{K}})$ has the same six invariants, corresponds to an element in the vector space $H^1(\Sigma_g, \mathbb{R})/{\image(f_1-\id)}$. Note that every two symplectic forms that correspond to the same element in $H^1(\Sigma_g, \mathbb{R})/{\image(f_1-\id)}$ are equivalent, by the identity map. Thus, the values of the de Rham invariant can be identified with equivalence classes of elements in the vector space $H^1(\Sigma_g, \mathbb{R})/{\image(f_1-\id)}$. Here, two classes $A$ and $B$ in $H^1(\Sigma_g, \mathbb{R})/{\image(f_1-\id)}$ are $\textbf{equivalent}$ if and only if there exists a $\Phi$-$T$-diffeomorphism $\Psi:M_{\mathcal{K}} \rightarrow M_{\mathcal{K}}$, such that
\begin{equation}\label{vanish_equation}
    {[\Psi^*\omega_{\mathcal{K}} - \omega_{\mathcal{K}}]}_{\reduced} + \overline{\Psi}^*(\overline{d^*}(B)) - \overline{d^*}(A) = 0
\end{equation}
in $H^2(M_{\mathcal{K}}/S^1, \mathbb{R})$, where $\overline{\Psi}:M_{\mathcal{K}}/S^1 \rightarrow M_{\mathcal{K}}/S^1$ is the induced $\Phi$-diffeomorphism on the quotient by $\Psi$. To summarize the above, we deduce the following description of the de Rham invariant, which depends on the choice of $(M_{\mathcal{K}}, \omega_{\mathcal{K}}, \Phi_{\mathcal{K}})$.
\begin{proposition}\label{de_rham_invariant_description}
    Let $(M ,\omega, \Phi)$ be a space with invariants $\mathcal{K}$. Then there exists an element $A$ of $H^1(\Sigma_g, \mathbb{R})/{\image(f_1-\id)}$, and a $\Phi$-$T$-diffeomorphism $\Psi:M_{\mathcal{K}} \rightarrow M$, such that $\Psi^*\omega - \omega_{\mathcal{K}} = \overline{d^*}(A)$. The element $A$ is unique up to equivalence, where $A$ and $B$ in $H^1(\Sigma_g, \mathbb{R})/{\image(f_1-\id)}$ are equivalent if there exists a $\Phi$-$T$-diffeomorphism $\Psi:M_{\mathcal{K}} \rightarrow M_{\mathcal{K}}$ satisfying Equation~\eqref{vanish_equation}.

    Furthermore, two spaces have the same de Rham invariant if and only if their corresponding elements are equivalent.
\end{proposition}
\begin{proof}
    Let $(M, \omega, \Phi)$ be a space with invariants $\mathcal{K}$. By Lemma~\ref{fibration_invariant_uniqueness}, there exists a $\Phi$-$T$-diffeomorphism $\Psi:M_{\mathcal{K}} \rightarrow M$, and we have ${[\Psi^*\omega - \omega_{\mathcal{K}}]}_{\reduced} = \overline{d^*}(A)$ for some element $A$ in the vector space $H^1(\Sigma_g, \mathbb{R})/{\image(f_1-\id)}$. Let $\Psi':M_{\mathcal{K}} \rightarrow M$ be another $\Phi$-$T$-diffeomorphism, then ${[{\Psi'}^*\omega - \omega_{\mathcal{K}}]}_{\reduced}$ defines another element $A'$ in $H^1(\Sigma_g, \mathbb{R})/{\image(f_1-\id)}$. The elements $A$ and $A'$ are equivalent by the $\Phi$-$T$-diffeomorphism ${(\Psi')}^{-1} \circ \Psi:M_{\mathcal{K}} \rightarrow M_{\mathcal{K}}$, and therefore the element $A$ defined by $(M, \omega, \Phi)$ is unique up to equivalence.

    Furthermore, let $(M', \omega', \Phi')$ be another space with corresponding element $B$ in $H^1(\Sigma_g, \mathbb{R})/{\image(f_1-\id)}$. By definition, $(M, \omega, \Phi)$ and $(M', \omega', \Phi')$ have the same de Rham invariant if and only if there exists a $\Phi$-$T$-diffeomorphism $\Psi:M \rightarrow M'$ satisfying ${[\Psi^*\omega' - \omega]}_{\reduced} = 0$. Using $\Phi$-$T$-diffeomorphisms to identify $M$ and $M'$ with $M_{\mathcal{K}}$, this happens if and only if there exists a $\Phi$-$T$-diffeomorphism $\Psi:M \rightarrow M$ satisfying Equation~\eqref{vanish_equation}.
\end{proof}
Lastly, we show that every element of $H^1(\Sigma_g, \mathbb{R})/{\image(f_1-\id)}$ can represent the de Rham invariant of a space with invariants $\mathcal{K}$. This consists of the existence part of the de Rham invariant.
\begin{lemma}\label{de_rham_invariant_existence_lemma}
    Let $A$ be a class in $H^1(\Sigma_g, \mathbb{R})/{\image(f_1-\id)}$. Then there exists an invariant symplectic form $\omega$ on $M_{\mathcal{K}}$, such that
    \begin{equation*}
        {[\omega - \omega_{\mathcal{K}}]}_{\reduced} = \overline{d^*}(A),
    \end{equation*}
    and such that $(M_{\mathcal{K}}, \omega', \Phi_{\mathcal{K}})$ has the same invariants $\mathcal{K}$.
\end{lemma}
\begin{proof}
    Let $A$ be an element in the vector space $H^1(\Sigma_g, \mathbb{R})/{\image(f_1-\id)}$. We wish to show that there exists a basic closed two-form $\beta \in H^2(M_{\mathcal{K}}, \mathbb{R})$ such that
    \begin{enumerate}
        \item The induced form $\beta_{\reduced}$ on the quotient is in the cohomology class $\overline{d^*}(A)$,
        \item the form $\omega_{\mathcal{K}} + \beta$ is symplectic, and
        \item the space $(M_{\mathcal{K}}, \omega_{\mathcal{K}} + \beta, \Phi_{\mathcal{K}})$ has the same invariants $\mathcal{K}$.
    \end{enumerate}
    Since $H^1(\Sigma_g, \mathbb{R})/{\image(f_1-\id)} \cong \ker(\rho) \cong \image(\overline{d^*})$, we can choose an element $B$ in $H^1(\overline{\Phi}^{-1}(U) \cap \overline{\Phi}^{-1}(V), \mathbb{R})$, which projects to $\overline{d^*}(A)$ through the map $d^*$ given by the Mayer-Vietoris sequence~\eqref{mayer_vietoris_short_exact_sequence}.
    Let $\tau$ be a representative for the class $B$, and let $f_U:U \rightarrow \mathbb{R}, f_V:V \rightarrow \mathbb{R}$ be a partition of unity subordinate to the cover $U, V$ of $\mathbb{R}/P$. We have
    \begin{equation}\label{tilde_tau_def}
        \begin{aligned}
        &\overline{d^*}(A) = d^*(B) = d^*([\tau]) = [\tilde \tau], \\
        &\text{where }\tilde \tau := \begin{cases}
            d(f_U\tau) \text{ on $U$}\\
            -d(f_V\tau) \text{ on $V$}.
        \end{cases}
        \end{aligned}
    \end{equation}
    Because $\tau$ is closed, the form $\tilde \tau$ can be written as
    \begin{equation*}
        \tilde \tau = (\frac{df_U}{dx} - \frac{df_V}{dx})d \Phi\wedge \tau,
    \end{equation*}
    where $x$ is the coordinate of $\mathbb{R}/P$. It follows that the restriction of $\tilde \tau$ to each fiber vanishes.
    
    We now verify that the basic form $\pi^*\tilde \tau$ satisfies the three requirements given in the beginning of this proof.
    First, by~\eqref{tilde_tau_def}, $\tilde \tau$ is indeed in the cohomology class $\overline{d^*}(A)$. Second, we wish to show that $\tilde \omega := \omega_{\mathcal{K}} + \pi^*\tilde \tau$ is symplectic. It is a closed form, so we only need to check that it is non-degenerate. Let $p\in M_{\mathcal{K}}$ be a point in the manifold, and let $y := \Phi(p)$ be its image in $\mathbb{R}/P$.
    Because there are no fixed points we can choose two vectors $v_1, v_2$ in $T_p M_{\mathcal{K}}$ that satisfy:
    \begin{equation*}
        \begin{aligned}
            & v_1 := X(p), \\
            & d \Phi(v_2) = 1.
        \end{aligned}
    \end{equation*}
    Let $\pi:M_{\mathcal{K}} \rightarrow M_{\mathcal{K}}/S^1$ be the quotient map, and $\omega_{\reduced}$ be the reduced form on the reduced space $\Phi^{-1}(y)/S^1$. Since $\omega_{\reduced}$ is non-degenerate, we can choose two vectors $v_3, v_4$ in $T_p M_{\mathcal{K}}$ that satisfy:
    \begin{equation*}
        \begin{aligned}
            & d \Phi(v_3) = d \Phi(v_4) = 0, \\
            & \omega_{\reduced}({\pi}_*(v_3), {\pi}_*(v_4)) = 1.
        \end{aligned}
    \end{equation*}
    The vectors $v_1,v_2,v_3,v_4$ form a basis for $T_p M_{\mathcal{K}}$.
    
    As mentioned before, $\tilde \tau$ vanishes on each fiber, hence we have
    \begin{equation*}
        \pi^*\tilde \tau(v_3, v_4) = 0.
    \end{equation*}
    Furthermore, $\iota_X(\pi^*\tilde \tau) = 0$ since $\pi^*\tilde \tau$ is a basic form, and therefore
    \begin{equation*}
        \pi^*\tilde \tau(v_1, v_3) = \pi^*\tilde \tau(v_1, v_4).
    \end{equation*}
    It follows that the form $\tilde \omega$ satisfies $\iota_X \tilde \omega = -d \Phi$, i.e., it generates the action with $\Phi:M_{\mathcal{K}} \rightarrow \mathbb{R}/P$.
    In matrix form, ${\tilde \omega}_p$ can be written as:
    \begin{equation*}
        \begin{bmatrix}
            0  & 1    & 0    & 0 \\
            -1 & 0    & c_1  & c_2 \\
            0  & -c_1  & 0    & 1 \\
            0  & -c_2  & -1   & 0
        \end{bmatrix},
    \end{equation*}
    for some constants $c_1, c_2$, and this is an invertible matrix. Therefore, $\tilde \omega$ is indeed non-degenerate.
    
    Lastly, we wish to show that $(M_{\mathcal{K}}, \tilde \omega, \Phi_{\mathcal{K}})$ has invariants $\mathcal{K}$. Since it is $\Phi$-$T$-diffeomorphic to $(M_{\mathcal{K}}, \omega_{\mathcal{K}}, \Phi_{\mathcal{K}})$, the only invariant that we need to check is the Duistermaat-Heckman constant. Because $\tilde \tau$ vanishes on the fibers, it immediately follows that the reduced spaces have the same areas with respect to $\omega_{\mathcal{K}}$ and $\tilde \omega$, and the proof is done.
\end{proof}
Together, Proposition~\ref{de_rham_invariant_description} and Lemma~\ref{de_rham_invariant_existence_lemma} finishes the classification of four dimensional tight circle-valued Hamiltonian $S^1$-spaces, that are $\Phi$-$T$-diffeomorphic, up to isomorphism (of circle-valued Hamiltonian $S^1$-spaces). In general, it seems complicated to give a general description for the de Rham invariant in terms of ``combinatorial'' invariants. Nevertheless, see Appendix~\ref{examples_appendix} for explicit calculations of the de Rham invariant for some cases.
\begin{remark}\label{b1_trivial_de_rham_invariant_remark}
    When $b_1(M/S^1) = 1$, the de Rham invariant is trivial (i.e., it only has one possible value). This follows because $H^1(\Sigma_g, \mathbb{Z})/\image(f_1 - \id)$ is trivial (see the proof of Proposition~\ref{condition_for_betti_one}). See also Remark~\ref{b1_trivial_fibration_invariant_remark}.
\end{remark}
\begin{remark}\label{basic_vs_de_rham_remark}
    It's important to note that the basic cohomology in the definition of the de Rham invariant cannot be replaced with standard de Rham cohomology. To be more precise, even if $[\omega] = [\omega']$ in de Rham cohomology, it doesn't necessarily follow that ${[\omega - \omega']}_{\reduced}$ vanishes. For example, assume that the $S^1$ action is free, and the Chern class $c_1$ of the principal $S^1$-bundle $M \rightarrow M/S^1$ is non-zero. Then by the Gysin long exact sequence of cohomology, we have that the kernel of the map $\pi^*:H^2(M/S^1, \mathbb{R}) \rightarrow H^2(M, \mathbb{R})$ consists of the multiples of $c_1$. If $\beta$ is a two-form on $M/S^1$ such that $[\beta] = rc_1$ for some $0 \neq r \in \mathbb{R}$, and such that $\omega' := \omega + \pi^*\beta$ is symplectic, then $\omega$ and $\omega'$ have the same de Rham cohomology class, but their difference induces $\beta$ on the quotient, which is not an exact form since $c_1 \neq 0$. Such $\beta$s exist by Lemma~\ref{de_rham_invariant_existence_lemma}. For generic $r$, we then have that $(M, \omega, \Phi)$ and $(M, \omega', \Phi)$ are non-isomorphic. Nevertheless, if the $S^1$ action is free, and the fibration invariant vanishes, i.e., the Chern class $c_1$ is zero, then we do have that $[\omega] = [\omega']$ implies that ${[\omega - \omega']}_{\reduced} = 0$, since the map $\pi^*$ is injective.
\end{remark}

\appendix

\section{Examples}\label{examples_appendix}
In this appendix, we examine some examples of symplectic non-Hamiltonian $S^1$ actions on connected compact four-dimensional symplectic manifolds. For each action we define a tight circle-valued Hamiltonian, and present calculations of the values of the invariants.

First, we will focus on diagonal circle actions on products of symplectic surfaces. There are only two effective symplectic circle actions on compact connected surfaces of fixed area, up to equivariant symplectomorphism. Namely, the circle action that rotates the sphere, and the circle action that rotates one of the coordinates of the torus. It follows that the diagonal actions on $S^2 \times \Sigma_g$, where $g > 1$ are precisely the rotations of $S^2$, which are Hamiltonian and therefore not of our interest (They correspond to Karshon graphs with just two (fat) vertices, see~\cite{periodic_hamiltonians}). Similarly, all diagonal actions on $S^2 \times S^2$ are Hamiltonian and not of our interest.

Therefore, we will consider diagonal actions on $\mathbb{T}^2 \times S^2$ and $\mathbb{T}^2 \times \mathbb{T}^2$. We skip the case of diagonal actions on $\mathbb{T}^2 \times \Sigma_g$, where $g > 1$. Its analysis is similar to the case of a diagonal action on $\mathbb{T}^2 \times S^2$ which acts trivially on the sphere (see Subsection~\ref{sphere_torus_subsection}, with $k=1$ and $l=0$), except for the fibration and de Rham invariants which requires some calculations as in the case of a diagonal action on $\mathbb{T}^2 \times \mathbb{T}^2$.

Later in the section, we will describe the Kodaira-Thurston manifold with a specific free symplectic non-Hamiltonian circle action, and some more related spaces, such as ones whose quotients are $\Phi$-diffeomorphic to Kodaira-Thurston's quotient, but that have different values of the fibration invariant.

\subsection{Diagonal actions on \texorpdfstring{$\mathbb{T}^2 \times S^2$}{T2xS2}}\label{sphere_torus_subsection}
\hfill\\
Let $A>0$ and $B>0$ be positive real numbers. Let $(\mathbb{T}^2 \cong \mathbb{R}^2/\mathbb{Z}^2, Adp \wedge dq)$ be a symplectic torus with area $A$, and let $(S^2, \frac{B}{4\pi}dh \wedge d\theta)$ be a symplectic sphere with area $B$. Here, we denote the coordinates on the torus by $p \in \mathbb{R}/\mathbb{Z}$ and $q \in \mathbb{R}/\mathbb{Z}$, and the cylindrical polar coordinates on the sphere (without its poles) by $\theta \in \mathbb{R}/{2\pi \mathbb{Z}}$ and $h \in (-1, 1)$.

Let $S^1 \cong \mathbb{R}/{2\pi \mathbb{Z}}$ act symplectically and diagonally on the product space $(\mathbb{T}^2 \times S^2, \omega = Adp \wedge dq + \frac{B}{4\pi}dh \wedge d\theta)$. Without loss of generality, it acts by rotating the $q$-coordinate of the torus with velocity $\frac{k}{2\pi}$, and the $\theta$-coordinate of the sphere with velocity $\ell$, for some integers $k$ and $\ell$. More precisely, the action is defined by
\begin{equation*}
    t \cdot (p, q, h, \theta) = (p, q + \frac{k}{2\pi}t, h, \theta + \ell t),
\end{equation*}
for every $t \in S^1$. This is a subgroup of the torus action described in~\cite[Example 5.11]{pelayo_review_paper}.
We assume that $\gcd(k, \ell) = 1$ so that the action will be effective, and that $k \ne 0$ so that the action will be non-Hamiltonian. The maps $(p, q, h, \theta) \mapsto (-p, -q, h, \theta)$ and $(p, q, h, \theta) \mapsto (p, q, -h, -\theta)$ are equivariant symplectomorphisms to the same manifold but with the actions whose velocities are $-\frac{k}{2\pi}, \ell$ and $\frac{k}{2\pi}, -\ell$, respectively. Hence, we can assume without loss of generality that $k > 0$ and $\ell \ge 0$.

The generating vector field $X \in \mathfrak{X}(\mathbb{T}^2 \times S^2)$ of the action is 
\begin{equation*}
    X = \frac{k}{2\pi} \frac{\partial}{\partial q} + \ell \frac{\partial}{\partial \theta}.
\end{equation*}
The one-form $\iota_X \omega$ is equal to
\begin{equation*}
    \iota_X \omega = - \frac{Ak}{2\pi}dp - \frac{B\ell}{4\pi}dh,
\end{equation*}
and its group of periods $P$ is $\frac{Ak}{2\pi} \mathbb{Z}$. By the assumption that $k \ne 0$, the group of periods is non-trivial. Because $P$ is a non trivial discrete subgroup of $\mathbb{R}$, we can define a tight circle-valued Hamiltonian by 
\begin{equation*}
    \Phi(p, q, h, \theta) = \frac{Ak}{2\pi}p + \frac{B\ell}{4\pi}h \mod \frac{Ak}{2\pi}.
\end{equation*}
This makes $(M, \omega, \Phi)$ a four-dimensional tight circle-valued Hamiltonian $S^1$-manifold. Fix some $y \in \mathbb{R}/P$. The level set of $\Phi$ at $y$ is of the form 
\begin{equation*}
    \Phi^{-1}(y) = \{(p, q, h, \theta)| \frac{Ak}{2\pi}p + \frac{B\ell}{4\pi}h = y \mod \frac{Ak}{2\pi}\},
\end{equation*}
and it's diffeomorphic to $S^1 \times S^2$, which is a connected space as expected by Lemma~\ref{fiber_connectivity_of_circle_actions}. The reduced space with respect to $\Phi$, at any value of the circle-valued Hamiltonian, is homeomorphic to $S^2$. Hence, the genus invariant of the space is $0$.

The isotropy data invariant of the space can be calculated from the isotropy representations of the non-free orbits in an arbitrary level set. When $k = 1$, the level sets only have free orbits, and the isotropy data invariant is trivial. Otherwise, every level set has two non-free orbits $\mathcal{O}_1$ and $\mathcal{O}_2$, which project on the sphere to the north and south pole. The stabilizer of these orbits is $\mathbb{Z}_k$, and their isotropy weights are $\ell \mod k$ and $-\ell \mod k$. Let $m$ be the unique integer such that $0 < m < k$, and $m = \ell \mod k$. Then the isotropy data invariant is given by the list $(k, m), (k, k - m)$.

The Duistermaat-Heckman function is constant, as expected from Lemma~\ref{duistermaat_heckman_constant}, and the Duistermaat-Heckman constant is $B/k$. Note that we have
\begin{equation*}
    \vol(M, \omega) = 2\pi\tau c_{\DuHe} = 2\pi \cdot \frac{Ak}{2\pi}\cdot \frac{B}{k} = AB
\end{equation*}
as expected from Equation~\eqref{dh_vol_eq}.

The topological mapping torus given by Proposition~\ref{topological_model_proposition} is a trivial sphere bundle over $\mathbb{R}/P$, with monodromy invariant equals to the conjugacy class of the element $[\id]$ in $\mcg(\Sigma_{0, (k, m), (k, k - m)})$.

Since $b_1(M/S^1)=1$, the fibration and de Rham invariants are trivial (see Remarks~\ref{b1_trivial_fibration_invariant_remark} and~\ref{b1_trivial_de_rham_invariant_remark}).

Using this description, let us characterize which pairs of diagonal actions on $\mathbb{T}^2 \times S^2$ are equivariantly symplectomorphic. Let $A'>0$, $B'>0$ be positive real numbers, and let $k' > 0$ and $\ell' \ge 0$ be integers that satisfy $\gcd(k', \ell') = 1$. Let $S^1$ act on $(\mathbb{T}^2 \times S^2, A'dp \wedge dq + \frac{B'}{4\pi}dh \wedge d\theta)$, by rotating the $q$-coordinate of the torus with velocity $\frac{k'}{2\pi}$, and the $\theta$-coordinate of the sphere with velocity $\ell'$. Then by Theorem~\ref{uniqueness_theorem}, the two actions are equivariantly symplectomorphic if and only if the following properties hold:
\begin{enumerate}
    \item $k = k'$
    \item $\ell \mod k = \pm \ell' \mod k$
    \item $A = A'$
    \item $B = B'$.
\end{enumerate}

\subsection{Diagonal actions on \texorpdfstring{$\mathbb{T}^2 \times \mathbb{T}^2$}{T2xT2}}
\hfill\\
Let $A>0$ and $B>0$ be positive real numbers. Let $\mathbb{T}^2 \times \mathbb{T}^2$ be a product of two tori, and denote its coordinates by $p_1$, $q_1$,$p_2$, and $q_2$ in $\mathbb{R}/\mathbb{Z}$.

Let $S^1$ act symplectically and diagonally on $(\mathbb{T}^2 \times \mathbb{T}^2, \omega = Adp_1 \wedge dq_1 + Bdp_2 \wedge dq_2)$. Without loss of generality, it acts by rotating the $q_1$-coordinate of the left torus with velocity $\frac{k}{2\pi}$, and the $q_2$-coordinate of the right torus with velocity $\frac{\ell}{2\pi}$, for some integers $k$ and $\ell$. The action is defined by
\begin{equation*}
    t \cdot (p_1, q_1, p_2, q_2) = (p_1, q_1 + \frac{k}{2\pi}t, p_2, q_2 + \frac{\ell}{2\pi}t).
\end{equation*}
This is a subgroup of the torus action described in~\cite[Example 5.9]{pelayo_review_paper}.
We assume that $\gcd(k, \ell) = 1$ so that the action will be effective. The generating vector field $X \in \mathfrak{X}(\mathbb{T}^2 \times \mathbb{T}^2)$ of the action is 
\begin{equation*}
    X = \frac{k}{2\pi} \frac{\partial}{\partial q_1} + \frac{\ell}{2\pi} \frac{\partial}{\partial q_2}.
\end{equation*}
The one-form $\iota_X \omega$ is equal to
\begin{equation*}
    \iota_X \omega = -\frac{Ak}{2\pi}dp_1 - \frac{B\ell}{2\pi}dp_2,
\end{equation*}
and its group of periods $P$ is $\frac{Ak}{2\pi} \mathbb{Z} + \frac{B\ell}{2\pi} \mathbb{Z}$. Note that $P$ is a discrete subgroup of $\mathbb{R}$ if and only if $Ak$ and $B\ell$ are rationally dependent. Assume that $P$ is a discrete subgroup of $\mathbb{R}$, then there exists some $0 < \tau \in \mathbb{R}$ such that $P = \tau \mathbb{Z}$. Then, we can define a tight circle-valued Hamiltonian by 
\begin{equation*}
    \Phi(p_1, q_1, p_2, q_2) = \frac{Ak}{2\pi}p_1 + \frac{B\ell}{2\pi}p_2 \mod \tau.
\end{equation*}
This makes $(M, \omega, \Phi)$ a four-dimensional tight circle-valued Hamiltonian $S^1$-manifold. Fix some $y \in \mathbb{R}/P$. The level set of $\Phi$ at $y$ is of the form 
\begin{equation*}
    \Phi^{-1}(y) = \{(p_1, q_1, p_2, q_2)| \frac{Ak}{2\pi}p_1 + \frac{B\ell}{2\pi}p_2 = y \mod \tau\},
\end{equation*}
and it's diffeomorphic to $\mathbb{T}^3$. As in the previous subsection, the level sets are connected, as expected by Lemma~\ref{fiber_connectivity_of_circle_actions}. The reduced spaces with respect to $\Phi$ are homeomorphic to $\mathbb{T}^2$. Hence, the genus invariant of the space is $1$.

The action is free, hence there are no non-free orbits and the isotropy data invariant is trivial. The Duistermaat-Heckman function is constant, as expected from Lemma~\ref{duistermaat_heckman_constant}, and the Duistermaat-Heckman constant is equal to $\frac{AB}{2\pi\tau}$ by Equation~\eqref{dh_vol_eq}. Moreover, since $\Phi^{-1}(y) \rightarrow \Phi^{-1}(y) / S^1$ is a principal $S^1$-bundle, its Seifert Euler number is just the regular Chern number, which is zero since the bundle is trivial, as expected from Lemma~\ref{duistermaat_heckman_constant}.

The topological mapping torus given by Proposition~\ref{topological_model_proposition} is a trivial torus bundle over $\mathbb{R}/P$, with no labeled points. Therefore, the monodromy invariant is the conjugacy class of the element $[\id]$ in $\mcg(\Sigma_1)$.

Since the monodromy invariant is trivial, it follows by Lemma~\ref{lemma_fibration_invariant_for_free_trivial_monodromy} that the fibration invariant can be identified with the minimal positive integral $K_M$ of the Chern class of the principal $S^1$-bundle $M \rightarrow M/S^1$, or zero if the Chern class is trivial. In this case, the Chern class is indeed trivial, and so the fibration invariant is identified with the value $K_M = 0$. See also the next subsection for a more in-depth description of the values of the fibration invariant in this case.

Recall that by the description of the de Rham invariant in Proposition~\ref{de_rham_invariant_description}, it is measured relatively to some fixed space. Thus, it is pointless to describe the value of this invariant if we don't compare it with other spaces. See the next subsection for the description of the values of the de Rham invariant for this case.

\subsection{Free actions with genus invariant \texorpdfstring{$1$}{1} and trivial monodromy}
\hfill\\
In this subsection, we present explicit calculations of the values of the fibration and de Rham invariants, for spaces with genus invariant $1$, trivial isotropy data invariant, and trivial monodromy invariant.

Let $(M, \omega, \Phi)$ be a circle-valued Hamiltonian $S^1$-manifold. Assume that its group of periods is $\tau \mathbb{Z}$, its Duistermaat-Heckman constant is $c_{\DuHe}$, its genus invariant is $1$, and that its isotropy data invariant and monodromy invariant are trivial.

We wish to calculate the possible values of the fibration invariant, for spaces with these invariants. These corresponds to the $\Phi$-$T$-diffeomorphism types of spaces whose quotients are $\Phi$-diffeomorphic to the quotient of the space from the previous subsection. As mentioned in the previous subsection, by Lemma~\ref{lemma_fibration_invariant_for_free_trivial_monodromy}, the fibration invariant be identified with the minimal positive integral $K_M$ of the Chern class, and therefore the values of the fibration invariant can be identified with $\mathbb{N}$ in this case. Nevertheless, for completeness, we wish to describe the values of the fibration invariants directly from definition as well.

The fibration invariant takes values in the affine subspace $\mathcal{H}_{\mathcal{J}} \subset H^2(M_f, \mathbb{Z})$, which is canonically isomorphic to $H^1(\Sigma_g, \mathbb{Z})/\image(f_1 - \id)$ since the action is free (see Equation~\eqref{eq_definition_affine_subspace}). Here, $H^1(\Sigma_g, \mathbb{Z})/\image(f_1 - \id) \cong \mathbb{Z}^2$, hence the values of the fibration invariant are represented by vectors in $\mathbb{Z}^2$. The mapping torus $M_f$ is homeomorphic to $T^2 \times \mathbb{R}/P$, with coordinates $p, q, t$ in $\halfopen{0}{1}$. If the fibration invariant is represented by the vector $(n_1, n_2)$ in $\mathbb{Z}^2$, then the space is diffeomorphic to the principal $S^1$-bundle over $M_f$ with Chern class $n_1 dp \wedge dt + n_2 dq \wedge dt$.

Two vectors $(n_1, n_2)$ and $(n_3, n_4)$ in $\mathbb{Z}^2$ are equivalent if and only if there exists an orientation-preserving homeomorphism of $T^2 \times \mathbb{R}/P$, which respects the projection $T^2 \times \mathbb{R}/P \rightarrow \mathbb{R}/P$, and such that it pullbacks the form $n_1 dp \wedge dt + n_2 dq \wedge dt$ to the form $n_3 dp \wedge dt + n_4 dq \wedge dt$. In particular, for every matrix in $\SL_2(\mathbb{Z})$, we can find a homeomorphism that pullbacks $n_1 dp \wedge dt + n_2 dq \wedge dt$ to $m_1 dp \wedge dt + m_2 dq \wedge dt$, where
\begin{equation*}
    (m_1, m_2) = A \cdot (n_1, n_2),
\end{equation*}
and therefore $(n_1, n_2)$ and $(n_3, n_4)$ are equivalent if and only if they are in the same orbit of $\SL_2(\mathbb{Z})$. By Bezout Theorem, This happens if and only if $\gcd(n_1, n_2) = \gcd(n_3, n_4)$ (and this happens if and only if the minimal positive integrals of their Chern classes agree).
\\

Next, we describe explicit calculations of the de Rham invariant. For simplicity, we assume that the minimal positive integral $K_M$ of the Chern class vanishes, i.e., that the fibration invariant is trivial. To describe the different values of the de Rham invariant, using Proposition~\ref{de_rham_invariant_description}, we fix the space $N:=(\mathbb{T}^2 \times \mathbb{T}^2, \omega = 2\pi \tau dp_1 \wedge dq_1 + \frac{c_{\DuHe}}{2\pi}dp_2 \wedge dq_2)$ with the $S^1$ action that rotates the $q_1$ coordinate. This space has trivial isotropy data invariant, monodromy invariant, and fibration invariant, has genus invariant $1$, group of periods $\tau \mathbb{Z}$, and Duistermaat-Heckman constant $c_{\DuHe}$. By Proposition~\ref{de_rham_invariant_description}, for every other space with these invariants, its de Rham invariant can be represented by a class $A$ in $H^1(\mathbb{T}^2, \mathbb{R})/\image(f_1-\id) \cong H^1(\mathbb{T}^2, \mathbb{R})$. This class is unique up to equivalence, where $A$ and $B$ in $H^1(\mathbb{T}^2, \mathbb{R})$ are equivalent if there exists a $\Phi$-$T$-diffeomorphism $\Psi:N \rightarrow N$ such that ${[\Psi^*\omega - \omega]}_{\reduced} + \Psi^*\overline{d^*}B - \overline{d^*}A$ vanishes in $H^2(N/S^1, \mathbb{R})$ (see also Equation~\eqref{vanish_equation}).

Since the action is free and the fibration invariant vanishes, it follows by Remark~\ref{basic_vs_de_rham_remark} that two symplectic forms satisfy $[\omega] = [\omega']$ if and only if they satisfy ${[\omega' - \omega]}_{\reduced} = 0$. Hence, $A$ and $B$ are equivalent if and only if there exists a $\Phi$-$T$-diffeomorphism $\Psi:N \rightarrow N$ such that $[\Psi^*(\omega + \overline{d^*}B)] = [\omega + \overline{d^*}A]$. Note that $\overline{d^*}A = A \wedge dp_1$ and $\overline{d^*}B = B \wedge dp_1$. The pullback map $F^*:H^1(N, \mathbb{R}) \rightarrow H^1(N, \mathbb{R})$ induced on first cohomology by $\Psi$ can be described with respect to the basis $dq_1, dp_1, dp_2, dq_2$ by the following matrix:
\begin{equation}\label{matrix_eq}
\begin{pmatrix}
1 & 0 & 0 & 0\\
x & 1 & e & f\\
y & 0 & a & b\\
z & 0 & c & d
\end{pmatrix},
\end{equation}
where $\begin{pmatrix}
a & b\\
c & d
\end{pmatrix}$ is a matrix in $SL_2(\mathbb{Z})$ and $x, y, z, e, f$ are some integers. Note that the second column is $\begin{pmatrix}
0 \\
1 \\
0 \\
0
\end{pmatrix}$ since $\Psi^*d \Phi = d \Phi$, and that the first row is $\begin{pmatrix}
1 & 0 & 0 & 0
\end{pmatrix}$ since $\Psi$ is equivariant.

We want to describe equivalence classes of vectors in $H^1(\mathbb{T}^2, \mathbb{R}) \cong \mathbb{R}^2$, under the equivalences by $\Phi$-$T$-diffeomorphisms. Let $A := (s, t)$ be such a vector, and let $\omega'$ be the corresponding symplectic form given by $\omega + \overline{d^*}A = \omega + (sdp_2 + tdq_2) \wedge dp_1$. Then $\Psi^* \omega'$ is given by $\omega + (s'dp_2 + t'dq_2) \wedge dp_1$ where
\begin{equation}\label{equivalence_de_rham_eq}
    \begin{aligned}
    s' &= sa + tc + c_{\DuHe}(ec - af) + y\tau \\
    t' &= sb + td + c_{\DuHe}(ed - bf) + z\tau,
    \end{aligned}
\end{equation}
and we have that $(s, t)$ is equivalent to $(s', t')$. Hence, the possible values for the de Rham invariant are the equivalence classes under the equivalences of Equation~\eqref{equivalence_de_rham_eq}, for all matrices as in Equation~\eqref{matrix_eq}. We note that if $r \in \mathbb{R}$ is an irrational number, which is not in the subgroup $\mathbb{Q} + c_{\DuHe} \mathbb{Q} + \tau \mathbb{Q}$, then the vectors $(1, r)$ and $(r, 1)$ are not equivalent. It follows that two spaces might be $\Phi$-$T$-diffeomorphic, with symplectic forms having the same values of integrals (i.e., such that the groups of periods $P(\omega)$ and $P(\omega')$ are the same), but still not be equivariantly symplectomorphic since they have different de Rham invariants.

\subsection{The Kodaira-Thurston manifold}\label{kodaira_like_subsection}
\hfill\\
In this subsection, we describe the Kodaira-Thurston manifold, together with a particular symplectic circle action, and present calculations of the values of the invariants. We also discuss other related spaces. The Kodaira-Thurston manifold was first studied by Kodaira~\cite{kodaira_manifold}, and was later rediscovered by Thurston~\cite{thurston_manifold}, where it was given as the first known example of a symplectic manifold that does not admit a Kähler form. 

Most commonly, the description of the Kodaira-Thurston manifold is given by taking a quotient of $(\mathbb{R}^4, dp_1 \wedge dq_1 + dp_2 \wedge dq_2)$ by an action of a discrete group that preserves the symplectic form. There is a very similar construction that is based on a gluing map. We will describe the second construction, since it can be easily related to the mapping torus given by Proposition~\ref{topological_model_proposition}, and consequently to the monodromy invariant. See~\cite[Example 3.1.17]{intro_to_symp_topology} for a description of both constructions.

Let $M := (\mathbb{R}/\mathbb{Z})\times \mathbb{R} \times {(\mathbb{R}/\mathbb{Z})}^2$ be a manifold with coordinates $p_1, q_1, p_2, q_2$, and a symplectic form $\omega = dp_1 \wedge dq_1 + dp_2 \wedge dq_2$. The map $f:M \rightarrow M$ defined by:
\begin{equation*}
    F(p_1, q_1, p_2, q_2) = (p_1, q_1 + 1, p_2 + q_2, q_2)
\end{equation*}
preserves the symplectic form. Hence, the quotient manifold $N:= M/{\sim}$, where $x \sim F(x)$, is a compact symplectic manifold. It admits a symplectic circle action whose generating vector field $X \in \mathfrak{X}(N)$ is:
\begin{equation*}
    X = \frac{1}{2\pi} \frac{\partial}{\partial q_1}.
\end{equation*}
This is a subcircle action of the torus action described in~\cite[Example 5.10]{pelayo_review_paper}. The one-form $\iota_X \omega$ is equal to
\begin{equation*}
    \iota_X \omega = -\frac{1}{2\pi} dp_1,
\end{equation*}
and its group of periods $P$ is $\frac{1}{2\pi} \mathbb{Z}$. The action is generated by the following tight circle-valued Hamiltonian
\begin{equation*}
    \Phi(p_1, q_1, p_2, q_2) = p_1 \mod \frac{1}{2\pi}, 
\end{equation*}
whose level sets are diffeomorphic to $\mathbb{T}^3$. The reduced spaces are homeomorphic to $\mathbb{T}^2$, and therefore the genus invariant is $1$. 

The action is free, and therefore there are no exceptional orbits, meaning that the isotropy data invariant is trivial. The Duistermaat-Heckman function is constant, as expected from Lemma~\ref{duistermaat_heckman_constant}, and by Equation~\eqref{dh_vol_eq} the Duistermaat-Heckman constant is equal to $1$. Moreover, just like in the previous example, the Seifert fibration is a principal $S^1$-bundle, hence its Seifert Euler number is just the regular Chern number, which is zero (see Lemma~\ref{duistermaat_heckman_constant}).

The monodromy invariant is the conjugacy class of the element $[f]$ in $\mcg(\Sigma_1)$, where the map $f:\mathbb{T}^2 \rightarrow \mathbb{T}^2$ is given by restricting the map $F$ to the last 2 coordinates $p_2$ and $q_2$:
\begin{equation}\label{equation_gluing_kt}
    f(p, q) = (p + q, q).
\end{equation}

Since the action is free, the affine subspace $\mathcal{H}_{\mathcal{J}}$ is canonically isomorphic to the space $H^1(\mathbb{T}^2, \mathbb{Z})/\image(f_1-\id) \cong \mathbb{Z}$, and the fibration invariant is represented by integers in this space. Since the principal $S^1$-bundle $M \rightarrow M/S^1$ is trivial, its Chern class vanishes, and therefore so does the Fintushel class. It follows that the fibration invariant of our space is represented by the zero class in $H^1(\mathbb{T}^2, \mathbb{Z})/\image(f_1-\id) \cong \mathbb{Z}$.

The other possible values for the fibration invariant correspond to other spaces that are $\Phi$-diffeomorphic to $M$, but that are not $\Phi$-$T$-diffeomorphic to it. These are given by equivalence classes of elements in $H^1(\mathbb{T}^2, \mathbb{Z})/\image(f_1-\id) \cong \mathbb{Z}$. Let $\alpha$ be a generator of $H^1(\mathbb{T}^2, \mathbb{Z})/\image(f_1-\id) \cong \mathbb{Z}$, and let $\beta$ be its image in $H^2(M_f, \mathbb{Z})$ by the map
\begin{equation}\label{embed_map_cohomology_eq}
    H^1(\mathbb{T}^2, \mathbb{Z})/\image(f_1-\id) \overset{j}{\rightarrow} H^2(M_f, \mathbb{Z})
\end{equation}
given by the short exact sequence in Equation~\ref{short_exact_sequence_fibration_invariant}, where $M_f$ is the mapping torus of $f$. Then for every integer $m$ with $|m| > 1$, the classes $\alpha$ and $m\alpha$ cannot be equivalent, since the integrals of $\beta$ and $m\beta$ over $H_2(M/S^1, \mathbb{Z})$ are different. On the other hand, the function $(p, q) \mapsto (-p, -q)$ commutes with the gluing map given in Equation~\eqref{equation_gluing_kt}, and is orientation preserving. Hence, applying it fiber-wise defines an orientation-preserving homeomorphism, which respects the projections to $\mathbb{R}/P$, and that sends $\alpha$ to $-\alpha$. Hence $\alpha$ and $-\alpha$ are equivalent. It follows that the values of the fibration invariant can be identified with natural numbers $\mathbb{N}$, where every class is identified with the minimal positive integral of the Chern class $K_M$. In this case, we get a similar description as in Proposition~\eqref{lemma_fibration_invariant_for_free_trivial_monodromy}, even though the monodromy invariant is non-trivial.
\\

Lastly, we discuss a related space, for which the fibration invariant does not correspond to the minimal positive integral of the Chern class. Let $(M, \omega, \Phi)$ be a space with genus invariant $1$, trivial isotropy data invariant, and monodromy invariant represented by the map
\begin{equation}\label{equation_gluing_other}
    f(p, q) = (p + kq, q),
\end{equation}
for some $k > 1$. Then the fibration invariant is represented by values in the space $\mathcal{H}_{\mathcal{J}} \subset H^2(M_f, \mathbb{Z})$, which is isomorphic to $H^1(\mathbb{T}^2, \mathbb{Z})/\image(f_1-\id) \cong \mathbb{Z}_k \times \mathbb{Z}$. The quotient space $M/S^1$ is diffeomorphic to $M_f$, and for each element $A$ of $\mathcal{H}_{\mathcal{J}} \cong \mathbb{Z}_k \times \mathbb{Z}$, by Lemma~\ref{fibration_existence_lemma}, there exists a circle-valued Hamiltonian $S^1$-space $M_A$ modeled over $M/S^1$ such that its Chern class corresponds under the diffeomorphism $M/S^1 \cong M_f$ to the class $A$. The values of the fibration invariant are equivalence classes of elements in $\mathcal{H}_{\mathcal{J}} \cong \mathbb{Z}_k \times \mathbb{Z}$. Two classes $A$ and $B$ are equivalent if there exists an orientation-preserving homeomorphism $\Psi:M_f \rightarrow M_f$, which respects the projections to $\mathbb{R}/P$, and such that $\Psi^*B = A$.

To be more explicit, for a vector $(a, b)$ in $\mathbb{Z}_k \times \mathbb{Z}$, the corresponding class in $\mathcal{H}_{\mathcal{J}} \subset H^2(M_f, \mathbb{Z})$ is given by $\alpha \wedge [d\Phi]$, where $\alpha$ is a class in $H^1(\mathbb{T}^2, \mathbb{Z})$ whose corresponding vector in $\mathbb{Z}^2$ is the vector $(\tilde a, b)$, for some $\tilde a$ such that $\tilde a \mod k = a$. Since $\Psi^*d\Phi = d\Phi$, we only care about how $\Psi$ pullbacks $\alpha$. This can be calculated by restricting $\Psi$ on a fiber, and the induced map on $H^1(T^2, \mathbb{Z})$ corresponds to a symplectic matrix which is invariant under conjugation by $f_1:H^1(T^2, \mathbb{Z}) \rightarrow H^1(T^2, \mathbb{Z})$. These correspond to matrices of the form
\begin{equation}\label{matrix_conj_eq}
    \begin{pmatrix}
        1 & x\\
        0 & 1
    \end{pmatrix},
\end{equation}
where $x \in \mathbb{Z}$. One can then check that $(a, b)$ and $(a', b')$ in $\mathbb{Z}_k \times \mathbb{Z}$ are equivalent, i.e., can be related by a matrix of the form in Equation~\eqref{matrix_conj_eq}, if and only if $b = b'$ and $a = a' \mod \gcd(b, k)$. Therefore, if we choose for example the vectors $(0, k)$ and $(1, k)$, then they are not equivalent, even though the minimal positive integral of the Chern class for the spaces with fibration invariants represented by $(0, k)$ and $(1, k)$ is equal to $k$.

\section{Smooth and topological mapping class groups} \label{mcg_appendix}
In this appendix, we prove a classical theorem that the forgetful map from the smooth mapping class group to the topological mapping class group is an isomorphism. While this theorem is regarded as known in the field, we weren't able to find a full proof for the case of surfaces with marked points. If one doesn't allow for marked points, but possibly allows for boundary, the surjectivity of the map can be seen from Hatcher's~\cite[Theorem B]{hatcher2013kirby}, and the injectivity can be seen from results of Earle-Eels~\cite[Corollary 1E]{earle_eels} and Earle-Schatz~\cite[Theorem 1D]{earle_schatz} (together with Lemma~\ref{smoothing_paths_lemma}). We show that Hatcher's proof of the surjectivity can be easily adapted to surfaces with marked points, and that the injectivity for surfaces with marked points can be deduced from the injectivity for surfaces with boundary.
\\

We denote by $\Sigma^b_{g,n}$, a compact connected orientable 2-dimensional smooth manifold of genus $g$, with $b$ boundary components, and $n$ marked points. When the surface has no marked points ($n = 0$), we will write $\Sigma^b_{g}$ instead of $\Sigma^b_{g,0}$. We denote by $\Diff_+(\Sigma^b_{g,n})$ the group of orientation-preserving diffeomorphisms that fix the marked points and the boundary, with the $C^\infty$-topology. We denote by $\Homeo_+(\Sigma^b_{g,n})$ the group of orientation-preserving homeomorphisms that fix the marked points and the boundary, with the compact--open topology.

The topological mapping class group is
\begin{equation*}
    \mcg(\Sigma^b_{g,n}) := \Homeo_+(\Sigma^b_{g,n}) /{\sim},
\end{equation*}
where $f \sim h$ if they are isotopic relative to the marked points and the boundary (that is, homotopic through homeomorphisms that fix the marked points and the boundary).

Similarly, we define the smooth mapping class group by
\begin{equation*}
    \mcg^\infty(\Sigma^b_{g,n}) := \Diff_+(\Sigma^b_{g,n}) /{\sim},
\end{equation*}
where $f \sim h$ if they are smoothly isotopic relative to the marked points and the boundary (that is, smoothly homotopic through diffeomorphisms that fix the marked points and the boundary).

We wish to prove the following classical theorem:
\begin{theorem}\label{mcg_forgetful_isomorphism}
    The forgetful map $\mcg^\infty(\Sigma^b_{g,n}) \rightarrow \mcg(\Sigma^b_{g,n})$ is an isomorphism.
\end{theorem}

This theorem can be split into two parts - showing the surjectivity and the injectivity of the forgetful map.
\begin{proposition}[Surjectivity] \label{surjectivity}
    Let $\Phi: \Sigma^b_{g,n} \rightarrow \Sigma^b_{g,n}$ be an orientation-preserving homeomorphism that fixes the marked points and the boundary. Then it is isotopic, relative to the marked points and the boundary, to a diffeomorphism that fixes the marked points and the boundary.
\end{proposition}
\begin{proposition}[Injectivity] \label{injectivity}
    Let $\Phi: \Sigma^b_{g,n} \rightarrow \Sigma^b_{g,n}$ be an orientation-preserving diffeomorphism that fixes the marked points and the boundary. Assume that it is isotopic to the identity map, relative to the marked points and the boundary. Then, it is smoothly isotopic to the identity map, relative to the marked points and the boundary.
\end{proposition}

\subsection{Surjectivity}
When the surface has no marked points, Proposition~\ref{surjectivity} follows from the handle smoothing theorem (see~\cite[Section I.3]{kirby_sib}). More recently, Hatcher~\cite{hatcher2013kirby} has written an expository paper on this topic, proving the handle smoothing theorem, and deducing the proposition for surfaces with boundary but no marked points. We will show that it is straight-forward to adapt his proof to allow for marked points. First, we give Hatcher's statement for the handle smoothing theorem, with a slight modification.
\begin{theorem}[Handle Smoothing Theorem]
    Let $\Sigma^b_g$ be a smooth surface.
    \begin{enumerate}
        \setcounter{enumi}{-1}
        \item A topological embedding $\mathbb{R}^2 \rightarrow \Sigma^b_g$ can be isotoped to be a smooth embedding in a neighborhood of the origin, staying fixed outside a larger neighborhood of the origin, and at the origin.
        \item A topological embedding $D^1 \times \mathbb{R} \rightarrow \Sigma^b_g$ which is a smooth embedding near $\partial D^1 \times \mathbb{R}$ can be isotoped to be a smooth embedding in a neighborhood of $D^1 \times \{0\}$, staying fixed outside a larger neighborhood of $D^1 \times \{0\}$ and near $\partial D^1 \times \mathbb{R}$.
        \item A topological embedding $D^2 \rightarrow \Sigma^b_g$ which is a smooth embedding near $\partial D^2$ can be isotoped to be a smooth embedding on all of $D^2$, staying fixed near $\partial D^2$.
    \end{enumerate}
\end{theorem}
The three different parts in the statement of the handle smoothing theorem correspond to $0$-handle smoothing, $1$-handle smoothing, and $2$-handle smoothing, respectively. The difference from Hatcher's statement is that we require that the isotopy of the $0$-handle fixes the origin of $\mathbb{R}^2$. In Hatcher's paper, he did not explicitly require this, but his proof shows this, and he even mentions this at the end of the proof. 
\begin{proof}[Proof of Proposition~\ref{surjectivity}]
    We follow Hatcher's proof whose steps are as follows:
    \begin{enumerate}
        \item Choose a triangulation of $\Sigma_{g, n}^b$.
        \item Isotope the homeomorphism such that it is a smooth embedding in a neighborhood of each vertex of the triangulation, using the $0$-handle smoothing theorem.
        \item Isotope the resulted homeomorphism such that it is a smooth embedding in a neighborhood of each edge of the triangulation, using the $1$-handle smoothing theorem.
        \item Isotope the resulted homeomorphism such that it is a smooth embedding in each face, using the $2$-handle smoothing theorem.
    \end{enumerate}
    The isotopies that smooth the edges and faces of the triangulation are supported away from a neighborhood of the vertices of the triangulation. Moreover, as we noted above, the isotopy given by the $0$-handle smoothing theorem fixes the origin of $\mathbb{R}^2$, and thus the isotopy that smooths the vertices also fixes the vertices. Therefore, the vertices of the triangulation are fixed by the final isotopy constructed in the proof. Hence, by choosing a triangulation such that the marked points lie on the vertices, we can construct the wanted isotopy in the exact same way as in the original proof.
\end{proof}

\subsection{Injectivity}
First, we introduce the theorems of~\cite{earle_eels} and~\cite{earle_schatz}. Let $D(\Sigma^b_{g, n})$ be the subgroup of $\Diff_+(\Sigma^b_{g, n})$ consisting of diffeomorphisms that are homotopic to the identity, relative to the marked points and the boundary.
\begin{theorem}[Corollary 1E in~\cite{earle_eels}] \label{EE67Thm}
    Let $\Sigma_g$ be a closed oriented surface of genus $g$.
    \begin{enumerate}
        \item If $g=0$, meaning that $\Sigma_g$ is the sphere $S^2$, then $D(S^2)$ has $SO(3)$ as strong deformation retract.
        \item If $g=1$, meaning that $\Sigma_g$ is the torus $\mathbb{T}^2$, then $D(\mathbb{T}^2)$ has $\mathbb{T}^2$ as strong deformation retract.
        \item If $g \ge 2$, then $D(\Sigma_g)$ is contractible.
    \end{enumerate}
\end{theorem}
\begin{theorem}[Theorem 1D in~\cite{earle_schatz}] \label{ES70Thm}
    Let $\Sigma^b_g$ be an oriented surface of genus $g$. If $b > 0$, then $D(\Sigma^b_g)$ is contractible.
\end{theorem}
For smoothing paths in the diffeomorphisms group, we will need the following lemma:
\begin{lemma}\label{smoothing_paths_lemma}
    Let $\Phi:M \rightarrow M$ be a diffeomorphism that fixes a subset $A$ of $M$ that contains the boundary $\partial M$. Assume that there is a continuous path of diffeomorphisms $F_t$ in $\Diff_+(M)$, connecting $\Phi$ to the identity map, such that each $F_t$ fixes $A$. Then $\Phi$ is smoothly isotopic to the identity map, relative to $A$.
\end{lemma}
\begin{proof}
    By writing $\Phi$ as a composition of finitely many diffeomorphisms that fix $A$ which are very close to the identity map in the $C^\infty$-topology, we are left with showing the following claim: If a diffeomorphism $F$, that fixes some set $A$, is very close to the identity map in the $C^\infty$ topology, then it is smoothly isotopic to the identity map, relative to $A$. This is a classical result, see for example the proof of~\cite[Proposition 5.3.2]{kupers2019lectures}.
\end{proof}
\begin{corollary}\label{homotopic_means_isotopic_without_marked_points}
    Let $\Phi: \Sigma^b_g \rightarrow \Sigma^b_g$ be an orientation-preserving diffeomorphism that fixes the boundary. Assume that it is homotopic to the identity map, relative to the boundary. Then, it is smoothly isotopic to the identity map, relative to the boundary.
\end{corollary}
\begin{proof}
    Immediate by Theorem~\ref{EE67Thm}, Theorem~\ref{ES70Thm}, and Lemma~\ref{smoothing_paths_lemma}.
\end{proof}
We will also need the following lemma:
\begin{lemma}\label{smooth_neighborhood_of_boundary}
    Let $\Sigma^b_{g, n}$ be a surface with $b>0$.
    Let $S$ be one of the boundary components in $\Sigma^b_{g, n}$, and let $U \cong S \times [0,1]$ be a collar neighborhood of $S$, disjoint from the marked points. Let $\Phi:\Sigma^b_{g, n} \rightarrow \Sigma^b_{g, n}$ be an orientation-preserving diffeomorphism that fixes the marked points, the boundary, and $U$. Assume that $\Phi$ is smoothly isotopic to the identity map, relative to the marked points and the boundary.
    Then $\Phi$ is smoothly isotopic to the identity map, relative to the marked points, the boundary, and $U$.
\end{lemma}
\begin{proof}
    Let $D_U(\Sigma^b_{g, n})$ be the group of diffeomorphisms of $\Sigma^b_{g, n}$ which are homotopic to the identity map, relative to the marked points, the boundary, and $U$.
    Let $\Emb(U, \Sigma^b_{g, n})$ be the space of smooth embeddings of $U$ into $\Sigma^b_{g, n}$, that fix $S$. By~\cite{cerf_dissertation}, the space $\Emb(U, \Sigma^b_g)$ is contractible, and the fibration
    \begin{equation*}
        D_U(\Sigma^b_{g, n}) \rightarrow D(\Sigma^b_{g, n}) \rightarrow \Emb(U, \Sigma^b_{g, n}),
    \end{equation*}
    implies that $D(\Sigma^b_{g, n})$ and $D_U(\Sigma^b_{g, n})$ have the same homotopy--type. By assumption, $\Phi$ is in the connected component of $\id$ in $D(\Sigma^b_{g, n})$, and therefore it is also in the connected component of $\id$ in $D_U(\Sigma^b_{g, n})$. The result then follows from Lemma~\ref{smoothing_paths_lemma}.
\end{proof}

We will need the following lemma from Frab and Margalit's book to relate between marked points and boundary components, through Dehn twists:
\begin{lemma}[Proposition 3.19 in~\cite{farb_margalit}] \label{farb_lemma}
    Let $\Sigma^{b-1}_{g,n+1}$ be the surface obtained from a surface $\Sigma^b_{g,n}$ by capping the boundary component $\beta$ with a once-marked disk. We denote by $p$ the marked point in this disk. Let $\capmap : \mcg(\Sigma^{b}_{g,n}) \rightarrow \mcg(\Sigma^{b-1}_{g,n+1})$ be the homomorphism induced by the inclusion of $\Sigma^{b}_{g,n}$ into $\Sigma^{b-1}_{g,n+1}$. Then the following sequence is exact:
    \begin{equation*}
        1 \rightarrow \langle T_\beta \rangle \rightarrow \mcg(\Sigma^{b}_{g,n}) \xrightarrow{\capmap} \mcg(\Sigma^{b-1}_{g,n+1}) \rightarrow 1,
    \end{equation*}
    where $T_\beta$ is a Dehn twist around $\beta$.
\end{lemma}
The last ingredient that we need for proving the proposition is the following Lemma:
\begin{lemma}\label{milnor_linearization}
    Let $M$ be an orientable manifold, and let $f:M \rightarrow M$ be an orientation-preserving diffeomorphism that fixes a point $p$ in $M$. Then there exists a smooth isotopy $F_t$ of $M$ that satisfies the following properties:
    \begin{enumerate}
        \item $F_0$ = f.
        \item $F_t$ fixes $p$ for every $t$.
        \item There exists a small neighborhood $U$ of $p$, such that the map $F_t = f$ outside of $U$, for every $t$.
        \item $F_1$ is the identity map in a smaller neighborhood of $p$.
    \end{enumerate}
\end{lemma}
\begin{proof}
    Choose a chart $\varphi: U \rightarrow \mathbb{R}^n$ around the point $p$, such that $p$ is sent to the origin. There exists some positive number $r > 0$ such that the image of the disc $D^n_r \subset \mathbb{R}^n$ through the map $f \circ \varphi^{-1}:\mathbb{R} \rightarrow M$ is contained in $U$. Hence, we get the following smooth map
    \begin{equation*}
        g:=\varphi \circ f \circ \varphi^{-1}:D^n_r \rightarrow \mathbb{R}^n,
    \end{equation*}
    which is a diffeomorphism onto its image.
    
    The differential $dg_0$ is a matrix in $\GL^+_n(\mathbb{R})$, so we can choose a smooth family of matrices $A_t$ such that $A_0 = I$ and $A_1 = dg_0^{-1}$. Applying this family of matrices to $\mathbb{R}^n$ defines a smooth time-dependent vector field $X_t$ on $\mathbb{R}^n$. Multiplying $X_t$ by a small bump function which equals to $1$ around the origin, and integrating the new vector field, we get a smooth isotopy relative to the origin, from $g:D^n_r\rightarrow \mathbb{R}^n$ to a map $\tilde g:D^n_r\rightarrow \mathbb{R}^n$ that satisfies $d\tilde g_0 = I$, $\tilde g(0) = 0$. By its construction, $\tilde{g}$ coincides with $g$ outside of $\varphi(U')$, where  the closure of $U'$ in $M$ is contained in $U$.
    
    Now, we use a construction from Milnor's book (see the proofs of Lemma 1 and 2 in page 34 of~\cite{milnor_book}). We define a family of diffeomorphisms by
    \begin{equation*}
    G_t(x) = \begin{cases}
			\frac{\tilde g(xt)}{t}, & \text{for } 0 < t \le 1\\
            x, & \text{for } t = 0
		 \end{cases}.
    \end{equation*}
    By Milnor, $G_t$ is a smooth isotopy from $\tilde g$ to $\id$. There exists a small neighborhood $V$ of the point $p$ such that $V$ is contained in $\subset G_t(U)$ for every $t$. We define a smooth time-dependent vector field on $V$ by
    \begin{equation*}
        Y_t(x) = \frac{\partial}{\partial s}|_{s=t}G_{s}(G_t^{-1}(x)).
    \end{equation*}
    We multiply $Y_t$ by a bump function which equals to $1$ around the origin, and whose support is the closure of a neighborhood $W$ of the origin, such that the closure of $W$ is contained in $V \cap \varphi(U')$. We get a vector field $Z_t$ which is supported in the closure of $W$, and that coincides with $Y_t$ in a small neighborhood $O$ of the origin. Integrating $Z_t$ to a flow, we get a smooth isotopy relative to the origin, from $\tilde g$ to a new map that coincides with $g$ outside of $\varphi(U')$, and coincides with the identity map in $O$.
    
    Applying the smooth isotopies that we constructed inside of the neighborhood $U$ to the map $f$, we get the wanted isotopy.    
\end{proof}

\begin{proof}[Proof of Proposition~\ref{injectivity}]
    If there are no marked points, the proposition follows from Corollary~\ref{homotopic_means_isotopic_without_marked_points}.
    
    Assume by induction that we've shown the proposition for all surfaces $\Sigma^{b}_{g,k}$ with $k < n$. Let $f$ be a diffeomorphism in $\Diff_+(\Sigma^{b}_{g,n})$. Assume that $[f] = [\id] \in \mcg(\Sigma^{b}_{g,n})$, i.e., that $f$ is isotopic to the identity map, relative to the boundary and the marked points. We wish to show that $f$ is smoothly isotopic to the identity map, relative to the boundary and the marked points.
    
    Let $p$ be a marked point of $\Sigma^{b}_{g,n}$. By Lemma~\ref{milnor_linearization}, There exists a small neighborhood $U$ of $p$, a slightly larger neighborhood $U \subset V$, and a smooth isotopy relative to $p$ and the complement of $V$ from $f$ to a map $\tilde f$, such that $\tilde f$ is the identity map in $U$.
    
    We remove a once-marked disc around $p$ whose closure is contained in $U$, to get the subsurface $\Sigma^{b+1}_{g,n-1}$. Let $A \subset \Sigma^{b+1}_{g,n-1}$ be the union of the set of marked points, the boundary, and  $U \cap \Sigma^{b+1}_{g,n-1}$. Restricting $\tilde f$ to the subsurface $\Sigma^{b+1}_{g,n-1}$, we get a diffeomorphism $g:\Sigma^{b+1}_{g,n-1} \rightarrow \Sigma^{b+1}_{g,n-1}$ that fixes $A$.

    We wish to show that it is possible to choose $\tilde f$ such that the resulting $g$ is isotopic to the identity map, relative to $A$. Then, by induction, and Corollary~\ref{smooth_neighborhood_of_boundary}, we will get a smooth isotopy relative to $A$, from $g$ to the identity map on $\Sigma^{b+1}_{g,n-1}$. This smooth isotopy will extend to a smooth isotopy from $\tilde f$ to the identity on $\Sigma^{b}_{g,n}$, relative to the boundary and the marked points. Together with the previous smooth isotopy from $f$ to $\tilde f$, this will finish the proof, by showing that $f$ is indeed smoothly isotopic to the identity, relative to the boundary and the marked points.

    To show that the map $\tilde f$ can be chosen such that the induced map $g$ is isotopic to the identity relative to the marked points and boundary, we will use the short exact sequence from Lemma~\ref{farb_lemma}:
    \begin{equation*}
        1 \rightarrow \langle T_\beta \rangle \rightarrow \mcg(\Sigma^{b+1}_{g,n-1}) \xrightarrow{\capmap} \mcg(\Sigma^{b}_{g,n}) \rightarrow 1,
    \end{equation*}
    where $\beta$ is the boundary component that we added when we removed the once-marked disc.
    
    The map $\capmap$ sends $[g]$ to the class $[\tilde f]$, which is trivial since $\tilde f$ is smoothly isotopic to $f$, and $[f]$ is trivial by assumption. Hence, $[g]$ is in the kernel of $\capmap$, meaning that there exists some $m \in \mathbb{Z}$ such that $[g] = [T_\beta^m]$.
    
    We can smoothly isotope the map $\tilde f$ to a map $\tilde f'$ that is still the identity on the neighborhood of the once marked disc, by integrating the product of a bump function with a smooth vector field that rotates a neighborhood of the once marked disc around the marked point $-m$ times. Then, its restriction to $\Sigma^{b+1}_{g,n-1}$ will give a map $g':\Sigma^{b+1}_{g,n-1} \rightarrow \Sigma^{b+1}_{g,n-1}$ whose class is $[g'] = [T^{-m}_\beta g] = [\id] \in \mcg(\Sigma^{b+1}_{g,n-1})$.
    
    If we use $\tilde f'$ and $g'$ instead of $\tilde f$ and $g$, we get that $g'$ is isotopic to the identity map relative to the boundary and marked points in $\Sigma^{b+1}_{g,n-1}$, and this finishes the proof.
\end{proof}

\printbibliography{}

\end{document}